\documentclass[10pt,intlimits]{amsart}
\usepackage{graphicx,enumitem,engrec,color}

\usepackage{eucal}

\usepackage[alphabetic,initials]{amsrefs}

\newtheorem{theorem}{Theorem}[section]

\newtheorem{lemma}[theorem]{Lemma}
\newtheorem{proposition}[theorem]{Proposition}
\theoremstyle{definition}
\newtheorem{definition}[theorem]{Definition}
\theoremstyle{remark}
\newtheorem{remark}[theorem]{Remark}
\newtheorem{claim}[theorem]{Claim}
\numberwithin{equation}{section}

\newcommand{\R}{\mathbb{R}}
\newcommand{\N}{\mathbb{N}}
\newcommand{\Z}{\mathbb{Z}}
\newcommand{\C}{\mathcal{C}}
\renewcommand{\S}{\mathfrak{S}}
\newcommand{\K}{\mathfrak{K}}
\renewcommand{\P}{\mathfrak{P}}
\newcommand{\M}{\mathcal{M}}

\newcommand{\cR}{\mathcal{R}}
\newcommand{\cS}{\mathcal{S}}

\renewcommand{\phi}{\varphi} 
\renewcommand{\epsilon}{\varepsilon} 

\def\XXint#1#2#3{{\setbox0=\hbox{$#1{#2#3}{\int}$}
\vcenter{\hbox{$#2#3$}}\kern-.5\wd0}}

\newcounter{step}
\newcounter{substep}
\newcounter{subsubstep}
\newcounter{subsubsubstep}
\numberwithin{substep}{step}
\numberwithin{subsubstep}{substep}
\numberwithin{subsubsubstep}{subsubstep}
\renewcommand{\thestep}{\arabic{step}}
\renewcommand{\thesubstep}{\thestep.\roman{substep}}
\renewcommand{\thesubsubstep}{\thesubstep.\alph{subsubstep}}
\renewcommand{\thesubsubsubstep}{\thesubsubstep.\engrec{subsubsubstep}}

\newcommand{\step}[1]{%
\bigskip\noindent%
\refstepcounter{step}%
\label{step:#1}%
(\thestep$^\circ$)}

\newcommand{\substep}[1]{%
\medskip\noindent%
\refstepcounter{substep}%
\label{step:#1}%
(\thesubstep$^\circ$)}

\newcommand{\subsubstep}[1]{%
\smallskip\noindent%
\refstepcounter{subsubstep}%
\label{step:#1}%
(\thesubsubstep$^\circ$)}

\newcommand{\subsubsubstep}[1]{%
\smallskip\noindent%
\refstepcounter{subsubsubstep}%
\label{step:#1}%
(\thesubsubsubstep$^\circ$)}

\newcommand{\stepref}[1]{(\ref{step:#1}$^\circ$)}

\renewcommand{\div}{\operatorname{div}}
\newcommand{\dist}{\operatorname{dist}}

\newcommand{\supp}{\operatorname{supp}}
\newcommand{\loc}{\mathrm{loc}}
\renewcommand{\Re}{\operatorname{Re}}
\renewcommand{\Im}{\operatorname{Im}}

\allowdisplaybreaks[1]

\begin{document}

\title[The Parabolic Signorini Problem]{Optimal Regularity and the
  Free Boundary\\in the Parabolic Signorini Problem}

\author[D.~Danielli]{Donatella Danielli} \address{Department of
  Mathematics, Purdue University, West Lafayette, IN 47907, USA}
\email{danielli@math.purdue.edu} \thanks{D.D. was supported in part by
  NSF grant DMS-1101246} \author[N.~Garofalo]{Nicola Garofalo}
\address{Dipartimento d'Ingegneria Civile e Ambientale (DICEA),
  Universit\`a di Padova, via Trieste 63, 35131 Padova, Italy}
\email{rembrandt54@gmail.com} \thanks{N.G. was supported in part by
  NSF grant DMS-1001317} \author[A.~Petrosyan]{Arshak Petrosyan}
\address{Department of Mathematics, Purdue University, West Lafayette,
  IN 47907, USA} \email{arshak@math.purdue.edu} \thanks{A.P. was
  supported in part by NSF grant DMS-1101139} \author[T.~To]{Tung To}
\address{Department of Mathematics, Purdue University, West Lafayette,
  IN 47907, USA} \email{totung@gmail.com}

\begin{abstract} We give a comprehensive treatment of the parabolic
  Signorini problem based on a generalization of Almgren's
  monotonicity of the frequency. This includes the proof of the
  optimal regularity of solutions, classification of free boundary
  points, the regularity of the regular set and the structure of the
  singular set.
\end{abstract}
\maketitle
\tableofcontents

\section{Introduction}
Given a domain $\Omega$ in $\R^n$, $n\geq 2$, with a sufficiently
regular boundary $\partial\Omega$, let $\M$ be a relatively open
subset of $\partial \Omega$ (in its relative topology), and set
$\cS=\partial\Omega\setminus\M$. We consider the solution of the
problem
\begin{align}\label{eq:signor-v-1}
  \Delta v-\partial_t v=0&\quad\text{in }\Omega_T:=\Omega\times(0,T],\\
  \label{eq:signor-v-2}
  v\geq \phi,\quad \partial_\nu v\geq 0,\quad (v-\phi)\partial_\nu
  v=0&\quad\text{on }\M_T:=\M\times(0,T],\\
  \label{eq:signor-v-3}
  v=g&\quad\text{on }\cS_T:=\cS\times(0,T],\\
  v(\cdot, 0)=\phi_0&\quad\text{on }\Omega_0:=\Omega\times\{0\},
  \label{eq:signor-v-4}
\end{align}
where $\partial_\nu$ is the outer normal derivative on
$\partial\Omega$, and $\phi:\M_T\to \R$, $\phi_0:\Omega_0\to\R$, and
$g:\cS_T\to \R$ are prescribed functions satisfying the compatibility
conditions: $\phi_0\geq \phi$ on $\M\times\{0\}$, $g\geq \phi$ on
$\partial \cS\times(0,T]$, and $g=\phi$ on $\cS\times\{0\}$, see
Fig.~\ref{fig:par-Sig-prob}. The condition \eqref{eq:signor-v-2} is
known as the \emph{Signorini boundary condition} and the problem
\eqref{eq:signor-v-1}--\eqref{eq:signor-v-4} as the \emph{(parabolic)
  Signorini problem} for the heat equation. The function $\phi$ is
called the \emph{thin obstacle}, since $v$ is restricted to stay above
$\phi$ on $\M_T$. Classical examples where
Signorini-type boundary conditions appear are the problems with
unilateral constraints in elastostatics (including the original
Signorini problem \cites{Sig,Fic}), problems with semipermeable
membranes in fluid mechanics (including the phenomenon of osmosis and
osmotic pressure in biochemistry), and the problems on the temperature
control on the boundary in thermics. We refer to the book of Duvaut
and Lions \cite{DL}, where many such applications are discussed and
the mathematical models are derived.

Another historical importance of the parabolic Signorini problem is
that it serves as one of the prototypical examples of evolutionary
\emph{variational inequalities}, going back to the foundational paper
by Lions and Stampacchia \cite{LS}, where the existence and uniqueness
of certain weak solutions were established.  In this paper, we work
with a stronger notion of solution. Thus, we say that a function $v\in
W^{1,0}_2(\Omega_T)$ solves
\eqref{eq:signor-v-1}--\eqref{eq:signor-v-4} if
\begin{align*}
  \int_{\Omega_T} \nabla v\nabla(w-v)+\partial_t v(w-v)\geq 0\quad\text{for every $w\in\K$},\\
  v\in\K,\quad \partial_t v\in L_2(\Omega_T),\quad v(\cdot,0)=\phi_0,
\end{align*}
where $\K=\{w\in W^{1,0}_2(\Omega_T)\mid w\geq \phi\text{ on } \M_T,\
w=g\text{ on }\cS_T\}$. The reader should see
Section~\ref{sec:parab-funct-class} for the definitions of the
relevant parabolic functional classes. The existence and uniqueness of
such $v$, under some natural assumptions on $\phi$, $\phi_0$, and $g$
can be found in \cites{Bre,DL,AU0,AU}. See also
Section~\ref{sec:existence-regularity} for more details.


Two major questions arise in the study of the problem
\eqref{eq:signor-v-1}--\eqref{eq:signor-v-4}:
\begin{itemize}
\item the regularity properties of $v$;
\item the structure and regularity of the \emph{free boundary}
  \begin{align*}
    \Gamma(v)&=\partial_{\M_T}\{(x,t)\in \M_T\mid v(x,t)>\phi(x,t)\},
  \end{align*}
  where $\partial_{\M_T}$ indicates the boundary in the relative
  topology of $\M_T$.
\end{itemize}

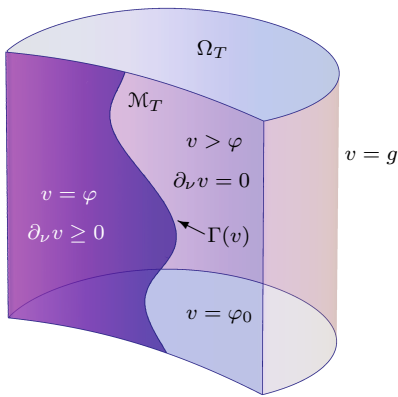
\begin{figure}[t]
  \begin{picture}(150,150)
    \put(0,0){\includegraphics[height=150pt]{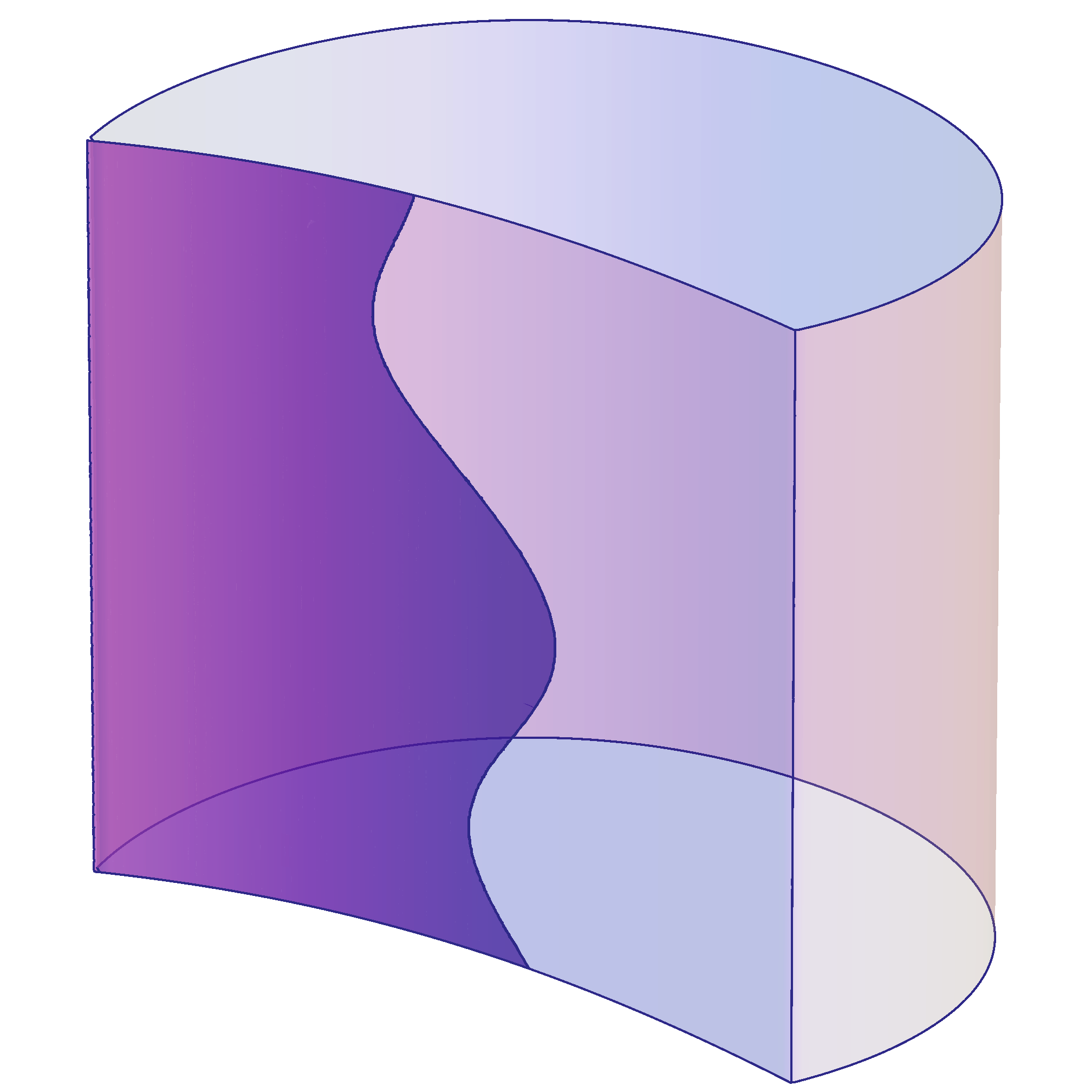}}
    \put(84,130){\footnotesize $\Omega_T$} \put(58,110){\footnotesize
      {$\M_T$}} \put(80,95){\footnotesize {$v>\phi$}}
    \put(75,80){\footnotesize $\partial_\nu v=0$}
    \put(25,75){\footnotesize \color{white}{$v=\phi$}}
    \put(20,60){\footnotesize \color{white}{$\partial_\nu v\geq 0$}}
    \put(80,30){\footnotesize $v=\phi_0$} \put(140,90){\footnotesize
      $v=g$} \put(88,59){\footnotesize $\Gamma(v)$}
    \put(87,62){\vector(-2,1){10}}
  \end{picture}
  \caption{The parabolic Signorini problem}
  \label{fig:par-Sig-prob}
\end{figure}

Concerning the regularity of $v$, it has long been known that the
spatial derivatives $\partial_{x_i}v$, $i=1,\ldots,n$, are
$\alpha$-H\"older continuous on compact subsets of $\Omega_T\cup\M_T$,
for some unspecified $\alpha\in (0,1)$. In the time-independent case,
such regularity was first proved by Richardson~\cite{Ric} in dimension
$n=2$, and by Caffarelli \cite{Ca-sig} for $n\ge 3$. In the parabolic
case, this was first proved by Athanasopoulos~\cite{Ath1}, and
subsequently by Uraltseva~\cite{Ur} (see also \cite{AU0}), under
certain regularity assumptions on the boundary data, which were
further relaxed by Arkhipova and Uraltseva \cite{AU}.

We note that the H\"older continuity of $\partial_{x_i}v$ is the best
regularity one should expect for the solution of
\eqref{eq:signor-v-1}--\eqref{eq:signor-v-4}. This can be seen from
the example
$$
v=\Re(x_1+ix_n)^{3/2},\quad x_n\geq 0,
$$
which is a harmonic function in $\R^n_+$, and satisfies the Signorini
boundary conditions on $\M=\R^{n-1}\times\{0\}$, with thin obstacle
$\phi\equiv 0$. This example also suggests that the optimal H\"older
exponent for $\partial_{x_i}v$ should be $1/2$, at least when $\M$ is
flat (contained in a hyperplane). Indeed, in the time-independent
case, such optimal $C^{1,1/2}$ regularity was proved in dimension
$n=2$ in the cited paper by Richardson~\cite{Ric}. The case of
arbitrary space dimension (still time-independent), however, had to
wait for the breakthrough work of Athanasopoulos and
Caffarelli~\cite{AC}. Very recently, the proof of the optimal
regularity for the original Signorini problem in elastostatics was
announced by Andersson \cite{And}.  

One of the main objectives of this paper is to establish, in the
parabolic Signorini problem, and for a flat thin manifold $\M$, that
$\nabla v\in H^{1/2,1/4}_\loc(\Omega_T\cup \M_T)$, or more precisely
that $v\in H^{3/2,3/4}_\loc(\Omega_T\cup \M_T)$, see
Theorem~\ref{thm:opt-reg} below.  Our approach is inspired by the
works of Athanasopoulos, Caffarelli, and Salsa \cite{ACS} and
Caffarelli, Salsa, and Silvestre \cite{CSS} on the time-independent
problem.  In such papers a generalization of the celebrated Almgren's
frequency formula established in \cite{Alm} was used, not only to give
an alternative proof of the optimal $C^{1,1/2}$ regularity of
solutions, but also to study the so-called \emph{regular set} $\cR(v)$
of the free boundary $\Gamma(v)$. This approach was subsequently
refined in \cite{GP} by the second and third named authors with the
objective of classifying the free boundary points according to their
separation rate from the thin obstacle $\phi$. In \cite{GP} the
authors also introduced generalizations of Weiss's and Monneau's
monotonicity formulas, originally developed in \cites{Wei1} and
\cite{Mon}, respectively, for the classical obstacle problem. Such
generalized Weiss's and Monneau's monotonicity formulas allowed to
prove a structural theorem on the so-called \emph{singular set}
$\Sigma(v)$ of the free boundary, see \cite{GP}. For an exposition of
these results in the case when the thin obstacle $\phi\equiv 0$ we
also refer to the book by the third named author, Shahgholian, and
Uraltseva \cite{PSU}*{Chapter 9}.

In closing we mention that, as far as we are aware of, the only result
presently available concerning the free boundary in the parabolic
setting is that of Athanasopoulos \cite{Ath2}, under assumptions on
the boundary data that guarantee boundedness and nonnegativity of
$\partial_t v$. In that paper it is shown that the free boundary is
locally given as a graph
$$
t=h(x_1,\ldots,x_{n-1}),
$$
for a Lipschitz continuous function $h$.

\subsection{Overview of the main results} In this paper we extend all
the above mentioned results from the elliptic to the parabolic
case. We focus on the situation when the principal part of the
diffusion operator is the Laplacian and that the thin manifold $\M$ is
flat and contained in $\R^{n-1}\times\{0\}$.

One of our central results is a generalization of Almgren's frequency
formula \cite{Alm}, see Theorem~\ref{thm:thin-monotonicity} below. As
it is well known, the parabolic counterpart of Almgren's formula was
established by Poon \cite{Poo}, for functions which are caloric in an
infinite strip $S_\rho=\R^n\times(-\rho^2,0]$.  In
Section~\ref{sec:gener-freq-funct}, we establish a truncated version
of that formula for the solutions of the parabolic Signorini problem,
similar to the ones in \cite{CSS} and \cite{GP}. The time dependent
case presents, however, substantial novel challenges with respect to
the elliptic setting. These are mainly due to the lack of regularity
of the solution in the $t$-variable, a fact which makes the
justification of differentiation formulas and the control of error
terms quite difficult. To overcome these obstructions, we have
introduced (Steklov-type) averaged versions of the quantities involved
in our main monotonicity formulas. This basic idea has enabled us to
successfully control the error terms.

Similarly to what was done in \cite{GP}, we then undertake a
systematic classification of the free boundary points based on the
limit at the point in question of the generalized frequency
function. When the thin obstacle $\phi$ is in the class
$H^{\ell,\ell/2}$, $\ell\geq 2$, this classification translates into
assigning to each free boundary point in $\Gamma(v)$ (or more
generally to every point on the \emph{extended free boundary}
$\Gamma_*(v)$, see Section~\ref{sec:classes-solutions}) a certain
\emph{frequency} $\kappa\leq \ell$, see
Sections~\ref{sec:exist-homog-blow} and \ref{sec:class-free-bound}.
At the points for which $\kappa<\ell$, the separation rate of the
solution $v$ from the thin obstacle can be ``detected'', in a sense
that it will exceed the truncation term in the generalized frequency
formula. At those points we are then able to consider the so-called
blowups, which will be parabolically $\kappa$-homogeneous solutions of
the Signorini problem, see Section~\ref{sec:exist-homog-blow}.

Next, we show that, similarly to what happens in the elliptic case,
the smallest possible value of the frequency at a free boundary point
is $\kappa =3/2$, see Section~\ref{sec:homog-glob-solut}.  We
emphasize that our proof of this fact does not rely on the
semiconvexity estimates, as in the elliptic case (see \cite{AC} or
\cite{CSS}). Rather, we use the monotonicity formula of
Caffarelli~\cite{Ca1} to reduce the problem to the spatial dimension
$n=2$, and then study the eigenvalues of the Ornstein-Uhlenbeck
operator in domains with slits (see Fig.~\ref{fig:slit-dom} in
Section~\ref{sec:homog-glob-solut}). The elliptic version of this
argument has appeared earlier in the book \cite{PSU}*{Chapter~9}. The
bound $\kappa\geq 3/2$ ultimately implies the optimal
$H^{3/2,3/4}_\loc$ regularity of solutions, see
Section~\ref{sec:optim-regul-solut}.

We next turn to studying the regularity properties of the free
boundary. We start with the so-called \emph{regular set} $\cR(v)$,
which corresponds to free boundary points of minimal frequency
$\kappa=3/2$. Similarly to the elliptic case, studied in
\cites{ACS,CSS}, the Lipschitz regularity of $\cR(v)$ with respect to
the space variables follows by showing that there is a cone of spatial
directions in which $v-\phi$ is monotone. The $1/2$-H\"older
regularity in $t$ is then a consequence of the fact that the blowups
at regular points are $t$-independent, see
Theorem~\ref{thm:lip-reg-reg-set}. Thus, after possibly rotating the
coordinate axes in $\R^{n-1}$, we obtain that $\cR(v)$ is given
locally as a graph
$$
x_{n-1}=g(x_1,\ldots,x_{n-2},t),
$$
where $g$ is parabolically Lipschitz (or
$\operatorname{Lip}(1,\frac12)$ in alternative terminology). To prove
the H\"older $H^{\alpha,\alpha/2}$ regularity of $\partial_{x_i}g$,
$i=1,\ldots, n-2$, we then use an idea that goes back to the paper of
Athanasopoulos and Caffarelli \cite{AC0} based on an application of
the boundary Harnack principle (forward and backward) for so-called
domains with thin Lipschitz complement, i.e, domains of the type
$$
Q_1\setminus\{(x',t)\in Q_1'\mid x_{n-1}\leq g(x'',t)\},
$$ 
see Lemma~\ref{lem:BHP}. This result was recently established in the
work of the third named author and Shi \cite{Shi}. We emphasize that,
unlike the elliptic case, the boundary Harnack principle for such
domains cannot be reduced to the other known results in the parabolic
setting (see e.g.\ \cites{Kem,FGS} for parabolically Lipschitz
domains, or \cite{HLN} for parabolically NTA domains with Reifenberg
flat boundary).

Another type of free boundary points that we study are the so-called
\emph{singular points}, where the coincidence set $\{v=\phi\}$ has
zero $\mathcal{H}^n$-density in the thin manifold with respect to the
thin parabolic cylinders. This corresponds to free boundary points
with frequency $\kappa=2m$, $m\in\N$. The blowups at those points are
parabolically $\kappa$-homogeneous polynomials, see
Section~\ref{sec:free-bound-sing}.

Following the approach in \cite{GP}, in
Section~\ref{sec:weiss-monneau-type} we establish appropriate
parabolic versions of monotonicity formulas of Weiss and Monneau
type. Using such formulas we are able to prove the uniqueness of the
blowups at singular free boundary points $(x_0,t_0)$, and consequently
obtain a Taylor expansion of the type
$$
v(x,t)-\phi(x',t)=q_\kappa(x-x_0,t-t_0)+o(\|(x-x_0,t-t_0)\|^\kappa),\quad
t\leq t_0,
$$
where $q_\kappa$ is a polynomial of parabolic degree $\kappa$ that
depends continuously on the singular point $(x_0,t_0)$ with frequency
$\kappa$. We note that such expansion holds only for $t\leq t_0$ and
may fail for $t>t_0$ (see
Remark~\ref{rem:taylor-cntrex}). Nevertheless, we show that this
expansion essentially holds when restricted to singular points
$(x,t)$, even for $t\geq t_0$. This is necessary in order to verify
the compatibility condition in a parabolic version of the Whitney's
extension theorem (given in
Appendix~\ref{sec:parab-whitn-extens}). Using the latter we are then
able to prove a structural theorem for the singular set. For the
elliptic counterpart of this result see \cite{GP}. It should be
mentioned at this moment that one difference between the parabolic
case treated in the present paper and its elliptic counterpart is the
presence of new types of singular points, which we call
\emph{time-like}. At such points the blowup may become independent of
the space variables $x'$. We show that such singular points are
contained in a countable union of graphs of the type
$$
t=g(x_1,\ldots,x_{n-1}),
$$
where $g$ is a $C^1$ function. The other singular points, which we
call \emph{space-like}, are contained in countable union of
$d$-dimensional $C^{1,0}$ manifolds ($d<n-1$). After a possible
rotation of coordinates in $\R^{n-1}$, such manifolds are locally
representable as graphs of the type
$$
(x_{d+1},\ldots,x_{n-1})=g(x_1,\ldots,x_d,t),
$$
with $g$ and $\partial_{x_i}g$, $i=1,\ldots,d$, continuous.

\subsection{Related problems}
The time-independent version of the Signorini problem is closely
related to the obstacle problem for the half-Laplacian in $\R^{n-1}$
$$
u\geq \phi,\quad (-\Delta_{x'})^{1/2}u\geq 0,\quad
(u-\phi)(-\Delta_{x'})^{1/2}u= 0\quad\text{in }\R^{n-1}.
$$
More precisely, if we consider a harmonic extension of $u$ from
$\R^{n-1}=\R^{n-1}\times\{0\}$ to $\R^n_+$ (by means of a convolution
with the Poisson kernel), we will have that
$$
(-\Delta_{x'})^{1/2}u=- c_n\partial_{x_n} u\quad\text{on
}\R^{n-1}\times\{0\}.
$$
Thus, the extension $u$ will solve the Signorini problem in
$\R^n_+$. More generally, in the above problem instead of the
half-Laplacian one can consider an arbitrary fractional power of the
Laplacian $(-\Delta_{x'})^{s}$, $0<s<1$, see e.g. the thesis of
Silvestre \cite{Sil}. Problems of this kind appear for instance in
mathematical finance, in the valuation of American options, when the
asset prices are modelled by jump processes. The time-independent
problem corresponds to the so-called perpetual options, with infinite
maturity time. In such framework, with the aid of an extension theorem
of Caffarelli and Silvestre \cite{CS}, many of the results known for
$s=1/2$ can be proved also for all powers $0<s<1$, see \cite{CSS}.

The evolution version of the problem above is driven by the fractional
diffusion and can be written as
\begin{align*}
  u(x',t)-\phi(x')&\geq 0,\\
  \left((-\Delta_{x'})^{s}+\partial_t\right) u&\geq
  0,\\
  \left(u(x',t)-\phi(x')\right)\left((-\Delta_{x'})^{s}+\partial_t\right)
  u&=0
\end{align*}
in $\R^{n-1}\times(0,T)$ with the initial condition
$$
u(x',0)=\phi(x').
$$
This problem has been recently studied by Caffarelli and Figalli
\cite{CF}. We emphasize that, although their time-independent versions
are locally equivalent, the problem studied in \cite{CF} is very
different from the one considered in the present paper.

In relation to temperature control problems on the boundary, described
in \cite{DL}, we would like to mention the two recent papers by
Athanasopoulos and Caffarelli \cite{AC-two-phase}, and by Allen and
Shi \cite{AS}. Both papers deal with two-phase problems that can be
viewed as generalizations of the one-phase problem (with $\phi=0$)
considered in this paper. The paper \cite{AS} establishes the
phenomenon of separation of phases, thereby locally reducing the study
of the two-phase problem to that of one-phase. A similar phenomenon
was shown earlier in the elliptic case by Allen, Lindgren, and the
third named author \cite{ALP}.

\subsection{Structure of the paper}

In what follows we provide a brief description of the structure of the
paper.

\begin{itemize}
\item In Section~\ref{sec:notat-prel} we introduce the notations used
  throughout the paper, and describe the relevant parabolic functional
  classes.
\item In Section~\ref{sec:existence-regularity} we overview some of
  the known basic regularity properties of the solution $v$ of the
  parabolic Signorini problem. The main ones are: $v\in
  W^{2,1}_{2,\loc}\cap L^\infty_\loc$; and, $\nabla v\in
  H^{\alpha,\alpha/2}_\loc$ for some $\alpha>0$. Such results will be
  extensively used in our paper.

\item In Section~\ref{sec:classes-solutions} we introduce the classes
  of solutions $\S_\phi(Q_1^+)$ of the parabolic Signorini problem
  with a thin obstacle $\phi$, and show how to effectively
  ``subtract'' the obstacle by maximally using its regularity. In this
  process we convert the problem to one with zero thin obstacle, but
  with a non-homogeneous right hand side $f$ in the equation. In order
  to apply our main monotonicity formulas we also need to to extend
  the resulting functions from $Q_1^+$ to the entire strip $S_1^+$. We
  achieve this by multiplication by a a cutoff function, and denote
  the resulting class of functions by $\S^f(S_1^+)$.

\item Section~\ref{sec:estimates-w2-1_2s_1+} contains generalizations
  of $W^{2,1}_2$ estimates to the weighted spaced with Gaussian
  measure. These estimates will be instrumental in the proof of the
  generalized frequency formula in Section~\ref{sec:gener-freq-funct}
  and in the study of the blowups in
  Section~\ref{sec:exist-homog-blow}. In order not to distract the
  reader from the main content, we have deferred the proof of these
  estimates to the Appendix~\ref{sec:est-gauss-proofs}.
\item Section~\ref{sec:gener-freq-funct} is the most technical part of
  the paper. There, we generalize Almgren's (Poon's) frequency formula
  to solutions of the parabolic Signorini problem.

\item In Section~\ref{sec:exist-homog-blow} we prove the existence and
  homogeneity of the blowups at free boundary points where the
  separation rate of the solution from the thin obstacle dominates the
  error (truncation) terms in the generalized frequency formula.
\item In Section~\ref{sec:homog-glob-solut} we prove that the minimal
  homogeneity of homogeneous solution of the parabolic Signorini
  problem is $3/2$.

\item Section~\ref{sec:optim-regul-solut} contains the proof of the
  optimal $H^{3/2,3/4}_\loc$ regularity of the solutions of the
  parabolic Signorini problem.

\item In the remaining part of the paper we study the free
  boundary. We start in Section~\ref{sec:class-free-bound} by
  classifying the free boundary points according to the homogeneity of
  the blowups at the point in question. We also use the assumed
  regularity of the thin obstacle in the most optimal way.

\item In Section~\ref{sec:free-bound-regul}, we study the so-called
  regular set $\cR(v)$ and show that it can be locally represented as
  a graph with $H^{\alpha,\alpha/2}$ regular gradient.

\item In Section~\ref{sec:free-bound-sing} we give a characterization
  of the so-called singular points.
\item Section~\ref{sec:weiss-monneau-type} contains some new Weiss and
  Monneau type monotonicity formulas for the parabolic problem. These
  results generalize the ones in \cite{GP} for the elliptic case.

\item In Section~\ref{sec:struct-sing-set} we prove the uniqueness of
  blowups at singular points and the continuous dependence of blowups
  on compact subsets of the singular set. We then invoke a parabolic
  version of Whitney's extension theorem (given in
  Appendix~\ref{sec:parab-whitn-extens}) to prove a structural theorem
  on the singular set.
\end{itemize}

\section{Notation and preliminaries}\label{sec:notat-prel}
\subsection{Notation}\label{sec:not}
To proceed, we fix the notations that we are going to use throughout
the paper.
\begin{alignat*}{3}
  &\N &&=\{1,2,\ldots\}&\qquad&\text{(natural numbers)}\\
  &\Z&&=\{0,\pm1,\pm2,\ldots\}&\qquad&\text{(integers)}\\
  &\Z_+&&=\N\cup\{0\}\qquad&&\text{(nonnegative integers)}\\
  &\R &&=(-\infty,\infty)&&\text{(real numbers)}\\
  & s^\pm&&=\max\{\pm s,0\},\quad s\in\R&&\text{(positive/negative
    part
    of $s$)}\\
  &\R^n&&=\{x=(x_1,x_2,\ldots, x_n)\mid x_i\in\R\}&\qquad& \text{(Euclidean space)}\\
  &\R^n_+&&=\{x\in \R^n \mid x_n > 0\}&& \text{(positive half-space)}\\
  &\R^n_-&&=\{x\in \R^n\mid x_n < 0\}&& \text{(negative half-space)}\\
  &\R^{n-1}&&\text{identified with $\R^{n-1}\times\{0\}\subset\R^n$}
  &&\text{(thin space)}\\
  & x' &&=(x_1, x_2,\ldots,x_{n-1})\quad\text{for }x\in\R^n\\
  &&& \text{we also identify $x'$ with $(x',0)$}\\
  &x''&&=(x_1,x_2,\ldots,x_{n-2})\\
  &|x|&&=\Big(\sum_{i=1}^n x_i^2\Big)^{1/2},\quad x\in\R^n&&\text{(Euclidean norm)}\\
  &\|(x,t)\|&&=(|x|^2+|t|)^{1/2},\quad x\in\R^n, t\in\R&&\text{(parabolic norm)}\\
  &x^\alpha &&=x_1^{\alpha_1}\dots x_n^{\alpha_n},\quad x\in\R^n, \alpha\in\Z_+^n\\
  &\overline{E}, E^\circ, \partial E&&\text{closure, interior,
    boundary of the set $E$}\\
  &\partial_{X} E &&\text{boundary in the relative topology of $X$}\\
  &\partial_p E &&\text{parabolic boundary of $E$}\\
  &E^c&&\text{complement of the set $E$}\\
  &\mathcal{H}^s(E)&&\text{$s$-dimensional Hausdorff measure}\\
  &&&\text{of a Borel set $E$}\\
  \intertext{For $x_0\in\R^{n}$, $t_0\in\R$, and $r>0$ we let}
  &B_r(x_0) &&=\{x\in\R^n\mid |x-x_0| < r\}&&\text{(Euclidean ball)}\\
  &B^\pm_r(x_0) &&= B_r(x_0)\cap \R^n_\pm,\quad x_0\in\R^{n-1} &&\text{(Euclidean half-ball)}\\
  &B'_r(x_0)&& =B_r(x_0)\cap \R^{n-1},\quad x_0\in\R^{n-1}&& \text{(thin ball)}\\
  &B''_r(x_0)&& =B_r'(x_0)\cap \R^{n-2},\quad x_0\in\R^{n-2}&&\\
  &\C_\eta'&&=\{x'\in\R^{n-1}\mid x_{n-1}\geq \eta |x''|\},\quad
  \eta>0&&\text{(thin cone)}\\
  &Q_r(x_0,t_0) &&= B_r(x_0)\times(t_0-r^2,t_0] &&\text{(parabolic cylinder)}\\
  &Q_r^\pm(x_0,t_0) &&= B_r^\pm(x_0)\times(t_0-r^2,t_0]&&\text{(parabolic half-cylinders)}\\
  &\tilde Q_r(x_0,t_0) &&=
  B_r(x_0)\times(t_0-r^2,t_0+r^2)&&\text{(full parabolic cylinder)}\\
  &Q_r'(x_0,t_0) &&= B_r'(x_0)\times(t_0-r^2,t_0]&&\text{(thin
    parabolic
    cylinder)}\\*
  &Q_r''(x_0,t_0) &&= B_r''(x_0)\times(t_0-r^2,t_0]\\
  &S_r &&=\R^n\times(-r^2,0] &&\text{(parabolic strip)}\\
  &S_r^\pm &&=\R^n_\pm\times(-r^2,0] &&\text{(parabolic half-strip)}\\
  &S'_r &&=\R^{n-1}\times(-r^2,0] &&\text{(thin parabolic strip)}
  \intertext{When $x_0=0$ and $t_0=0$, we routinely omit indicating
    the centers $x_0$ and $(x_0,t_0)$ in the above notations.}
  &\partial_e u, u_e&&\text{partial derivative in the direction $e$}
  \\
  &\partial_{x_i}u, u_{x_i} &&=\partial_{e_i}u,\quad
  \text{for standard coordinate}\\*
  &&&\text{vectors }e_i,\ i=1,\ldots, n\\
  &\partial_t u, u_t&&\text{partial derivative in $t$ variable}\\
  &u_{x_{i_1}\cdots x_{i_k}}&&=\partial_{x_{i_1}\cdots
    x_{i_k}}u=\partial_{x_{i_1}}\cdots\partial_{x_{i_k}}u\\
  &\partial_x^\alpha
  u&&=\partial_{x_1}^{\alpha_1}\cdots\partial_{x_n}^{\alpha_n}u,\\*
  &&&\text{for $\alpha=(\alpha_1,\ldots,\alpha_n)$, $\alpha_i\in\Z_+$}\\
  &\nabla u, \nabla_{x}u&&=(\partial_{x_1}u,\ldots,\partial_{x_{n}}u)&&\text{(gradient)}\\
  &\nabla'u,
  \nabla_{x'}u&&=(\partial_{x_1}u,\ldots,\partial_{x_{n-1}}u)&&\text{(tangential
    or thin gradient)}\\
  &\nabla''u,\nabla_{x''}u&&=(\partial_{x_1}u,\ldots,\partial_{x_{n-2}}u)\\
  &D^ku,D^k_x u&&=(\partial_x^\alpha u)_{|\alpha|=k},\quad k\in\Z_+,\\*
  &&& \text{where }
  |\alpha|=\alpha_1+\cdots+\alpha_n\\
  &\Delta u,\Delta_xu&&=\sum_{i=1}^n \partial_{x_ix_i}u&&\text{(Laplacian)}\\
  &\Delta'u,\Delta_{x'}u&&=\sum_{i=1}^{n-1} \partial_{x_ix_i}u&&\text{(tangential
    or thin Laplacian)}
\end{alignat*}
We denote by $G$ the \emph{backward heat kernel} on $\R^n\times\R$
\[
G(x,t) =\begin{cases}
  (-4\pi t)^{-\frac{n}{2}} e^{\frac{x^2}{4t}}, & t<0,\\
  0,& t\geq 0.
\end{cases}
\]
We will often use the following properties of $G$:
\begin{equation}\label{eq:G}
  \Delta G+\partial_t G = 0,\quad G(\lambda x,\lambda^2 t)=
  \lambda^{-n} G(x, t),\quad\nabla G =\frac{x}{2t}\, G.
\end{equation}
Besides, to simplify our calculations, we define the differential
operator:
\begin{equation}\label{eq:Z}
  Zu = x\nabla u+2t\partial_tu ,
\end{equation}
which is the generator of the parabolic scaling in the sense that
\begin{equation}\label{eq:Z-gen}
  Zu(x,t)=\frac{d}{d\lambda}\Big|_{\lambda=1} u(\lambda x,\lambda^2 t).
\end{equation}
Using \eqref{eq:G}, the operator $Z$ can also be defined through the
following identity:
\begin{equation}\label{eq:ZG}
  Zu =2t\Big(\nabla u\frac{\nabla G}{G}+\partial_t u\Big).
\end{equation}

\subsection{Parabolic functional
  classes} \label{sec:parab-funct-class}

For the parabolic functional classes, we have opted to use notations
similar to those in the classical book of Ladyzhenskaya, Solonnikov,
and Uraltseva \cite{LSU}.

Let $\Omega\subset\R^n$ be an open subset in $\R^n$ and
$\Omega_{T}=\Omega\times(0,T]$ for $T>0$. The class
$C(\Omega_T)=C^{0,0}(\Omega_T)$ is the class of functions continuous
in $\Omega_T$ with respect to parabolic (or Euclidean)
distance. Further, given for $m\in\Z_+$ we say $u\in
C^{2m,m}(\Omega_T)$ if for $|\alpha|+2j\leq 2m$ $
\partial_x^\alpha\partial_t^j u\in C^{0,0}(\Omega_T)$, and define the
norm
$$
\|u\|_{C^{2m,m}(\Omega_T)}=\sum_{|\alpha|+2j\leq 2m} \sup_{(x,t)\in
  \Omega_T}|\partial_x^\alpha\partial_t^j u(x,y)|.
$$
The parabolic H\"older classes $H^{\ell,\ell/2}(\Omega_T)$, for
$\ell=m+\gamma$, $m\in\Z_+$, $0<\gamma\leq 1$ are defined as follows.
First, we let
\begin{align*}
  \langle u \rangle^{(0)}_{\Omega_T}&=|u|^{(0)}_{\Omega_T}=\sup_{(x,t)\in \Omega_T} |u(x,t)|,\\
  \langle u \rangle^{(m)}_{\Omega_T}&=\sum_{|\alpha|+2j=m}
  |\partial^\alpha_x\partial_t^j u|^{(0)}_{\Omega_T},\\
  \langle
  u\rangle^{(\beta)}_{x,\Omega_T}&=\sup_{\substack{(x,t),(y,t)\in
      \Omega_T\\0<|x-y|\leq \delta_0}}\frac{|u(x,t)-u(y,t)|}{|x-y|^\beta},\quad 0<\beta\leq1,\\
  \langle
  u\rangle^{(\beta)}_{t,\Omega_T}&=\sup_{\substack{(x,t),(x,s)\in
      \Omega_T\\0<|t-s|<\delta_0^2}}\frac{|u(x,t)-u(x,s)|}{|t-s|^\beta},\quad 0<\beta\leq1,\\
  \langle u
  \rangle^{(\ell)}_{x,\Omega_T}&=\sum_{|\alpha|+2j=m}\langle \partial^\alpha_x\partial^j_t
  u\rangle^{(\gamma)}_{x,\Omega_T},\\
  \langle u \rangle^{(\ell/2)}_{t,\Omega_T}&=\sum_{m-1\leq
    |\alpha|+2j\leq m} \langle\partial_{x}^\alpha\partial_t^j
  u\rangle^{((\ell-|\alpha|-2j)/2)}_{t,\Omega_T},\\
  \langle u \rangle^{(\ell)}_{\Omega_T}&=\langle u
  \rangle^{(\ell)}_{x,\Omega_T}+\langle u
  \rangle^{(\ell/2)}_{t,\Omega_T}.
\end{align*}
Then, we define $H^{\ell,\ell/2}(\Omega_T)$ as the space of functions
$u$ for which the following norm is finite:
$$
\|u\|_{H^{\ell,\ell/2}(\Omega_T)}=\sum_{k=0}^m\langle u
\rangle^{(k)}_{\Omega_T}+\langle u \rangle^{(\ell)}_{\Omega_T}.
$$
The parabolic Lebesgue space $L_q(\Omega_T)$ indicates the Banach
space of those measurable functions on $\Omega_T$ for which the norm
$$
\|u\|_{L_q(\Omega_T)}=\Big(\int_{\Omega_T} |u(x,t)|^q dx dt\Big)^{1/q}
$$
is finite.  The parabolic Sobolev spaces $W^{2m,m}_q(\Omega_T)$, $m\in
\Z_+$, denote the spaces of those functions in $L_q(\Omega_T)$, whose
distributional derivative $\partial_x^\alpha\partial_t^j u$ belongs to
$\in L_q(\Omega_T)$, for $|\alpha|+2j\leq 2m$. Endowed with the norm
$$
\|u\|_{W^{2m,m}_q(\Omega_T)}=\sum_{|\alpha|+2j\leq 2m}
\| \partial^\alpha_x\partial^j_t u\|_{L^q(\Omega_T)},
$$ 
$W^{2m,m}_q(\Omega_T)$ becomes a Banach space.

We also denote by $W^{1,0}_q(\Omega_T)$, $W^{1,1}_q(\Omega_T)$ the
Banach subspaces of $L^q(\Omega_T)$ generated by the norms
\begin{align*}
  \|u\|_{W^{1,0}_q(\Omega_T)}&=\|u\|_{L_q(\Omega_T)}+\|\nabla u\|_{L_q(\Omega_T)},\\
  \|u\|_{W^{1,1}_q(\Omega_T)}&=\|u\|_{L_q(\Omega_T)}+\|\nabla
  u\|_{L_q(\Omega_T)}+\|\partial_t u \|_{L_q(\Omega_T)}.
\end{align*}
Let $E\subset S_R$ for some $R>0$. The weighted Lebesgue space $L_p(E,
G)$, $p>1$, with Gaussian weight $G(x,t)$, will appear naturally in
our proofs. The norm in this space is defined by
$$
\|u\|_{L_p(E,G)}^p=\int_{E} |u(x,t)|^pG(x,t)dxdt.
$$
When $E$ is a relatively open subset of $S_R$, one may also define the
respective weighted Sobolev spaces. We will also consider weighted
spatial Lebesgue and Sobolev spaces $L_p(\Omega,G(\cdot,s))$, and
$W^m_p(\Omega,G(\cdot,s))$ with Gaussian weights $G(\cdot,s)$ on
$\R^n$ for some fixed $s<0$. For instance, the norm in the space
$L_p(\Omega,G(\cdot,s))$ is given by
$$
\|u\|_{L_p(\Omega,G(\cdot,s))}^p=\int_{\Omega} |u(x)|^pG(x,s)dx.
$$

\section{Known existence and regularity results}\label{sec:existence-regularity}
In this section we recall some known results about the existence and
the regularity of the solution of the parabolic Signorini problem that
we are going to take as the starting point of our analysis. For
detailed proofs we refer the reader to the works of Arkhipova and
Uraltseva \cites{AU0,AU}. For simplicity we state the relevant results
only in the case of the unit parabolic half-cylinder $Q_1^+$.

Suppose we are given functions $f\in L_\infty(Q_1^+)$, $\phi\in
W^{2,1}_\infty(Q'_1)$, $g\in W^{2,1}_\infty((\partial B_1)^+\times
(-1,0])$, and $\phi_0\in W^2_\infty(B_1^+)$ obeying the compatibility
conditions
\begin{alignat*}{2}
  \phi_0&=g(\cdot,-1)&&\text{a.e.\ on }(\partial B_1)^+,\\
  \phi_0&\geq\phi(\cdot,-1)&\quad&\text{a.e.\ on }B'_1,\\
  g&\geq\phi&&\text{a.e.\ on }\partial B_1'\times(-1,0].
\end{alignat*}
Given $\phi$ and $g$ as above, we introduce the following closed
subset of $W^{1,0}_2(Q_1^+)$
$$
\K=\{v\in W^{1,0}_2(Q_1^+)\mid v\geq \phi\text{ a.e.\ on } Q'_1,\
v=g\text{ a.e.\ on }(\partial B_1)^+\times(-1,0]\}.
$$
We say that $u\in W^{1,0}_2(Q_1^+)$ satisfies
\begin{align}
  \label{eq:HEu-Q1}
  \Delta u-\partial_t u = f(x,t)&\quad\text{in } Q_{1}^+,\\
  \label{eq:u-signor-Q1}
  u\geq \phi,\quad -\partial_{x_n} u\geq 0,\quad
  (u-\phi) \partial_{x_n}u = 0&\quad\text{on } Q_1',\\
  \label{eq:u-Dir-bdry-Q1}
  u=g&\quad\text{on }(\partial B_1)^+\times(-1,0],\\
  \label{eq:u-init-Q1}
  u(\cdot,-1)=\phi_0&\quad\text{on } B_1^+,
\end{align}
if $u$ solves the variational inequality
\begin{gather*}
  \int_{Q_1^+} [\nabla u\nabla(v-u) +\partial_t u (v-u)+f(v-u)]\geq
  0\quad\text{for any }v\in\K,\\
  u\in\K, \quad \partial_t u\in L_2(Q_1^+),\\
  u(\cdot, -1)=\phi_0 \quad\text{on } B_1^+.
\end{gather*}
Under the assumptions above there exists a unique solution to the
problem \eqref{eq:HEu-Q1}--\eqref{eq:u-init-Q1}. Moreover, the
solution will have H\"older continuous spatial gradient: $\nabla u\in
H^{\alpha,\alpha/2}(Q_r^+\cup Q_r')$ for any $0<r<1$ with the H\"older
exponent $\alpha>0$ depending only on the dimension, see
\cites{AU0,AU}.  Below, we sketch the details in the case $\phi=0$ and
$g=0$, that would be most relevant in our case.

For any $\epsilon>0$ we denote by $f^\epsilon$ a mollifications of $f$
and consider the solution $u^\epsilon$ to the approximating problem
\begin{align*}
  \Delta u^\epsilon-\partial_t
  u^\epsilon=f^\epsilon(x,t)&\quad\text{in
  }Q_1^+,\\
  \partial_{x_n} u^\epsilon=\beta_\epsilon(u^\epsilon)&\quad\text{on
  }Q_1',\\
  u^\epsilon=0& \quad\text{on } (\partial B_1)^+\times(-1,0],\\
  u^\epsilon(\cdot,-1)=\phi_0&\quad\text{on } B_1^+,
\end{align*}
where the penalty function $\beta_\epsilon\in C^\infty(\R)$ is such
that
\begin{multline}\label{eq:betaeps}
  \beta_\epsilon(s)=0\quad\text{for }s\geq
  0,\quad\beta_\epsilon(s)=\epsilon+s/\epsilon\quad\text{for }s\leq
  -2\epsilon^2,\\
  \text{and}\quad \beta'_\epsilon(s)\geq 0\quad\text{for all }s\in\R.
\end{multline}
The solutions $u^\epsilon$ to the penalization problems are shown to
be smooth in $Q_1^+$ up to $Q_1'$, and it can be proved that they are
uniformly bounded $W^{1,1}_2(Q_1^+)$. To this end, for any $\eta\in
W^{1,0}_2(Q_1^+)$ vanishing a.e.\ on $(\partial B_1)^+\times(-1,0]$
and $(t_1,t_2]\subset (-1,0]$, one writes the integral identity
\begin{equation}\label{eq:penalty}
  \int_{B_1^+\times(t_1,t_2]}(\nabla u^\epsilon\nabla
  \eta+u^\epsilon_t\eta+f^\epsilon\eta)dx
  dt+\int_{B_1'\times(t_1,t_2]}\beta_\epsilon(u^\epsilon)\eta dx'dt=0.
\end{equation}
Taking in \eqref{eq:penalty} first $\eta=u^\epsilon$, and then
$\eta=u^\epsilon_t$, one obtains the following global uniform bounds
for the family $\{u^\epsilon\}_{0<\epsilon<1}$:
\begin{align*}
  \sup_{t\in(-1,0]}\|u^\epsilon(\cdot,t)\|_{L_2(B_1^+)}^2+\|\nabla
  u^\epsilon\|_{L_2(Q_1^+)}^2&\leq C_n \big(\|\phi_0\|^2_{L_2(B_1^+)}+\|f\|_{L_2(Q_1^+)}^2\big),\\
  \sup_{t\in(-1,0]}\|\nabla
  u^\epsilon(\cdot,t)\|_{L_2(B_1^+)}^2+\|\partial_t
  u^\epsilon\|_{L_2(Q_1^+)}^2&\leq C_n \big(\|\nabla
  \phi_0\|^2_{L_2(B_1^+)}+\|f\|_{L_2(Q_1^+)}^2\big).
\end{align*}
Thus, the family $\{u^\epsilon\}_{0<\epsilon<1}$ is uniformly bounded
in $W^{1,1}_2(Q_1^+)$ and passing to the weak limit as $\epsilon\to 0$
one obtains the existence of solutions of the Signorini problem
\eqref{eq:HEu-Q1}--\eqref{eq:u-init-Q1} in the case $\phi=0$,
$g=0$. Besides, by choosing the test functions
$\eta=\beta_\epsilon(u^\epsilon-w^\epsilon)|\beta_\epsilon(u^\epsilon-w^\epsilon)|^{p-2}$,
$p>1$, where $w$ solves the boundary value problem
\begin{align*}
  \Delta w-\partial_t w=f^\epsilon&\quad\text{in }Q_1^+,\\
  w=0&\quad\text{on }\partial_pQ_1^+,
\end{align*}
one can show the global uniform bound
$$
\sup_{Q_1^+} |\beta_\epsilon(u^\epsilon)|\leq C_n
\big(\|\phi_0\|_{W^2_\infty(B_1^+)}+\|f\|_{L_{\infty}(Q_1^+)}\big).
$$
For complete details, see the proofs of Lemmas 4 and 5 in \cite{AU0}.

Next we have a series of local estimates.  With $\zeta\in
C^\infty_0(B_1)$, we take the function
$\eta=\partial_{x_i}[(\partial_{x_i}u^\epsilon)\zeta^2(x)]$,
$i=1,\ldots, n$ in \eqref{eq:penalty}. Integrating by parts, we obtain
the following local, uniform, second order estimate
\begin{align*}
  \|D^2 u^\epsilon\|_{L_2(Q_r^+)}&\leq C_{n,r}\big(\|\nabla
  u\|^2_{L_2(Q_1^+)}
  +\|f\|_{L_2(Q_1^+)})\\
  &\leq C_{n,r}\big(
  \|\phi_0\|_{L_2(B_1^+)}+\|f\|_{L_2(Q_1^+)}\big),\quad 0<r<1.
\end{align*}
One should compare with our proof of Lemma~\ref{lem:w212} in
Appendix~\ref{sec:est-gauss-proofs} which is the weighted version of
this estimate. Furthermore, with more work one can establish the
following locally uniform spatial Lipschitz bound
$$
\|u^\epsilon\|_{W^1_\infty(Q_r^+)}\leq
C_{n,r}\big(\|\phi_0\|_{W^2_\infty(B_1^+)}+\|f\|_{L_{\infty}(Q_1^+)}\big),\quad
0<r<1,
$$
see Lemma 6 in \cite{AU0}.

Finally, one can show that there exists a dimensional constant
$\alpha>0$ such that $\nabla u^\epsilon \in
H^{\alpha,\alpha/2}(Q_r^+\cup Q_r')$ for any $0<r<1$, with the
estimate
\begin{align*}
  \|\nabla u^\epsilon\|_{H^{\alpha,\alpha/2}(Q_r^+\cup Q_r')}&\leq
  C_{n,r,\rho}\big(\|\nabla u^\epsilon\|_{L_\infty(Q_\rho
    ^+)}+\|f\|_{L_{\infty}(Q_\rho^+)}\big),\quad 0<r<\rho<1\\
  &\leq
  C_{n,r}\big(\|\phi_0\|_{W^2_\infty(B_1^+)}+\|f\|_{L_{\infty}(Q_1^+)}\big),
\end{align*}
see Theorem 2.1 in \cite{AU}.

We summarize the estimates above in the following two lemmas.

\begin{lemma}\label{lem:known-W22}
  \pushQED{\qed} Let $u\in W^{1,1}_2(Q_1^+)$ be a solution of the
  Signorini problem \eqref{eq:HEu-Q1}--\eqref{eq:u-init-Q1} with $f\in
  L_{2}(Q_1^+)$, $\phi_0\in W^{1}_2(B_1^+)$, $\phi=0$, and
  $g=0$. Then, $u\in W^{2,1}_2(Q_r^+)$ for any $0<r<1$ and
  \begin{align*}
    \|u\|_{W^{2,1}_2(Q_r^+)}&\leq
    C_{n,r}\big(\|\phi_0\|_{W^1_2(B_1^+)}+\|f\|_{L_{2}(Q_1^+)}\big).\qedhere
  \end{align*}
  \popQED
\end{lemma}

\begin{lemma}\label{lem:known-Halpha}
  \pushQED{\qed} Let $u\in W^{1,1}_2(Q_1^+)$ be a solution of the
  Signorini problem \eqref{eq:HEu-Q1}--\eqref{eq:u-init-Q1} with $f\in
  L_{\infty}(Q_1^+)$, $\phi_0\in W^{2}_\infty(B_1^+)$, $\phi=0$, and
  $g=0$. Then, for any $0<r<1$, $u\in L_\infty(Q_r^+)$, $\nabla u\in
  H^{\alpha,\alpha/2}(Q_r^+\cup Q_r')$ with a dimensional constant
  $\alpha>0$ and
  \[
  \|u\|_{L_\infty(Q_r^+)}+\|\nabla u\|_{H^{\alpha,\alpha/2}(Q_r^+\cup
    Q_r')}\leq
  C_{n,r}\big(\|\phi_0\|_{W^2_\infty(B_1^+)}+\|f\|_{L_{\infty}(Q_1^+)}\big).\qedhere
  \]
  \popQED
\end{lemma}

We will also need the following variant of
Lemma~\ref{lem:known-Halpha} that does not impose any restriction on
the boundary data $g$ and $\phi_0$.

\begin{lemma}\label{lem:known-Halpha-2} Let $v\in W^{1,1}_2(Q_1^+)\cap
  W^{1,0}_\infty(Q_1^+)$ be a solution of the Signorini problem
  \eqref{eq:HEu-Q1}--\eqref{eq:u-init-Q1} with $f\in L_\infty(Q_1^+)$,
  and $\phi\in H^{2,1}(Q_1')$. Then, for any $0<r<1$, $\nabla v\in
  H^{\alpha,\alpha/2}(Q_r^+\cup Q_r')$ with a universal $\alpha>0$,
  and
  \[
  \|\nabla v\|_{H^{\alpha,\alpha/2}(Q_r^+\cup Q_r')}\leq
  C_{n,r}\big(\|v\|_{W^{1,0}_\infty(Q_1^+)}+\|f\|_{L_{\infty}(Q_1^+)}+\|\phi\|_{H^{2,1}(Q_1')}\big).
  \]
\end{lemma}
\begin{proof} Consider the function
$$
u(x,t)=[v(x,t)-\phi(x',t)] \eta(x,t)
$$
with $\eta\in C^\infty_0(Q_1^+)$, such that
$$
\eta=1\quad\text{on }Q_{r},\quad \eta(x',-x_n,t)=\eta(x',x_n,t).
$$
In particular, $\partial_{x_n}\eta=0$ on $Q_1'$. Then $u$ satisfies
the conditions of Lemma~\ref{lem:known-Halpha} with $\phi=0$, $g=0$,
$\phi_0=0$ and with $f$ replaced by
$$
[f-(\Delta'-\partial_t)\phi]\eta+(v-\phi)(\Delta-\partial_t)\eta+2(\nabla
v-\nabla\phi)\nabla \eta.
$$
The assumptions on $v$ and $\phi$ now imply the required estimate from
that in Lemma~\ref{lem:known-Halpha}.
\end{proof}

In Section~\ref{sec:estimates-w2-1_2s_1+} we generalize the estimates
in Lemma~\ref{lem:known-W22} for the appropriate weighted Gaussian
norms. The proof of these estimates are given in
Appendix~\ref{sec:est-gauss-proofs}. One of our main results in this
paper is the optimal value of the H\"older exponent $\alpha$ in
Lemma~\ref{lem:known-Halpha-2}. We show that $\nabla u\in
H^{1/2,1/4}_\loc$, or slightly stronger, that $u\in H^{3/2,3/4}_\loc$,
when $f$ is bounded, see Theorem~\ref{thm:opt-reg}.

\section{Classes of solutions}
\label{sec:classes-solutions}

In this paper we are mostly interested in local properties of the
solution $v$ of the parabolic Signorini problem and of its free
boundary. In view of this, we focus our attention on solutions in
parabolic (half-)cylinders. Furthermore, thanks to the results in
Section~\ref{sec:existence-regularity}, we can, and will assume that
such solutions possess the regularity provided by
Lemmas~\ref{lem:known-W22} and~\ref{lem:known-Halpha}.

\begin{definition}[Solutions in cylinders] Given $\phi\in
  H^{2,1}(Q_1')$, we say that $v\in\S_{\phi}(Q_1^+)$ if $v\in
  W^{2,1}_2(Q_1^+)\cap L_\infty(Q_1^+)$, $\nabla v\in
  H^{\alpha,\alpha/2}(Q_1^+\cup Q_1')$ for some $0<\alpha<1$, and $v$
  satisfies
  \begin{gather}
    \Delta v-\partial_t v = 0\quad\text{in } Q_1^+,\\
    v-\phi\geq 0,\quad -\partial_{x_n} v\geq
    0,\quad(v-\phi)\partial_{x_n} v = 0\quad\text{on } Q'_1,
  \end{gather}
  and
  \begin{equation}
    (0,0)\in\Gamma_*(v):=\partial_{Q_1'}\{(x',t)\in Q_1'\mid v(x',0,t)=\phi(x',t),\ \partial_{x_n}v(x',0,t)=0\},
  \end{equation}
  where $\partial_{Q_1'}$ is the boundary in the relative topology of
  $Q_1'$.
\end{definition}
We call the set $\Gamma_*(v)$ the \emph{extended free boundary} for
the solution $v$. Recall that the \emph{free boundary} is given by
$$
\Gamma(v):=\partial_{Q_1'}\{(x',t)\in Q_1'\mid v(x',0,t)>\phi(x',t)\}.
$$
Note that by definition $\Gamma_*(v)\supset \Gamma(v)$. The reason for
considering this extension is that parabolic cylinders do not contain
information on ``future times''.  This fact may create a problem when
restricting solutions to smaller subcylinders. The notion of extended
free boundary removes this problem. Namely, if $(x_0,t_0)\in
\Gamma_*(v)$ and $r>0$ is such that $Q_r^+(x_0,t_0)\subset Q_1^+$,
then $(x_0,t_0)\in
\Gamma_*\big(v\big|_{Q_r^+(x_0,t_0)}\big)$. Sometimes, we will abuse
the terminology and call $\Gamma_*(v)$ the free boundary.

Replacing $Q_1^+$ and $Q_1'$ by $Q_R^+$ and $Q_R'$ respectively in the
definition above, we will obtain the class $\S_\phi(Q_R^+)$. Note that
if $v\in\S_\phi(Q_R^+)$ then the parabolic rescaling
$$
v_R(x,t)=\frac{1}{C_R}v(Rx, R^2t),$$ where $C_R>0$ can be arbitrary
(but typically chosen to normalize a certain quantity), belongs to the
class $\S_{\phi_R}(Q_1^+)$. Having that in mind, we will state most of
the results only for the case $R=1$.

The function $v\in \S_\phi(Q_1^+)$ allows a natural extension to the
entire parabolic cylinder $Q_1$ by the even reflection in $x_n$
coordinate:
$$
v(x',-x_n,t):=v(x',x_n,t).
$$
Then $v$ will satisfy
$$
\Delta v-\partial_t v=0\quad\text{in }Q_1\setminus \Lambda(v),
$$
where
$$
\Lambda(v):=\{(x',t)\in Q_1'\mid v(x',0,t)=\phi(x',t)\},
$$
is the so-called \emph{coincidence set}.  More generally,
\begin{align*}
  \Delta v-\partial_t v\leq 0&\quad\text{in }Q_1,\\
  \Delta v-\partial_t v
  =2(\partial_{x_n}^+v)\mathcal{H}^n\big|_{\Lambda(v)}&\quad\text{in
  }Q_1,
\end{align*}
in the sense of distributions, where $\mathcal{H}^{n}$ is the
$n$-dimensional Hausdorff measure and by $\partial_{x_n}^+v$ we
understand the limit from the right $\partial_{x_n}v(x',0+,t)$ on
$Q_1'$.

We next show how to reduce the study of the solutions with nonzero
obstacle $\phi$ to the ones with zero obstacle.  As the simplest such
reduction we consider the difference $v(x,t)-\phi(x',t)$, which will
satisfy the Signorini conditions on $Q_1'$ with zero obstacle, but at
an expense of solving nonhomogeneous heat equation instead of the
homogeneous one. One may further extend this difference to the strip
$S_1^+=\R^n_+\times(-1,0]$ by multiplying with a cutoff function in
$x$ variables. More specifically, let $\psi\in C^\infty_0(\R^n)$ be
such that
\begin{gather}\label{eq:psi-1}
  0\leq\psi\leq 1,\quad\psi=1\quad\text{on } B_{1/2},\quad\supp
  \psi\subset B_{3/4},\\\label{eq:psi-2}
  \psi(x',-x_n)=\psi(x',x_n),\quad x\in\R^n,
\end{gather}
and consider the function
\begin{equation}
  u(x,t)=[v(x,t)-\phi(x',t)]\psi(x)\quad\text{for }(x,t)\in S_1^+.
\end{equation}
It is easy to see that $u$ satisfies the nonhomogeneous heat equation
in $S_1^+$
\begin{align*}
  \Delta u-\partial_t u &= f(x,t)\quad\text{in } S_1^+,\\
  f(x,t)&=-\psi(x)[\Delta'\phi-\partial_t\phi]+[v(x,t)-\phi(x',t)]\Delta\psi+2\nabla
  v\nabla\psi,
\end{align*}
and the Signorini boundary conditions on $S_1'$
$$ 
u\geq 0,\quad -\partial_{x_n}u\geq 0,\quad u \partial_{x_n}u =
0\quad\text{on } S'_1.
$$
Moreover, it is easy to see that $f$ is uniformly bounded.

\begin{definition}[Solutions in strips]\label{def:Sf} We say that $u\in \S^f(S_1^+)$, for $f\in
  L_\infty(S_1^+)$ if $u\in W^{2,1}_{2}(S_1^+)\cap L_\infty(S_1^+)$,
  $\nabla u\in H^{\alpha,\alpha/2}(S_1^+\cup S_1')$, $u$ has a bounded
  support and solves
  \begin{align}
    \label{eq:HEv-f}
    \Delta u-\partial_t u=f&\quad\text{in } S_1^+,\\
    \label{eq:v-signor-f} u\geq 0,\quad-\partial_{x_n}u\geq 0,\quad
    u \partial_{x_n}u = 0&\quad\text{on } S'_1,
  \end{align}
  and
$$
(0,0)\in \Gamma_*(u)=\partial_{S_1'}\{(x',t)\in S_1'\mid
u(x',0,t)=0,\ \partial_{x_n}u(x',0,t)=0\}.
$$
\end{definition}

If we only assume $\phi\in H^{2,1}(Q_1')$, then in the construction
above we can only say that the function $f\in L_\infty(S_1^+)$. For
some of the results that we are going to prove (such as the optimal
regularity in Theorem~\ref{thm:opt-reg}) this will be
sufficient. However, if we want to study a more refined behavior of
$u$ near the origin, we need to assume more regularity on $\phi$.

Thus, if we assume $\phi\in H^{\ell,\ell/2}(Q_1')$ with
$\ell=k+\gamma\geq 2$, $k\in\N$, $0<\gamma\leq 1$, then for its
parabolic Taylor polynomial $q_k(x',t)$ of parabolic degree $k$ at the
origin, we have
$$
|\phi(x',t)-q_k(x',t)|\leq M\|(x',t)\|^{\ell},\\
$$
for a certain $M>0$, and more generally
$$
|\partial_{x'}^{\alpha'}\partial_t^j\phi(x',t)
-\partial_{x'}^{\alpha'}\partial_t^j q_k(x',t)|\leq
M\|(x',t)\|^{\ell-|\alpha'|-2j},
$$
for $|\alpha'|+2j\leq k$.

To proceed, we calorically extend the polynomial $q_k(x',t)$ in the
following sense.

\begin{lemma}[Caloric extension of polynomials]\label{lem:calor-ext}
  For a given polynomial $q(x',t)$ on $\R^{n-1}\times\R$ there exists
  a caloric extension polynomial $\tilde q(x,t)$ in $\R^{n}\times\R$,
  symmetric in $x_n$, i.e.,
$$
(\Delta-\partial_t) \tilde q(x,t)=0,\quad \tilde
q(x',0,t)=q(x',t),\quad \tilde q(x',-x_n, t)=\tilde q (x',x_n,t).
$$
Moreover, if $q(x',t)$ is parabolically homogeneous of order $\kappa$,
then one can find $\tilde q (x,t)$ as above with the same homogeneity.
\end{lemma}
\begin{proof} It is easily checked that the polynomial
$$
\tilde q (x',x_n,t)=\sum_{j=0}^{N}(\partial_t-\Delta_{x'})^j
q(x',t)\frac{x_n^{2j}}{2j!}
$$
is the desired extension. Here $N$ is taken so that the parabolic
degree of the polynomial $q(x',t)$ does not exceed $2N$.
\end{proof}

Let now $\tilde q_k$ be the extension of the parabolic Taylor
polynomial $q_k$ of $\phi$ at the origin and consider
$$
v_k(x,t):= v(x,t)-\tilde
q_k(x,t),\quad\phi_k(x',t):=\phi(x',t)-q_k(x',t).
$$
It is easy to see that $v_k$ solves the Signorini problem with the
thin obstacle $\phi_k$, i.e. $v_k\in \S_{\phi_k}(Q_1^+)$, now with an
additional property
$$
|\partial_{x'}^{\alpha'}\partial_t^j\phi_k(x',t)|\leq
M\|(x',t)\|^{\ell-|\alpha'|-2|j|},\quad\text{for }|\alpha'|+2j\leq k.
$$
Then if we proceed as above and define
\begin{align*}
  u_k(x,t)&=[v_k(x,t)-\phi_k (x',t)]\psi(x)\\
  &=[v(x,t)-\phi(x,t)-(\tilde q_k(x,t)-q_k(x',t))]\psi(x)
\end{align*}
then $u_k$ will satisfy \eqref{eq:HEv-f}--\eqref{eq:v-signor-f} with
the right-hand side
$$
f_k=-\psi(x)[\Delta'\phi_k-\partial_t\phi_k]+[
v_k(x,t)-\phi_k(x',t)]\Delta\psi+2\nabla v_k\nabla\psi,
$$
which additionally satisfies
$$
|f_k(x,t)|\leq M\|(x,t)\|^{\ell-2}\quad \text{for }(x,t)\in S_1^+,
$$
for $M$ depending only on $\psi$, $\|u\|_{W^{1,0}_\infty(Q_1^+)}$, and
$\|\phi\|_{H^{\ell,\ell/2}(Q_1^+)}$. Moreover, for $|\alpha|+2j\leq
k-2$ one will also have
$$
|\partial_x^\alpha\partial_t^j f_k(x,t)|\leq M_{\alpha,j}
\|(x,t)\|^{\ell-2-|\alpha|-2j}\quad \text{for }(x,t)\in Q_{1/2}^+.
$$

We record this construction in the following proposition.

\begin{proposition}\label{prop:uk-def}
  \pushQED{\qed} Let $v\in\S_\phi(Q_1^+)$ with $\phi\in
  H^{\ell,\ell/2}(Q_1')$, $\ell=k+\gamma\geq 2$, $k\in N$,
  $0<\gamma\leq 1$. If $q_k$ is the parabolic Taylor polynomial of
  order $k$ of $\phi$ at the origin, $\tilde q_k$ is its extension
  given by Lemma~\ref{lem:calor-ext}, and $\psi$ is a cutoff function
  as in \eqref{eq:psi-1}--\eqref{eq:psi-2}, then
$$
u_k(x,t)=[v(x,t)-\tilde q_k(x,t)-(\phi(x',t)-q_k(x',t))]\psi(x)
$$
belongs to the class $\S^{f_k}(S_1^+)$ with
$$
|f_k(x,t)|\leq M\|(x,t)\|^{\ell-2}\quad \text{for }(x,t)\in S_1^+
$$
and more generally, for $|\alpha|+2j\leq k-2$,
$$
|\partial_x^\alpha\partial_t^j f_k(x,t)|\leq M_{\alpha,j}
\|(x,t)\|^{\ell-2-|\alpha|-2j}\quad \text{for }(x,t)\in Q_{1/2}^+.
$$
Furthermore, $u_k(x',0,t)=v(x',0,t)-\phi(x,'t)$,
$\partial_{x_n}u_k(x',0,t)=\partial_{x_n}v(x',0,t)$ in $Q_{1/2}'$ and
therefore
\begin{align*}
  \Gamma(u_k)\cap Q_{1/2}'&=\Gamma(v)\cap Q_{1/2}',\\
  \Gamma_*(u_k)\cap Q_{1/2}'&=\Gamma_*(v)\cap Q_{1/2}'.\qedhere
\end{align*}
\popQED
\end{proposition}

\section{Estimates in Gaussian spaces}
\label{sec:estimates-w2-1_2s_1+}
In this section we state $W^{2,1}_2$-estimates with respect to the
Gaussian measure $G(x,t)dxdt$ in the half-strips $S_\rho^+$. The
estimates involves the quantities that appear in the generalized
frequency formula that we prove in the next section. Since the
computations are rather long and technical, to help with the
readability of the paper, we have moved the proofs to
Appendix~\ref{sec:est-gauss-proofs}.

\begin{lemma}\label{lem:w212}
  Let $u\in \S^f(S_1^+)$ with $f\in L_\infty(S_1^+)$. Then, for any
  $0<\rho<1$ we have the estimates
  \begin{align*}
    \int_{S_{\rho}^+} |t||\nabla u|^2  G&\leq C_{n,\rho}\int_{S_1^+} (u^2 +|t|^2f^2) G,\\
    \int_{S_{\rho}^+} |t|^2(|D^2 u|^2+u_t^2)G &\leq
    C_{n,\rho}\int_{S_1^+} (u^2+|t|^2f^2)G.
  \end{align*}
\end{lemma}
\begin{remark}\label{rem:w212-alt}

  Even though the estimates above are most natural for our further
  purposes, we would like to note that slightly modifying the proof
  one may show that
$$
\int_{S_{\rho}^+} |\nabla u|^2 G\leq C_{n,\rho}\int_{S_1^+}
(u^2+f^2)G,
$$
i.e., without the weights $|t|$ and $|t|^2$ in the integrals for
$|\nabla u|^2$ and $f^2$.  More generally, the same estimate can be
proved if $u(\Delta -\partial_t)u\geq 0$ in $S_1$, with the
half-strips $S_\rho^+$ and $S_1^+$ replaced by the full strips
$S_\rho$ and $S_1$.
\end{remark}

\begin{lemma}\label{lem:w112-diff} Suppose $u_i\in\S^{f_i}(S_1^+)$,
  $i=1,2$, with $f_i\in L_\infty(S_1^+)$. Then for any $0<\rho<1$ we
  have the estimate
  \begin{align*}
    \int_{S_{\rho}^+} |t||\nabla (u_1-u_2)|^2 G&\leq
    C_{n,\rho}\int_{S_1^+} [(u_1-u_2)^2+|t|^2(f_1-f_2)^2 ] G.
  \end{align*}
\end{lemma}

\section{The generalized frequency function}
\label{sec:gener-freq-funct}

In this section we will establish a monotonicity formula, which will
be a key tool for our study. The origins of this formula go back to
Almgren's Big Regularity Paper \cite{Alm}, where he proved that for
(multiple-valued) harmonic functions in the unit ball, the frequency
function
$$
N_u(r)=r\frac{\int_{B_r}|\nabla u|^2}{\int_{\partial B_r} u^2}
$$
is monotonically increasing in $r\in(0,1)$. Versions of this formula
have been used in different contexts, most notably in unique
continuation \cites{GL1,GL2} and more recently in the thin obstacle
problem \cites{ACS,GP}.  Almgren's monotonicity formula has been
generalized by Poon to solutions of the heat equation in the unit
strip $S_1$. More precisely, he proved in \cite{Poo} that if $\Delta u
- u_t = 0$ in $S_1$, then its caloric frequency, defined as
$$
N_u(r)=\frac{r^2\int_{\R^n} |\nabla u|^2(x,-r^2)
  G(x,-r^2)dx}{\int_{\R^n} u(x,-r^2)^2 G(x,-r^2)dx},
$$
is monotone non-decreasing in $r\in(0,1)$. This quantity, which to a
large extent plays the same role as Almgren's frequency function,
differs from the latter in that it requires $u$ to be defined in an
entire strip. Thereby, it is not directly applicable to caloric
functions which are only locally defined, for instance when $u$ is
only defined in the unit cylinder $Q_1$. One possible remedy to this
obstruction is to consider an extension of a caloric function in $Q_1$
to the entire strip $S_1$ by multiplying it with a spatial cutoff
function $\psi$, supported in $B_1$:
$$
v(x,t)=u(x,t)\psi(x).
$$
Such extension, however, will no longer be caloric in $S_1$, and
consequently the parabolic frequency function $N_v$ is no longer going
to be monotone.  However, there is a reasonable hope that $N_v$ is
going to exhibit properties close to monotonicity. In fact, to be able
to control the error terms in the computations, we will need to
consider a ``truncated'' version of $N$. Moreover, we will be able to
extend this result to functions $u\in \S^f(S_1^+)$, and to functions
$v\in \S_\phi(Q_1^+)$, via the constructions in
Proposition~\ref{prop:uk-def}.

To proceed, we define the following quantities:
\begin{align*}
  h_u(t) &=\int_{\R^n_+} u(x, t)^2 G(x, t)dx,\\
  i_u(t) &=-t\int_{\R^n_+} |\nabla u(x, t)|^2 G(x, t)dx,
\end{align*}
for any function $u$ in the parabolic half-strip $S_1^+$ for which the
integrals involved are finite. Note that, if $u$ is an even function
in $x_n$, then Poon's parabolic frequency function is given by
$$
N_u(r)=\frac{i_u(-r^2)}{h_u(-r^2)}.
$$
There are many substantial technical difficulties involved in working
with this function directly. To overcome such difficulties, we
consider the following averaged versions of $h_u$ and $i_u$:
\begin{align*}
  H_u(r) &=\frac{1}{r^2}\int_{-r^2}^0
  h_u(t)dt=\frac{1}{r^2}\int_{S_r^+}u(x,t)^2 G(x,t)dxdt,\\
  I_u(r) &=\frac{1}{r^2}\int_{-r^2}^0
  i_u(t)dt=\frac1{r^2}\int_{S_r^+}|t||\nabla u(x,t)|^2 G(x,t)dxdt.
\end{align*}
One further obstruction is represented by the fact that the above
integrals may become unbounded near the endpoint $t=0$, where $G$
becomes singular. To remedy this problem we introduce the following
truncated versions of $H_u$ and $I_u$. For a constant $0<\delta<1$,
let
\begin{align*}
  H_u^\delta(r) &=\frac{1}{r^2}\int_{-r^2}^{-\delta^2r^2}
  h_u(t)dt=\frac{1}{r^2}\int_{S_r^+\setminus S_{\delta r}^+}u(x,t)^2
  G(x,t)dxdt,
  \\
  I_u^\delta(r) &=\frac{1}{r^2}\int_{-r^2}^{-\delta^2r^2}
  i_u(t)dt=\frac1{r^2}\int_{S_r^+\setminus S_{\delta r}^+}|t||\nabla
  u(x,t)|^2 G(x,t)dxdt.
\end{align*}

The following lemma plays a crucial role in what follows.

\begin{lemma}\label{lem:differentiation-formulae-smooth}
  Assume that $v\in C^{4,2}_0(S_1^+\cup S_1')$ satisfies
  \begin{align*}
    \Delta v-\partial_t v = g(x,t)\quad\text{in } S_1^+.
  \end{align*}
  Then, for any $0< \delta<1$ we have the following differentiation
  formulas
  \begin{align*}
    (H_v^\delta)'(r)
    &=\frac{4}{r}I_v^\delta(r)-\frac{4}{r^3}\int_{S_r^+\setminus
      S_{\delta r}} t vg G-
    \frac{4}{r^3}\int_{S'_r\setminus S'_{\delta r}} t v v_{x_n} G\\
    (I_v^\delta)'(r) &=\frac{1}{r^3}\int_{S_r^+\setminus S^+_{\delta
        r}}(Zv)^2 G+ \frac{2}{r^3}\int_{S_r^+\setminus S^+_{\delta r}}
    t (Zv) g G+\frac{2}{r^3}\int_{S'_r\setminus S'_{\delta r}} t
    v_{x_n}(Zv) G,
  \end{align*}
  where the vector field $Z$ is as in \eqref{eq:Z}.
\end{lemma}

\begin{proof}
  \setcounter{step}{0} The main step in the proof consists in
  establishing the following differentiation formulas for $-1<t<0$:
  \begin{equation}\label{eq:h'}
    h_v'(t) =\frac{2}{t} i_v(t)-2\int_{\R^n_+} v  g G-2\int_{\R^{n-1}} v v_{x_n} G,
  \end{equation}
  and
  \begin{equation}\label{eq:i'}
    i_v'(t) =\frac{1}{2t}\int_{\R^n_+}(Z v)^2 G+\int_{\R^n_+} (Zv) g G+\int_{\R^{n-1}}v_{x_n}(Zv) G.
  \end{equation}
  Once this is done, then noting that
  \begin{align*}
    H_u^\delta(r) &= \int_{-1}^{-\delta^2} h_u(r^2 s)ds,\ \ \
    I_u^\delta(r) = \int_{-1}^{-\delta^2} i_u(r^2 s)ds,
  \end{align*}
  we find
  \begin{align*}
    (H_u^\delta)'(r) &= 2r \int_{-1}^{-\delta^2} s h_u'(r^2 s)ds,\ \ \
    (I_u^\delta)'(r) = 2r \int_{-1}^{-\delta^2} s i_u'(r^2 s)ds.
  \end{align*}
  Using \eqref{eq:h'} we thus obtain
  \begin{align*}
    (H_v^\delta)'(r) & =\frac{2}{r^3}\int_{-r^2}^{-\delta^2r^2} t h_v'(t)dt\\
    &=\frac{2}{r^3}\int_{-r^2}^{-\delta^2r^2}\Big(2 i_v(t)-2\int_{\R^n_+} t v g G(\cdot,t)\,dx-2\int_{\R^{n-1}} v v_{x_n} G(\cdot,t)\,dx\Big)dt\\
    &=\frac{4}{r} I_v^\delta(r)-\frac{4}{r^3}\int_{S_r^+\setminus
      S_{\delta r}^+} t v g G-\frac{4}{r^3}\int_{S'_r\setminus
      S'_{\delta r}}t v v_{x_n} G .
  \end{align*}
  The formula for $(I_v^\delta)'(r)$ is computed similarly. We are
  thus left with proving \eqref{eq:h'} and \eqref{eq:i'}.

  \step{1-diff-form-smooth} We start with claiming that
$$
i_{v}(t)=\frac12\int_{\R^n_+} v Z v G+t\int_{\R^n_+} v g
G+t\int_{\R^{n-1}}v v_{x_n} G.
$$
Indeed, noting that $\Delta(v^2/2) = v\Delta v+|\nabla v|^2$ in
$S^+_1$, and keeping in mind that the outer unit normal to $\R^n_+$ on
$\R^{n-1}$ is given by $\nu = - e_n = (0,\ldots,0,-1)$, we integrate
by parts to obtain
\begin{align*}
  i_v(t) &=-t\int_{\R^n_+}\left(\Delta(v^2/2)-v\Delta v\right) G\\
  &= t\int_{\R^n_+} v\nabla v\nabla G +t\int_{\R^{n-1}} v v_{x_n} G+
  t\int_{\R^n_+} v v_t G+t\int_{\R^n_+} v g G\\
  &=t\int_{\R^n_+} \Big(\nabla v\frac{\nabla G}{G}+v_t\Big)v G+t\int_{\R^n_+} v g G+t\int_{\R^{n-1}} v v_{x_n} G\\
  &=\frac12\int_{\R^n_+} v (Z v) G+t\int_{\R^n_+} v g
  G+t\int_{\R^{n-1}} v v_{x_n} G,
\end{align*}
where in the first integral of the last equality we have used
\eqref{eq:ZG}. This proves the claim.

\step{2-diff-form-smooth} We now prove the formula \eqref{eq:h'} for
$h_v'$. Note that for $\lambda>0$ we have
\begin{align*}
  h_v(\lambda^2 t)&=\int_{\R^n_+} v(x,\lambda^2
  t)^2 G(x,\lambda^2t)dx\\
  &=\int_{\R^n_+} v(\lambda y,\lambda^2 t)^2 G(\lambda y,\lambda^2
  t)\lambda^n dy=\int_{\R^n_+} v(\lambda y,\lambda^2 t)^2 G(y, t)dy.
\end{align*}
Here, we have used the identity $ G(\lambda y,\lambda^2 t)\lambda^n=
G(y,t)$.  Differentiating with respect to $\lambda$ at $\lambda=1$,
and using \eqref{eq:Z-gen}, we therefore obtain
$$
2th'_v(t)=2\int_{\R^n_+} v Zv G,
$$
or equivalently
$$
h_v'(t)=\frac1t\int_{\R^n_+} v Zv G.
$$
Using now the formula for $i_v$ in \stepref{1-diff-form-smooth}, and
the fact that $\Delta v - v_t = g$, we obtain
$$
h_v'(t) =\frac{2}{t} i_v(t)-2\int_{\R^n_+} v g G-2\int_{\R^{n-1}} v
v_{x_n} G.$$

\step{3-diff-form-smooth} To obtain the differentiation formula
\eqref{eq:i'} for $i_v$, note that using the scaling properties of $
G$, similarly to what was done for $h_v$, we have
$$
i_v(\lambda^2t)=-\lambda^2t\int_{\R^n_+}|\nabla v(\lambda y,\lambda^2
t)|^2 G(y,t)dy.
$$
Differentiating with respect to $\lambda$ at $\lambda=1$, we obtain
$$
2t\, i'_v(t)=-t\int_{\R^n_+} Z(|\nabla v|^2) G-2t\int_{\R^n_+}|\nabla
v|^2 G = - t \int_{\R^n_+} \left(Z(|\nabla v|^2 + 2 |\nabla
  v|^2\right) G.
$$
We now use the following easily verifiable identity
$$
Z(|\nabla v|^2)+2|\nabla v|^2=2\nabla v\cdot\nabla(Zv),
$$
which, after substitution in the latter equation and integration by
parts, yields
\begin{align*}
  2t\, i'_v(t)&=-2t\int_{\R^n_+}\nabla v\cdot\nabla(Z
  v) G\\
  &=2t\int_{\R^n_+}\Delta v(Zv) G+2t\int_{\R^n_+}(Zv)\nabla v\nabla
  G+2t\int_{\R^{n-1}}v_{x_n}(Zv) G\\
  &=2t\int_{\R^n_+}g(Zv) G+2t\int_{\R^n_+}Zv\Big(\nabla v\frac{\nabla
    G}{G}+v_t\Big)G+2t\int_{\R^{n-1}}v_{x_n}(Zv) G\\
  &=2t\int_{\R^n_+}g(Zv) G+\int_{\R^n_+}(Zv)^2
  G+2t\int_{\R^{n-1}}v_{x_n}(Zv) G.
\end{align*}
Hence, we obtain
$$
i'_v(t)=\frac{1}{2t}\int_{\R^n_+}(Zv)^2 G+\int_{\R^n_+} g(Zv)
G+\int_{\R^{n-1}}v_{x_n}(Zv) G,
$$
which establishes \eqref{eq:i'}.
\end{proof}

\medskip
With Lemma \ref{lem:differentiation-formulae-smooth} in hands, we turn
to establishing the essential ingredient in the proof of our main
monotonicity result.

\begin{proposition}[Differentiation formulas for $H_u$ and $I_u$]
  \label{prop:diff-form-signor} Let $u\in \S^f(S_1^+)$. Then, $H_u$
  and $I_u$ are absolutely continuous functions on $(0,1)$ and for
  a.e.\ $r\in(0,1)$ we have
  \begin{align*}
    H_u'(r) &=\frac{4}{r}I_u(r)-\frac{4}{r^3}\int_{S_r^+} t u f G,
    \\
    I_u'(r) &\geq\frac{1}{r^3}\int_{S_r^+}(Zu)^2 G+
    \frac{2}{r^3}\int_{S_r^+} t (Zu)f G.
  \end{align*}
\end{proposition}

\begin{proof}
  We note that, thanks to the estimates in Lemma~\ref{lem:w212} above,
  all integrals in the above formulas are finite.  The idea of the
  proof of the proposition is to approximate $u$ with smooth solutions
  $u^\epsilon$ to the Signorini problem, apply
  Lemma~\ref{lem:differentiation-formulae-smooth}, and then pass to
  the limit in $\epsilon$.  The limit process is in fact more involved
  then one may expect.  One complication is that although we have the
  estimates in Lemma~\ref{lem:w212} for the solution $u$, we do not
  have similar estimates, uniform in $\epsilon$, for the approximating
  $u^\epsilon$. This is in fact the main reason for which we have to
  consider truncated quantities $H_{u^\epsilon}^\delta$ and
  $I_{u^\epsilon}^\delta$, let $\epsilon\to 0$ first, and then
  $\delta\to 0$. However, the main difficulty is to show that the
  integrals over $S_r'$ vanish. This is relatively easy to do for
  $H_u$, since $u u_{x_n} =0$ on $S_1'$.  On the other hand, proving
  the formula for $I'_u$ is considerably more difficult, since one has
  to justify that $u_{x_n} Zu=0$ on $S_1'$. Furthermore, the vanishing
  of this term should be interpreted in a proper sense, since we
  generally only know that $Z u\in L_2(S_1^+)$, and thus its trace may
  not even be well defined on $S_1'$.

  With this being said, in the sequel we justify only the formula for
  $I'_u$, the one for $H'_u$ being analogous, but much simpler.

  \setcounter{step}{0} \step{1-diff-form} Assume that $u$ is supported
  in $B_{R-2}^+\times (-1,0]$, $R\geq 3$. Multiplying $u$ with a
  cutoff function $\eta(t)$ such that $\eta=1$ on $[-\tilde r^2,0]$
  and $\eta=0$ on $(-1,-\tilde{\tilde r}^2]$ for $0<r<\tilde
  r<\tilde{\tilde r}<1$, without loss of generality we may assume that
  $u(\cdot,-1)=0$. We then approximate $u$ in $B_R^+\times(-1,0]$ with
  the solutions of the penalized problem
  \begin{align*}
    \Delta u^\epsilon-\partial_t u^\epsilon=f^\epsilon&\quad\text{in }
    B_R^+\times(-1,0],\\
    \partial_{x_n}u^\epsilon=\beta_\epsilon(u^\epsilon)&\quad\text{on } B_R'\times(-1,0],\\
    u^\epsilon=0&\quad\text{on }(\partial B_R)^+\times(-1,0],\\
    u^\epsilon(\cdot, -1)=0&\quad\text{on }B_R^+,
  \end{align*}
  where $f^\epsilon$ is a mollification of $f$.  For any
  $\rho\in(0,1)$, let
$$
Q_{R-1,\rho}^+=B_{R-1}^+\times(-\rho^2,0],\quad Q_{R-1,\rho}
'=B_{R-1}'\times(-\rho^2,0].
$$
From the estimates in Section~\ref{sec:existence-regularity} we have
\begin{align*}
  \|u^\epsilon\|_{W^{2,1}_2(Q_{R-1, \rho}^+)},\
  \|u^\epsilon\|_{L_\infty(Q_{R-1, \rho}^+)},\ \|\nabla
  u^\epsilon\|_{L_\infty(Q_{R-1,\rho}^+)}&\leq C(\rho,u),\\
  \max_{Q_{R-1,\rho}'}|\beta_\epsilon(u^\epsilon)|&\leq C(\rho,u),
\end{align*}
uniformly in $\epsilon\in(0,1)$.

\step{2-diff-form} Now, in order to extend $u^\epsilon$ to $S_1^+$,
pick a cutoff function $\zeta\in C^\infty_0(\R^n)$ such that
$$
0\leq\zeta\leq 1,\quad\zeta=1\quad\text{on }B_{R-2},\quad\supp
\zeta\subset B_{R-1},\quad\zeta(x',-x_n)=\zeta(x',x_n),
$$
and define
$$
v^\epsilon(x,t)=u^\epsilon(x,t)\zeta(x).
$$
Note that since $u$ is supported in $B_{R-2}^+\times(-1,0]$,
$v^\epsilon$ will converge to $u$ and we will have the uniform
estimates
$$
\|v^\epsilon\|_{W^{2,1}_2(S_\rho^+)},\
\|v^\epsilon\|_{L_\infty(S_{\rho}^+)},\ \|\nabla
v^\epsilon\|_{L_\infty(S_{\rho}^+)}\leq C(\rho,u)<\infty.
$$
For $v^\epsilon$ we have that
$$
\Delta v^\epsilon-\partial_t
v^\epsilon=f^\epsilon\zeta+u^\epsilon\Delta\zeta+2\nabla
u^\epsilon\nabla \zeta=:g^\epsilon\quad\text{in }S_1^+.
$$
It is easy to see that $g^\epsilon$ converges strongly to $f$ in
$L_2(S_\rho^+\setminus S_{\delta \rho}^+,G)$ for $0<\delta<1$.
Besides, we also have
$$
v^\epsilon_{x_n}=u^\epsilon_{x_n}\zeta,\quad Zv^\epsilon=\zeta (Z
u^\epsilon)+(Z\zeta) u^\epsilon\quad\text{on }S_1',
$$
and therefore
$$
v^\epsilon_{x_n} Zv^\epsilon=\zeta (Z\zeta)
u^\epsilon\beta_\epsilon(u^\epsilon)+\zeta^2u^\epsilon_{x_n}
(Zu^\epsilon)\quad\text{on }S_1'.
$$

\step{3-diff-form} We now fix a small $\delta>0$, apply the
differentiation formulas in
Lemma~\ref{lem:differentiation-formulae-smooth} to $v^\epsilon$ and
pass to the limit. We have
\begin{align*}
  (I^\delta_{v^\epsilon})'(r)&=\frac{1}{r^3}\int_{S_r^+\setminus
    S_{\delta r}^+}(Zv^\epsilon)^2 G+
  \frac{2}{r^3}\int_{S_r^+\setminus S_{\delta r}^+} t (Zv^\epsilon)
  g^\epsilon G +
  \frac{4}{r^3}\int_{S'_r\setminus S_{\delta r}'} t v^\epsilon_{x_n} (Zv^\epsilon) G\\
  &=J_1+J_2+J_3.
\end{align*}
\substep{3.1-diff-form} To pass to the limit in $J_1$, we note that
$Zv^\epsilon$ converges to $Zu$ weakly in $L_2(S_r^+\setminus
S_{\delta r}^+, G)$. Indeed, this follows from the uniform $W^{2,1}_2$
estimates on $v^\epsilon$ in \stepref{2-diff-form} and the boundedness
of $G$ in $S_r^+\setminus S_{\delta r}^+$. Thus, in the limit we
obtain
$$
\frac{1}{r^3}\int_{S_r^+\setminus S_{\delta r}^+}(Zu)^2 G\leq
\liminf_{\epsilon\to 0}\frac{1}{r^3}\int_{S_r^+\setminus S_{\delta
    r}^+}(Zv^\epsilon)^2 G.
$$
Note that here we cannot claim equality, as we do not have a strong
convergence of $Zv^\epsilon$ to $Zu$.

\substep{3.2-diff-form} In $J_2$, the weak convergence of
$Zv^\epsilon$ to $Zu$, combined with the strong convergence of
$g^\epsilon$ to $f$ in $L_2(S_r^+\setminus S_{\delta r}^+, G)$ is
enough to conclude that
$$
\frac{2}{r^3}\int_{S_r^+\setminus S_{\delta r}^+} t (Zu) f
G=\lim_{\epsilon\to 0}\frac{2}{r^3}\int_{S_r^+\setminus S_{\delta
    r}^+} t (Zv^\epsilon) g^\epsilon G.
$$
Moreover, the convergence will be uniform in $r\in[r_1,r_2]\subset
(0,1)$.

\substep{3.3-diff-form} Finally, we claim that $J_3\to 0$ as
$\epsilon\to 0$, i.e.,
$$
\lim_{\epsilon\to 0+}\frac{4}{r^3}\int_{S'_r\setminus S_{\delta r}'} t
v^\epsilon_{x_n} (Zv^\epsilon) G=0.
$$
Indeed, we have \newcommand{\B}{\mathcal{B}}
\begin{align*}
  \int_{S'_r\setminus S_{\delta r}'} t v^\epsilon_{x_n} (Zv^\epsilon)
  G &=\int_{S_r'\setminus S_{\delta r}'} t\zeta (Z\zeta) u^\epsilon
  \beta_\epsilon(u^\epsilon) G+\int_{S_r'\setminus S_{\delta r}'}t\zeta^2\beta_\epsilon(u^\epsilon)(Zu^\epsilon) G\\
  &=:E_1+E_2.
\end{align*}
We then estimate the integrals $E_1$ and $E_2$ separately.

\subsubstep{3.3.1-diff-form} We start with $E_1$. Recall that
$|\beta_\epsilon(u^\epsilon)|\leq C(\rho)$ in $Q_{R-1,\rho}'$. By
\eqref{eq:betaeps}, this implies that $u^\epsilon\geq-C(\rho)\epsilon$
in $Q_{R-1,\rho}'$ and therefore
$$
|E_1|\leq C(\rho)r^3\epsilon,\quad 0<r\leq\rho<1.
$$ 

\subsubstep{3.3.2-diff-form} A similar estimate holds also for $E_2$,
but the proof is a little more involved. To this end consider
$$
\B_\epsilon(t)=\int_{0}^t\beta_\epsilon(s)ds,\quad t\in\R.
$$
From \eqref{eq:betaeps}, it is easy to see that
$$
\B_\epsilon(t)=0\quad\text{for }t>0,\quad \mathcal{B}_\epsilon(t)\geq
0\quad\text{for all }t,\quad \B_\epsilon(t)=C_\epsilon+\epsilon
t+\frac{t^2}{2\epsilon}\quad\text{for }t\leq -2\epsilon^2,
$$
with $C_\epsilon=\mathcal{B}_\epsilon(-2\epsilon^2)\leq
2\epsilon^3$. Note that the uniform bound $u^\epsilon\geq
-C(\rho)\epsilon$ in $Q_{R-1,\rho}'$ implies that
$$
\max_{Q_{R-1,\rho}'}|\B_\epsilon(v^\epsilon)|\leq C(\rho)\epsilon.
$$
To use this fact, note that
\begin{align*}
  E_2&=\int_{S_r'\setminus S_{\delta r}'}
  t\zeta^2\beta_\epsilon(u^\epsilon)(Z
  u^\epsilon) G = \int_{S_r'\setminus S_{\delta r}'}t\zeta^2 [Z\B_\epsilon(u^\epsilon)] G\\
  &=\int_{S_r'\setminus S_{\delta
      r}'}Z[t\zeta^2\B_\epsilon(u^\epsilon)]
  G-\int_{S_r'\setminus S_{\delta r}'}Z(t\zeta^2)\B_\epsilon(u^\epsilon) G\\
  &=:E_{22}-E_{21}.
\end{align*}
\subsubsubstep{3.3.2.1-diff-form} The estimate for $E_{21}$ is
straightforward:
$$
|E_{21}|\leq C(\rho)r^3\epsilon,\quad 0<r\leq\rho.
$$
\subsubsubstep{3.3.2.2-diff-form} To estimate $E_{22}$, denote
$U^\epsilon=t\zeta^2\B_\epsilon(u^\epsilon)$. Then, substituting
$t=-\lambda^2$, $x'=\lambda y'$, we have
\begin{align*}
  E_{22}&=\int_{S_r'}Z U^\epsilon G\, dx' dt=\int_{\delta
    r}^r\int_{\R^{n-1}}(Z U^\epsilon)(\lambda
  y',-\lambda^2) G(\lambda y',-\lambda^2)2\lambda^{n} dy'd\lambda\\
  &=2\int_{\delta r}^r\int_{\R^{n-1}}\lambda\frac{d}{d\lambda}
  U^\epsilon(\lambda y',-\lambda^2) G(y',-1)dy'd\lambda\\
  &=2\int_{\R^{n-1}} [r U^\epsilon(r
  y',-r^2) -\delta r U^\epsilon(\delta r y',-\delta^2r^2)]G(y',-1)dy'\\
  &\qquad-2\int_{\delta r}^r\int_{\R^{n-1}} U^\epsilon(\lambda
  y',-\lambda^2) G(y',-1)dy'd\lambda
\end{align*}
and consequently
$$
|E_{22}|\leq C(\rho)r^3\epsilon,\quad 0<r\leq\rho.
$$
Combining the estimates in
\stepref{3.3.1-diff-form}--\stepref{3.3.2-diff-form} above, we obtain
$$
|J_3|=\Big|\frac4{r^3}\int_{S'_r\setminus S'_{\delta r}} t
v^\epsilon_{x_n} (Zv^\epsilon) G\Big|\leq C(\rho)\epsilon\to 0,\quad
0<r\leq\rho.
$$
\step{4-diff-form} Now, writing for any $0<r_1<r_2<1$
$$
I_{v^\epsilon}^\delta(r_2)-I_{v^\epsilon}^\delta(r_1)=\int_{r_1}^{r_2}
(I_{v^\epsilon}^\delta)'(r) dr,
$$
collecting the facts proved in
\stepref{3.1-diff-form}--\stepref{3.3-diff-form}, and passing to the
limit as $\epsilon\to 0$, we obtain
$$
I_u^\delta(r_2)-I_u^\delta(r_1)\geq \int_{r_1}^{r_2}
\bigg(\frac{1}{r^3}\int_{S_r^+\setminus S_{\delta r}^+}(Zu)^2 G+
\frac{2}{r^3}\int_{S_r^+\setminus S_{\delta r}^+} t (Zu)f G\bigg)dr.
$$
Next, note that by Lemma~\ref{lem:w212} the integrands are uniformly
bounded with respect to $\delta$ (for fixed $r_1$ and
$r_2$). Therefore, we can let $\delta\to 0$ in the latter inequality,
to obtain
$$
I_u(r_2)-I_u(r_1)\geq \int_{r_1}^{r_2}
\bigg(\frac{1}{r^3}\int_{S_r^+}(Zu)^2 G+ \frac{2}{r^3}\int_{S_r^+} t
(Zu)f G\bigg)dr.
$$
This is equivalent to the sought for conclusion for $I_u'$.
\end{proof}

To state the main result of this section, the generalized frequency
formula, we need the following notion. We say that a positive function
$\mu(r)$ is a $\log$-convex function of $\log r$ on $\R^+$ if $\log
\mu(e^t)$ is a convex function of $t$. This simply means that
\[
\mu(e^{(1-\lambda)s + \lambda t}) \leq \mu(e^s)^{1-\lambda}
\mu(e^t)^\lambda,\quad 0\le \lambda \leq 1.
\]
This is equivalent to saying that $\mu$ is locally absolutely
continuous on $\R_+$ and $r\mu'(r)/\mu(r)$ is nondecreasing. For
instance, $\mu(r)=r^\kappa$ is a $\log$-convex function of $\log r$
for any $\kappa$. The importance of this notion in our context is that
Almgren's and Poon's frequency formulas can be regarded as
$\log$-convexity statements in $\log r$ for the appropriately defined
quantities $H_u(r)$.

\begin{theorem}[Generalized frequency formula]\label{thm:thin-monotonicity} Let $u\in\S^f(S_1^+)$
  with $f$ satisfying the following condition: there is a positive
  monotone nondecreasing $\log$-convex function $\mu(r)$ of $\log r$,
  and constants $\sigma>0$ and $C_\mu > 0$, such that
$$
\mu(r)\geq C_\mu r^{4-2\sigma}\int_{\R^n} f^2(\cdot,-r^2)
G(\cdot,-r^2)\,dx.
$$
Then, there exists $C>0$, depending only on $\sigma$, $C_\mu$ and $n$,
such that the function
\[
\Phi_u(r) = \frac 12 r e^{C
  r^{\sigma}}\frac{d}{dr}\log\max\{H_u(r),\mu(r)\}+2 (e^{C
  r^{\sigma}}-1)
\]
is nondecreasing for $r\in(0, 1)$.
\end{theorem}

Note that on the open set where $H_u(r)>\mu(r)$ we have $\Phi_u(r)\sim
\frac12 r H_u'(r)/H_u(r) $ which coincides with $2N_u$, when
$f=0$. The purpose of the ``truncation'' of $H_u(r)$ with $\mu(r)$ is
to control the error terms in computations that appear from the
right-hand-side $f$.

\begin{proof}[Proof of Theorem~\ref{thm:thin-monotonicity}]
  First, we want to make a remark on the definition of $r\mapsto
  \Phi_u(r)$, for $r\in (0,1)$. The functions $H_u(r)$ and $\mu(r)$
  are absolutely continuous and therefore so is $\max\{H_u(r),
  \mu(r)\}$. It follows that $\Phi_u$ is uniquely identified only up
  to a set of measure zero. The monotonicity of $\Phi_u$ should be
  understood in the sense that there exists a monotone increasing
  function which equals $\Phi_u$ almost everywhere. Therefore, without
  loss of generality we may assume that
  $$
  \Phi_u(r)=\frac12re^{Cr^\sigma}\frac{\mu'(r)}{\mu(r)}+2(e^{Cr^\sigma}-1)
  $$
  on $\mathcal{F}=\{H_u(r)\leq \mu(r)\}$ and
  $$
  \Phi_u(r)=\frac12re^{Cr^\sigma}\frac{H_u'(r)}{H_u(r)}+2(e^{Cr^\sigma}-1)
  $$
  in $\mathcal{O}=\{H_u(r)>\mu(r)\}$. Following an idea introduced in
  \cites{GL1,GL2} we now note that it will be enough to check that
  $\Phi_u'(r)>0$ in $\mathcal{O}$. Indeed, from the assumption on
  $\mu$, it is clear that that $\Phi_u$ is monotone on
  $\mathcal{F}$. Next, if $(r_0,r_1)$ is a maximal open interval in
  $\mathcal{O}$, then $H_u(r_0)=\mu(r_0)$ and $H_u(r_1)=\mu(r_1)$
  unless $r_1=1$. Besides, if $\Phi_u$ is monotone in $(r_0,r_1)$, it
  is easy to see that the limits $H_u'(r_0+)$ and $H_u'(r_1-)$ will
  exist and satisfy
$$
\mu'(r_0+)\leq H_u'(r_0+),\quad H_u'(r_1-)\leq \mu'(r_1-)\quad
(\text{unless }r_1=1)
$$
and therefore we will have
  $$
  \Phi_u(r_0)\leq \Phi_u(r_0+)\leq \Phi_u(r_1-)\leq \Phi_u(r_1),
  $$
  with the latter inequality holding when $r_1<1$. This will imply the
  monotonicity of $\Phi_u$ in $(0,1)$.

  Therefore, we will concentrate only on the set
  $\mathcal{O}=\{H_u(r)>\mu(r)\}$, where the monotonicity of
  $\Phi_u(r)$ is equivalent to that of
  \[
  \Big(r\frac{H_u'(r)}{H_u(r)}+4\Big) e^{C r^{\sigma}}=2\Phi_u(r)+4.
  \]
  The latter will follow, once we show that
  \[
  \frac{d}{dr}\Big(r\frac{H_u'(r)}{H_u(r)}\Big)\geq-C\Big(r\frac{H_u'(r)}{H_u(r)}+4\Big)
  r^{-1+\sigma}
  \]
  in $\mathcal{O}$.  Now, from Proposition~\ref{prop:diff-form-signor}
  we have
  \begin{align*}
    r\frac{H_u'(r)}{H_u(r)} &=
    4\frac{I_u(r)}{H_u(r)}-\frac{4}{r^2}\frac{\int_{S_r^+} t u f
      G}{H_u(r)}:=4E_1(r)+4 E_2(r).
  \end{align*}
  We then estimate the derivatives of each of the quantities $E_i(r)$,
  $i=1,2$.

  \setcounter{step}{0} \step{1-mon-form} Using the differentiation
  formulas in Proposition~\ref{prop:diff-form-signor}, we compute
  \begin{align*}
    r^5 H_u^2(r) E_1'(r)&=r^5 H_u^2(r)\frac{d}{dr}\Big(\frac{I_u(r)}{H_u(r)}\Big)\\
    &= r^5(I_u'(r) H_u(r)-I_u(r)H_u'(r))\\
    &\geq r^2 H_u(r)\Big(\int_{S_r^+}(Zu)^2 G+2\int_{S_r^+} t (Zu)
    f G\Big)\\
    &\qquad-r^2 I_u(r) r^3 H_u'(r)\\
    &= r^2 H_u(r)\Big(\int_{S_r^+}(Zu+t f)^2 G-\int_{S_r^+} t^2
    f^2 G\Big)\\&\qquad -\Big(\int_{S_r^+}t u f G+\frac{r^3}{4}H_u'(r)\Big)r^3 H_u'(r)\\
    &=\int_{S_r^+} u^2 G\Big(\int_{S_r^+}(Zu+t f)^2 G-\int_{S_r^+} t^2
    f^2 G\Big)\\
    &\qquad-\Big(\frac{r^3}{2}H_u'(r)+\int_{S_r^+} t u f G\Big)^2+\Big(\int_{S_r^+} t u f G\Big)^2\\
    &=\Big[\int_{S_r^+} u^2 G\int_{S_r^+}(Zu+t f)^2
    G-\Big(\int_{S_r^+} u(Zu+t f) G\Big)^2\Big]\\&\qquad-\int_{S_r^+}
    u^2 G\int_{S_r^+} t^2 f^2 G+\Big(\int_{S_r^+} t u f G\Big)^2.
  \end{align*}
  Applying the Cauchy-Schwarz inequality we obtain that the term in
  square brackets above is nonnegative. Therefore,
  \begin{align*}
    r^5 H_u^2(r) E_1'(r) &\geq-\int_{S_r^+} u^2 G\int_{S_r^+} t^2 f^2 G\\
    \intertext{or equivalently,} E_1'(r) &\geq-\frac{1}{r^{3}
      H_u(r)}\int_{S_r^+} t^2 f^2 G.
  \end{align*}

  \step{2-mon-form} We next estimate the derivative of $E_2(r)$.
  \begin{align*}
    E_2'(r)&=\frac{d}{dr}\bigg(-\frac{1}{r^2}\frac{\int_{S_r^+} t u
      f G}{H_u(r)}\bigg)\\
    &=\frac{2}{r^3}\frac{\int_{S_r^+} t u f G}{H_u(r)}-\frac{2}{r}\frac{\int_{\R^n_+}(-r^2) u(\cdot,-r^2)f(\cdot,-r^2) G(\cdot,-r^2)}{H_u(r)}\\
    &\qquad+\frac{1}{r^2}\frac{H_u'(r)\int_{S_r^+} t u f G}{H_u(r)^2}\\
    &\geq-\frac{2}{r^3 H_u(r)}\Big(\int_{S_r^+} u^2 G\Big)^{1/2}\Big(\int_{S_r^+} t^2 f^2 G\Big)^{1/2}\\
    &\qquad-\frac{2}{r H_u(r)}\Big(\int_{\R^n_+}u^2(\cdot,-r^2) G(\cdot,-r^2)\Big)^{1/2}\Big(r^4\int_{\R^n_+}f^2(\cdot,-r^2) G(\cdot,-r^2)\Big)^{1/2}\\
    &\qquad-r\frac{H_u'(r)}{H_u(r)}\frac{1}{r^3 H_u(r)}\Big(\int_{S_r^+} u^2 G\Big)^{1/2}\Big(\int_{S_r^+} t^2 f^2 G\Big)^{1/2}\\
    &=-\frac{2}{r^2H_u(r)^{1/2}}\Big(\int_{S_r^+} t^2 f^2 G\Big)^{1/2}\\
    &\qquad-\frac{2}{r H_u(r)}\Big(\int_{\R^n_+}u^2(\cdot,-r^2) G(\cdot,-r^2)\Big)^{1/2}\Big(r^4\int_{\R^n_+}f^2(\cdot,-r^2) G(\cdot,-r^2)\Big)^{1/2}\\
    &\qquad-r\frac{H_u'(r)}{H_u(r)}\frac{1}{r^2H_u(r)^{1/2}}\Big(\int_{S_r^+}
    t^2 f^2 G\Big)^{1/2}.
  \end{align*}

  \step{3-mon-form} Combining together the estimates for $E_1'(r)$ and
  $E_2'(r)$, we have
  \begin{align*}
    \frac{d}{dr}\Big(r\frac{H_u'(r)}{H_u(r)}\Big) &\geq-\frac{4}{r^3 H_u(r)}\int_{S_r^+} t^2 f^2 G-\frac{8}{r^2H_u(r)^{1/2}}\Big(\int_{S_r^+} t^2 f^2 G\Big)^{1/2}\\
    &\quad-\frac{8}{r H_u(r)}\Big(\int_{\R^n_+}u^2(\cdot,-r^2) G(\cdot,-r^2)\Big)^{1/2}\Big(r^4\int_{\R^n_+}f^2(\cdot,-r^2) G(\cdot,-r^2)\Big)^{1/2}\\
    &\quad-r\frac{H_u'(r)}{H_u(r)}\frac{4}{r^2H_u(r)^{1/2}}\Big(\int_{S_r^+}
    t^2 f^2 G\Big)^{1/2}.
  \end{align*}
  We estimate the third term separately. First, from
  \[
  \frac{d}{dr}\int_{S_r^+} u^2 G = 2r\int_{\R^n_+} u^2(\cdot,-r^2)
  G(\cdot,-r^2)\,dx
  \]
  we have
  \begin{align*}
    \int_{\R^n_+} u^2(\cdot,-r^2) G(\cdot,-r^2)&=\frac{1}{2r}\frac{d}{dr}(r^2 H_u(r))\\
    &= H_u(r)+\frac{r}{2} H_u'(r).
  \end{align*}
  From here we see that
  \[
  1+\frac{r}2\frac{H_u'(r)}{H_u(r)}\geq 0
  \]
  and therefore we also have
  \begin{align*}
    2\Big(\int_{\R^n_+} u^2(\cdot,-r^2) G(\cdot,-r^2)\Big)^{1/2} &= 2H_u(r)^{1/2}\left(1+\frac{r}{2}\frac{H_u'(r)}{H_u(r)}\right)^{1/2}\\
    &\leq
    H_u(r)^{1/2}\left(2+\frac{r}{2}\frac{H_u'(r)}{H_u(r)}\right).
  \end{align*}
  Substituting into the inequality above, we then have
  \begin{align*}
    \frac{d}{dr}\Big(r\frac{H_u'(r)}{H_u(r)}\Big) &\geq-\frac{4}{r^3 H_u(r)}\int_{S_r^+} t^2 f^2 G-\frac{8}{r^2H_u(r)^{1/2}}\Big(\int_{S_r^+} t^2 f^2 G\Big)^{1/2}\\
    &\quad-\Big(2+\frac{r}{2}\frac{H_u'(r)}{H_u(r)}\Big)\frac{4}{rH_u(r)^{1/2}}\Big(r^4\int_{\R^n_+}f^2(\cdot,-r^2) G(\cdot,-r^2)\Big)^{1/2}\\
    &\quad-r\frac{H_u'(r)}{H_u(r)}\frac{4}{r^2H_u(r)^{1/2}}\Big(\int_{S_r^+}
    t^2 f^2 G\Big)^{1/2}.
  \end{align*}
  On the set $\{H_u(r) >\mu(r)\}$, we easily have
  \begin{align*}
    H_u(r)&\geq C_\mu r^{4-2\sigma}\int_{\R^n_+} f^2(\cdot,-r^2) G(\cdot,-r^2),\\
    H_u(r)&\geq C_\mu r^{-2-2\sigma}\int_{S_r^+} t^2 f^2 G,
  \end{align*}
  and consequently,
  \begin{align*}
    \frac{d}{dr}\Big(r\frac{H_u'(r)}{H_u(r)}\Big) &\geq-C r^{-1+2\sigma}-C r^{-1+\sigma}-\Big(2+\frac{r}{2}\frac{H'_u(r)}{H_u(r)}\Big) C r^{-1+\sigma}-C r\frac{H_u'(r)}{H_u(r)} r^{-1+\sigma}\\
    &\geq-C\Big(r\frac{H_u'(r)}{H_u(r)}+4\Big) r^{-1+\sigma}.
  \end{align*}
  Note that in the last step we have again used the fact that
  $1+\frac{r}2 \frac{H_u'(r)}{H_u(r)}\geq 0$.  The desired conclusion
  follows readily.
\end{proof}

\section{Existence and homogeneity of blowups}
\label{sec:exist-homog-blow}
In this section, we show how the generalized frequency formula in
Theorem~\ref{thm:thin-monotonicity} can be used to study the behavior
of the solution $u$ near the origin. The central idea is to consider
some appropriately normalized rescalings of $u$, indicated with $u_r$
(see Definition~\ref{def:resc-H}), and then pass to the limit as $r\to
0+$ (see Theorem~\ref{thm:exist-homogen-blowups}). The resulting
limiting functions (over sequences $r=r_j\to 0+$) are known as
\emph{blowups}.  However, because of the truncation term $\mu(r)$ in
the generalized frequency function $\Phi_u(r)$, we can show the
existence of blowups only when the growth rate of $u$ can be
``detected,'' in a certain proper sense to be made precise
below. Finally, as a consequence of the monotonicity of $\Phi_u(r)$,
we obtain that the blowups must be parabolically homogeneous solutions
of the Signorini problem in $S_\infty=\R^n\times(-\infty,0]$.

Henceforth, we assume that $u\in \S^f(S_1^+)$, and that $\mu(r)$ be
such that the conditions of Theorem~\ref{thm:thin-monotonicity} are
satisfied. In particular, we assume that
\begin{equation*}
  r^{4}\int_{\R^n} f^2(\cdot,-r^2) G(\cdot,-r^2)\,dx \le \frac{r^{2\sigma} \mu(r)}{C_\mu}.
\end{equation*}
Consequently, Theorem~\ref{thm:thin-monotonicity} implies that the
function
\[
\Phi_u(r) = \frac12 r e^{C
  r^{\sigma}}\frac{d}{dr}\log\max\{H_u(r),\mu(r)\}+2(e^{C
  r^{\sigma}}-1)
\]
is nondecreasing for $r\in (0,1)$. Hence, there exists the limit
\begin{equation}\label{eq:kappa}
  \kappa:=\Phi_u(0+)=\lim_{r\to 0+}\Phi_u(r).
\end{equation}
Since we assume that $r\mu'(r)/\mu(r)$ is nondecreasing, the limit
\begin{equation}\label{eq:kappamu}
  \kappa_\mu:=\frac12\lim_{r\to 0+}\frac{r\mu'(r)}{\mu(r)}
\end{equation}
also exists. We then have the following basic proposition concerning
the values of $\kappa$ and $\kappa_\mu$.

\begin{lemma}\label{lem:Hu-mu} Let $u\in \S^f(S_1^+)$ and $\mu$ satisfy the
  conditions of Theorem~\ref{thm:thin-monotonicity}. With $\kappa$,
  $\kappa_\mu$ as above, we have
  \begin{equation*}
    \kappa\leq \kappa_\mu.
  \end{equation*}
  Moreover, if $\kappa<\kappa_\mu$, then there exists $r_u>0$ such
  that $H_u(r)\geq \mu(r)$ for $0<r\leq r_u$. In particular,
$$
\kappa=\frac12\lim_{r\to 0+}\frac{r H'_u(r)}{H_u(r)}=2\lim_{r\to
  0+}\frac{I_u(r)}{H_u(r)}.
$$
\end{lemma}
\begin{proof}
  As a first step we show that
  \begin{equation}\label{eq:Hmu}
    \kappa\not= \kappa_\mu\quad\Rightarrow\quad  \text{there exists}\ r_u>0\
    \text{such that}\ H_u(r)\geq \mu(r)\text{ for } 0<r\leq
    r_u.
  \end{equation}
  Indeed, if the implication claimed in \eqref{eq:Hmu} fails, then for
  a sequence $r_j\to 0+$ we have $H_{u}(r_j)< \mu(r_j)$. This implies
  that
$$
\Phi_u(r_j)=\frac12
r_je^{C{r_j^\sigma}}\frac{\mu'(r_j)}{\mu(r_j)}+2(e^{Cr_j^\sigma}-1),
$$
and therefore
$$
\kappa=\lim_{j\to\infty} \Phi_u(r_j)=\frac12\lim_{j\to
  \infty}r_j\frac{\mu'(r_j)}{\mu(r_j)}=\kappa_\mu,
$$
which contradicts $\kappa\not= \kappa_\mu$. We have thus proved
\eqref{eq:Hmu}.

Now, \eqref{eq:Hmu} implies that, if $\kappa\not= \kappa_\mu$, then
$$
\Phi_u(r)=\frac12
re^{C{r^\sigma}}\frac{H_u'(r)}{H_u(r)}+2(e^{Cr^\sigma}-1),\quad
0<r<r_u.
$$
Passing to the limit, we conclude that, if $\kappa\not= \kappa_\mu$,
then
\begin{equation}\label{eq:firstlim}
  \kappa=\lim_{r\to 0+} \Phi_u(r)=\frac12\lim_{r\to 0+} \frac{r H'_u(r)}{H_u(r)}.
\end{equation}
However, in this case we also have
\begin{equation}\label{eq:secondlim}
  \frac12\lim_{r\to 0+}\frac{r H'_u(r)}{H_u(r)}=2\lim_{r\to
    0+}\frac{I_u(r)}{H_u(r)}.
\end{equation}
Indeed, recall that (see Proposition~\ref{prop:diff-form-signor})
$$
r\frac{H'_u(r)}{H_u(r)}=4\frac{I_u(r)}{H_u(r)}-\frac{4}{r^2}\frac{\int_{S_r^+}
  tuf G}{H_u(r)},
$$
and \eqref{eq:secondlim} will follow once we show that
$$
\lim_{r\to 0+}\frac{1}{r^2}\frac{\int_{S_r^+} tuf G}{H_u(r)}=0.
$$
Using the Cauchy-Schwarz inequality, we have
\begin{align*}
  \frac{\int_{S_r^+} tuf G}{r^2H_u(r)}&\leq\frac{\big(\int_{S_r^+}
    t^2f^2 G\big)^{1/2}\big(\int_{S_r^+}u^2
    G\big)^{1/2}}{r^2H_u(r)}=\frac{\big(\int_{S_r^+}
    t^2f^2 G\big)^{1/2}}{rH_u(r)^{1/2}}\\
  &\leq \Big(\frac{\mu(r)}{C_\mu H_u(r)}\Big)^{1/2}r^\sigma\leq
  C^{-1}_\mu r^\sigma\to 0,
\end{align*}
where in the last inequality before the limit we have used
\eqref{eq:Hmu}.  Summarizing, the assumption $\kappa\not=\kappa_\mu$
implies \eqref{eq:firstlim}--\eqref{eq:secondlim} above. Therefore,
the proof will be completed if we show that the case
$\kappa>\kappa_\mu$ is impossible.

So, assume towards a contradiction that $\kappa>\kappa_\mu$, and fix
$0<\epsilon <\kappa - \kappa_\mu$. For such $\epsilon$ choose
$r_\epsilon>0$ so that
$$
\frac{r H_u'(r)}{H_u(r)}> 2\kappa-\epsilon,\quad \frac{r
  \mu'(r)}{\mu(r)}<2\kappa_\mu+\epsilon,\quad 0<r<r_\epsilon.
$$
Integrating these inequalities from $r$ to $r_\epsilon$, we obtain
$$
H_u(r)\leq \frac{H_u(r_\epsilon)}{r_\epsilon^{2\kappa-\epsilon}}
r^{2\kappa-\epsilon},\quad \mu(r)\geq
\frac{\mu(r_\epsilon)}{r_\epsilon^{2\kappa_\mu+\epsilon}}r^{2\kappa_\mu+\epsilon}.
$$
Since by our choice of $\epsilon>0$ we have
$2\kappa-\epsilon>2\kappa_\mu+\epsilon$, the above inequalities imply
that $H_u(r)<\mu(r)$ for small enough $r$, contrary to the established
conclusion of \eqref{eq:Hmu} above. Hence, the case
$\kappa>\kappa_\mu$ is impossible, which implies that we always have
$\kappa\leq \kappa_\mu$.
\end{proof}

To proceed, we define the appropriate notion of rescalings that works
well with the generalized frequency formula.

\begin{definition}[Rescalings]\label{def:resc-H} For $u\in\S^f(S_1^+)$
  and $r>0$ define the \emph{rescalings}
$$
u_r(x,t):=\frac{u(rx, r^2 t)}{H_u(r)^{1/2}},\quad (x,t)\in
S_{1/r}^+=\R^n_ +\times(-1/r^2,0].
$$
\end{definition}
It is easy to see that the function $u_r$ solves the nonhomogeneous
Signorini problem
\begin{align*}
  \Delta u_r-\partial_t u_r = f_r(x,t)&\quad\text{in } S_{1/r}^+,
  \\
  u_r\geq 0,\quad-\partial_{x_n}u_r\geq 0,\quad u_r\partial_{x_n} u_r
  = 0&\quad\text{on } S'_{1/r},
\end{align*}
with
$$
f_r(x,t)=\frac{r^2f(rx,r^2t)}{H_u(r)^{1/2}}.
$$
In other words, $u_r\in \S^{f_r}(S_{1/r}^+)$. Further, note that $u_r$
is normalized by the condition
$$
H_{u_r}(1)=1,
$$
and that, more generally, we have
$$
H_{u_r}(\rho)=\frac{H_u(\rho r)}{H_u(r)}.
$$
We next show that, unless we are in the borderline case
$\kappa=\kappa_\mu$, we will be able to study the so-called blowups of
$u$ at the origin. The condition $\kappa<\kappa_\mu$ below can be
understood, in a sense, that we can ``detect'' the growth of $u$ near
the origin.

\begin{theorem}[Existence and homogeneity of blowups]
  \label{thm:exist-homogen-blowups}
  Let $u\in \S^f(S_1^+)$, $\mu$ satisfy the conditions of
  Theorem~\ref{thm:thin-monotonicity}, and
$$
\kappa:=\Phi_u(0+)< \kappa_\mu= \frac12\lim_{r\to
  0+}r\frac{\mu'(r)}{\mu(r)}.
$$
Then, we have:
\begin{enumerate}[label=\textup{\roman*)}]
\item For any $R>0$, there is $r_{R,u}>0$ such that
$$
\int_{S_R^+}(u_r^2+|t||\nabla u_r|^2+|t|^2|D^2u_r|^2+|t|^2 (\partial_t
u_r)^2) G\leq C(R),\quad 0<r<r_{R,u}.
$$
\item There is a sequence $r_j\to 0+$, and a function $u_0$ in
  $S^+_\infty=\R^n_+\times(-\infty, 0]$, such that
$$
\int_{S_R^+}(|u_{r_j}-u_0|^2+|t||\nabla (u_{r_j}-u_0)|^2) G \to 0.
$$
We call any such $u_0$ a \emph{blowup} of $u$ at the origin.
\item $u_0$ is a nonzero \emph{global solution} of Signorini problem:
  \begin{align*}
    \Delta u_0-\partial_t u_0=0&\quad\text{in }S_\infty^+\\
    u_0\geq 0,\quad -\partial_{x_n}u_0\geq 0,\quad
    u_0\partial_{x_n}u_0=0&\quad\text{on }S_\infty',
  \end{align*}
  in the sense that it solves the Signorini problem in every $Q_R^+$.
\item $u_0$ is parabolically homogeneous of degree $\kappa$:
$$
u_0(\lambda x,\lambda^2t)=\lambda^\kappa u_0(x,t),\quad (x,t)\in
S_\infty^+,\ \lambda>0
$$
\end{enumerate}
\end{theorem}

The proof of Theorem \ref{thm:exist-homogen-blowups} is based on the
following lemmas.

\begin{lemma}\label{lem:Hu-est}If $\kappa<\kappa'<\kappa_\mu$ and $r_u>0$ is such that
  $\Phi_u(r)<\kappa'$ and $H_u(r)\geq \mu(r)$ for $0<r< r_u$, then
  \begin{align*}
    H_{u_r}(\rho)&\geq \rho^{2\kappa'}\quad\text{for any
    }0<\rho\leq 1,\ 0<r<r_u,\\
    H_{u_r}(R) &\leq R^{2\kappa'}\quad\text{for any }R\geq 1,\
    0<r<r_u/R.
  \end{align*}
\end{lemma}
\begin{proof} From the assumptions we have
$$
\Phi_u(r)=\frac12
re^{Cr^\sigma}\frac{H'_u(r)}{H_u(r)}+2(e^{Cr^\sigma}-1)\leq \kappa',
$$
for $0<r<r_u$, which implies that
$$
\frac{H_u'(r)}{H_u(r)}\leq \frac{2\kappa'}{r}.
$$
Integrating from $\rho r$ to $r$ and exponentiating, we find
$$
\frac{H_u(r)}{H_u(\rho r)}\leq \rho^{-2\kappa'},
$$
which implies that
$$
H_{u_r}(\rho)=\frac{H_u(\rho r)}{H_u(r)}\geq \rho^{2\kappa'}.
$$
Similarly, integrating from $r$ to $Rr$ (under the assumption that
$Rr\leq r_u$) we find
$$
H_{u_r}(R)=\frac{H_u(Rr)}{H_u(r)}\leq R^{2\kappa'}.
$$
\end{proof}

\begin{lemma}\label{lem:fr-est} Under the notations of the previous
  lemma, for any $R\geq 1$ and $0<r<r_{u}/R$, we have
$$
\int_{S_R^+} t^2 f_r^2 G\leq c_\mu R^{2+2\sigma+2\kappa'} r^{2\sigma}.
$$
\end{lemma}

\begin{proof}
  Note that from the assumptions we have
  \begin{align*}
    \mu(r)&\geq C_\mu r^{4-2\sigma}\int_{\R_n^+} f^2(\cdot,-r^2) G(\cdot,-r^2),\\
    \mu(r)&\geq C_\mu r^{-2-2\sigma}\int_{S_r^+} t^2 f^2 G.
  \end{align*}
  Now take $R\geq 1$. Then, making the change of variables and using
  the inequalities above, we have
  \begin{align*}
    \int_{S_R^+} t^2 f_r^2 G&=\frac{r^4}{H_u(r)}\int_{S_R^+}t^2
    f(rx,r^2t)^2 G(x,t)dxdt\\
    &=\frac{1}{r^2 H_u(r)}\int_{S_{Rr}^+}t^2 f^2 G\leq
    R^{2+2\sigma}r^{2\sigma}\frac{\mu(Rr)}{C_\mu H_u(r)}.
  \end{align*}
  Thus, if $0<r< r_{u}/R$, then $H_u(Rr)\geq \mu(Rr)$ and therefore
$$
\int_{S_R^+} t^2 f_r^2 G\leq c_\mu R^{2+2\sigma}
r^{2\sigma}\frac{H_u(Rr)}{H_u(r)}\leq c_\mu R^{2+2\sigma+2\kappa'}
r^{2\sigma}.
$$
This completes the proof.
\end{proof}

We will also need the following well-known inequality (see \cite{Gro})
and one of its corollaries.

\begin{lemma}[$\log$-Sobolev inequality]
  \pushQED{\qed} For any $f\in W^{1}_2(\R^n, G(\cdot,s))$ one has
  \[
  \int_{\R^n} f^2\log(f^2) G(\cdot,s)\leq \Big(\int_{\R^n} f^2
  G(\cdot,s)\Big)\log\Big(\int_{\R^n} f^2
  G(\cdot,s)\Big)+4|s|\int_{\R^n} |\nabla f|^2 G(\cdot,s).\qedhere
  \]
  \popQED
\end{lemma}
\begin{lemma}\label{lem:cor-log-Sob} For any $f\in W^{1}_2(\R^n,
  G(\cdot,s))$, let $\omega=\{|f|>0\}$. Then,
$$
\log\frac{1}{|\omega|_s}\int_{\R^n} f^2 G(\cdot,s)\leq 2|s|\int_{\R^n}
|\nabla f|^2 G(\cdot,s),
$$
where
$$
|\omega|_s=\int_{\omega} G(\cdot,s).
$$
\end{lemma}
\begin{proof} Let $\psi(y)=y\log y$ for $y>0$ and $\psi(0)=0$. Then,
  the $\log$-Sobolev inequality can be rewritten as
  $$
  \int_{\R^n}\psi(f^2) G(\cdot,s)\leq \psi\bigg(\int_{\R^n}f^2
  G(\cdot,s)\bigg)+2|s|\int_{\R^n} |\nabla f|^2 G(\cdot,s).
  $$
  On the other hand, since $\psi$ is convex on $[0,\infty)$, by
  Jensen's inequality we have
  $$
  \frac1{|\omega|_s}\int_{\R^n}\psi(f^2) G(\cdot,s)\geq
  \psi\bigg(\frac1{|\omega|_s}\int_{\R^n}f^2 G(\cdot,s)\bigg).
  $$
  Combining these inequalities and using the identity
  $\lambda\,\psi\left(\dfrac{a}{\lambda}\right)-\psi(a)=a\,\log\dfrac1\lambda$,
  we arrive at the claimed inequality.
\end{proof}

\begin{proof}[Proof of Theorem~\ref{thm:exist-homogen-blowups}]

  i) From Lemmas~\ref{lem:Hu-est} and \ref{lem:fr-est} as well as
  Lemma~\ref{lem:w212}, for $R\geq 1$, $0<r<r_{R,u}$ we have
$$
\int_{S_{R/2}^+}(u_r^2+|t||\nabla u_r|^2+|t|^2|
D^2u_r|^2+|t|^2|\partial_t u_r|^2) G\leq C_R(1+c_\mu r^{2\sigma}).
$$
Since $R\geq 1$ is arbitrary, this implies the claim of part i).

ii) Note that, in view of Lemma~\ref{lem:w112-diff}, it will be enough
to show the existence of $u_0$, and the convergence
$$
\int_{S_R^+}|u_{r_j}-u_{0}|^2 G\to 0.
$$
From Lemma~\ref{lem:Hu-mu} it follows that $\kappa\geq 0$,\footnote{We
  will prove later that $\kappa\geq 3/2$, but the information
  $\kappa\geq 0$ will suffice in this proof.} and therefore we obtain
$$
r\frac{H_u'(r)}{H_u(r)}\geq-1,\quad 0<r<r_0.
$$
Integrating, we obtain that for small $\delta>0$
$$
H_u(r\delta)\leq H_u(r)\delta^{-1},\quad
$$
which gives
$$
H_{u_r}(\delta)\leq \delta^{-1}
$$
and consequently
$$
\int_{S_\delta^+} u_r^2 G<\delta,\quad 0<r<r_0.
$$
Next, let $\zeta_A\in C^\infty_0(\R^n)$ be a cutoff function, such
that
$$
0\leq\zeta\leq 1,\quad \zeta_A=1\quad\text{on }B_{A-1},\quad \supp
\zeta\subset B_{A}.
$$
We may take $A$ so large that $\int_{\R^n\setminus B_{A-1}}
G(x,t)dx<e^{-1/\delta}$ for $-R^2< t<0$. Then from
Lemma~\ref{lem:cor-log-Sob}, we have that
$$
\int_{S_R^+\cap\{|x|\geq A\}}u_r^2 G \leq
\int_{S_R^+}u_r^2(1-\zeta_A)^2 G\leq \delta
\int_{S_R^+}(u_r^2+|t||\nabla u_r|^2) G\leq \delta C(R),
$$
for small enough $r$, where in the last step we have used the uniform
estimate from part i).

Next, notice that on $E=E_{R,\delta,A}=(S_R^+\setminus
S_\delta^+)\cap\{|x|\leq A\}=B_A\times(-R^2,-\delta^2]$ the function $
G$ is bounded below and above and therefore the estimates in i) imply
that the family $\{u_r\}_{0<r<r_{R,u}}$ is uniformly bounded in
$W^{1,1}_2(E^\circ)$ and thus we can extract a subsequence $u_{r_j}$
converging strongly in $L_2(E)$ and consequently in $L_2(E,
G)$. Letting $\delta\to 0$ and $A\to \infty$, combined with the
estimates above, by means of the Cantor diagonal method, we complete
the proof of this part.

iii) We first start with the Signorini boundary conditions. From the
estimates in ii) we have that $\{u_r\}$ is uniformly bounded in
$W^{2,1}_2(B_R^+\times(-R^2,-\delta^2])$ for any $0<\delta<R$. We thus
obtain that $u_{r_j}\to u_0$ strongly, and $\partial_{x_n}
u_{r_j}\to \partial_{x_n} u_0$ weakly in
$L_2(B_R'\times(-R^2,-\delta^2])$. This is enough to pass to the limit
in the Signorini boundary conditions and to conclude that
$$
u_0\geq 0,\quad -\partial_{x_n}u_0\geq 0,\quad
u_0\partial_{x_n}u_0=0\quad\text{on }\R^{n-1}\times(-\infty,0).
$$
Besides, arguing similarly, and using Lemma~\ref{lem:fr-est} we obtain
that
$$
\Delta u_0-\partial_t u_0=0\quad\text{in }\R^n_+\times(-\infty,0).
$$
Thus, to finish the proof of this part it remains to show that $u_0$
is in the unweighted Sobolev class $W^{1,1}_2(Q_R^+)$ for any
$R>0$. Because of the scaling properties, it is sufficient to prove it
only for $R=1/8$. We argue as follows. First, extend $u_0$ by even
symmetry in $x_n$ to $\R^n\times(-\infty,0)$. We then claim that
$u_0^\pm$ are subcaloric functions in $\R^n\times(-\infty,0)$. Indeed,
this would follow immediately, if we knew the continuity of $u_0$,
since $u_0$ is caloric where nonzero. But since we do not know the
continuity of $u_0$ at this stage, we argue as follows. By continuity
of $u_r$ we easily obtain
$$
(\Delta -\partial_t) u_r^\pm\geq -f_r^\mp\quad\text{in }
B_R\times(-R^2,-\delta^2].
$$
Then, passing to the limit as $r=r_j\to 0$, we conclude that $u_0^\pm$
are subcaloric, since $|f_r|\to 0$ in $L_2(B_R\times(-R^2,-\delta^2])$
by Lemma~\ref{lem:fr-est}.  Further, we claim that $u_0^\pm$ satisfy
the sub mean-value property
$$
u_0^\pm(x,t)\leq\int_{\R^n} u_0^\pm(y,-1) G(x-y,-t-1)dy,
$$
for any $(x,t)\in\R^n\times(-1,0)$. The proof of this fact is fairly
standard, since, by the estimates in part i), $u_0$ satisfies an
integral Tychonoff-type condition in the strips $S_1\setminus
S_\delta$, $\delta>0$. Nevertheless, for completeness we give the
details below. For large $R>0$ let $\zeta_R\in C^\infty_0(\R^n)$ be a
cutoff function such that $0\leq \zeta_R\leq 1$, $\zeta_R=1$ on $B_R$,
$\supp\zeta_R\subset B_{R+1}$, $|\nabla \zeta_R|\leq 1$. Let now
$w=u_0^\pm\zeta_R$ in $\R^n\times(-1,0)$. From the fact that $u_0^\pm$
are subcaloric, we have that
$$
(\Delta-\partial_t)w\geq 2\nabla u_0^\pm\nabla \zeta_R.
$$
The advantage of $w$ now is that it has a bounded support, and
therefore we can write
\begin{multline*}
  u_0(x,t)^\pm\zeta_R(x)=w(x,t)\leq\int_{\R^n}u_0^\pm(y,-1)\zeta_R(y)
  G(x-y,-t-1)\\+2\int_{-1}^t\int_{\R^n}|\nabla u_0(y,s)||\nabla
  \zeta_R(y,s)| G(x-y, s-t)dyds.
\end{multline*}
To proceed, fix $A>0$ large and $a>0$ small and consider $|x|\leq A$
and $-1<t<-a$. We want to show that the second integral above will
vanish as we let $R\to \infty$. This will be done with suitable
estimates on the kernel $G$.

\begin{claim}\label{clm:rho-est} Let  $|x|\leq A$,
  $-1<s<-a<0$, and $s<t<0$. Then
  \begin{equation*}
    G(x-y,s-t)\leq 
    \begin{cases} C G(y,s),&\quad \text{if } t-s<-s/8,\ |y|\geq R\\
      C G(y,s) e^{C|y|},&\quad \text{if } t-s\geq-s/8
    \end{cases}
  \end{equation*}
  with $C=C_{n,a,A}$, $R=R_{n,a,A}$.
\end{claim}
\begin{proof}
  \setcounter{step}{0} \step{1-G-est} $t-s<-s/8$. Choose $R=2A+1$ and
  let $|y|\geq R$. Then
$$
|x-y|^2\geq \frac{|x-y|^2}{2}+\frac{|x-y|^2}{2}\geq
\frac{|R-A|^2}{2}+\frac{(|y|/2)^2}{2}\geq \frac12+\frac{|y|^2}8
$$
and therefore
\begin{align*}
  G(x-y,s-t)&=\frac{1}{(4\pi(t-s))^{n/2}}e^{-\frac{|x-y|^2}{4(t-s)}}\leq
  \frac{1}{(4\pi(t-s))^{n/2}}e^{-\frac{1}{8(t-s)}}e^{-\frac{|y|^2}{32(t-s)}}\\
  &\leq
  \frac{1}{(4\pi(t-s))^{n/2}}e^{-\frac{1}{8(t-s)}}e^{\frac{|y|^2}{4s}}\leq
  C_ne^{\frac{|y|^2}{4s}}\\
  &\leq \frac{C_{n,a}}{(4\pi(-s))^{n/2}}e^{\frac{|y|^2}{4s}}\leq
  C_{n,a} G(y,s),
\end{align*}
where we have used that the function $r\mapsto 1/(4\pi r)^{n/2}
e^{-1/4r}$ is uniformly bounded on $(0,\infty)$.

\step{2-G-est} Suppose now $t-s\geq-s/8$. Then
\begin{align*}
  G(x-y,s-t)&=\frac{1}{(4\pi(t-s))^{n/2}}e^{-\frac{|x-y|^2}{4(t-s)}}\leq
  \frac{8^{n/2}}{(4\pi(-s))^{n/2}}e^{-\frac{|x-y|^2}{4s}}\\
  &\leq
  \frac{C_n}{(4\pi(-s))^{n/2}}e^{-\frac{(|y|-|A|)^2}{4s}}\\
  &\leq
  \frac{C_n}{(4\pi(-s))^{n/2}}e^{-\frac{|y|^2}{4s}}e^{C_{a,A}|y|}\leq
  C_n G(y,s)e^{C_{a,A}|y|}.
\end{align*}
Hence, the claim follows.
\end{proof}

Using the claim, the facts that $\int_{S_1}|\nabla u_0|^2 G<\infty$
and $\int_{S_1}e^{C|y|} G(y,s)dyds<\infty$, and letting $R\to \infty$,
we then easily obtain
$$
u_0^\pm(x,t)\leq\int_{\R^n} u_0^\pm(y,-1) G(x-y,-t-1)dy.
$$
Then, using the second estimate in Claim~\ref{clm:rho-est}, for
$(x,t)\in Q_{1/2}$ we obtain
$$
|u_0(x,t)|\leq C_n\int_{\R^n}|u_0(y,-1)| G(y,-1)e^{C_n|y|}dy.
$$
More generally, changing the initial point $s=-1$ to arbitrary point
$s\in(-1,-1/2]$ we will have
$$
|u_0(x,t)|\leq C_n\int_{S_1\setminus S_{1/2}}|u_0(y,s)|
G(y,s)e^{C_n|y|}dyds
$$
and by applying Cauchy-Schwarz
$$
\|u_0\|_{L_\infty(Q_{1/2})}^2\leq C_n\int_{S_1}u_0^2 G=C_n
H_{u_0}(1)<\infty.
$$
The energy inequality applied to $u_0^\pm$ then yields $u_0\in
W^{1,0}(Q_{1/4})$. Further, applying the estimate in
Lemma~\ref{lem:known-W22} for $u(x,t)=u_0(x,t)\zeta(x,t)$, where
$\zeta$ is a smooth cutoff function in $Q_{1/4}$, equal to $1$ on
$Q_{1/8}$, with $\phi_0=0$ and $f=2\nabla u_0\nabla
\zeta+u_0(\Delta\zeta-\partial_t\zeta)$, we obtain that $u_0\in
W^{2,1}_2(Q_{1/8}^+)$. As remarked earlier, the scaling properties
imply that $u_0\in W^{2,1}_2(Q_{R}^+)$ for any $R>0$.

iv) Finally, we show that $u_0$ is parabolically homogeneous of degree
$\kappa$. Let $r_j\to 0+$ be such that $u_{r_j}\to u_0$ as in
ii). Then by part ii) again we have for any $0<\rho<1$
$$
H_{u_{r_j}}(\rho)\to H_{u_0}(\rho),\quad I_{u_{r_j}}(\rho)\to
I_{u_0}(\rho).
$$
Moreover, since by Lemma~\ref{lem:Hu-est} $H_{u_r}(\rho)\geq
\rho^{2\kappa'}$ for sufficiently small $r$, we also have
$$
H_{u_0}(\rho)\geq \rho^{2\kappa'},\quad 0<\rho<1.
$$
Hence, we obtain that for any $0<\rho<1$
$$
2\frac{I_{u_0}(\rho)}{H_{u_0}(\rho)}=2\lim_{j\to
  \infty}\frac{I_{u_{r_j}}(\rho)}{H_{u_{r_j}}(\rho)}=2\lim_{j\to
  \infty}\frac{I_{u}(r_j\rho)}{H_u(r_j\rho)}=\kappa,
$$
by Lemma~\ref{lem:Hu-mu}. Thus the ratio
$2I_{u_0}(\rho)/H_{u_0}(\rho)$ is constant in the interval
$(0,1)$. Further, notice that passing to the limit in the
differentiation formulas in Proposition~\ref{prop:diff-form-signor},
we will obtain the similar formulas hold for $u_0$ for any
$0<r<\infty$. Thus, from computations in step \stepref{1-mon-form} in
Theorem~\ref{thm:thin-monotonicity}, before the application of the
Cauchy-Schwarz inequality, we have
$$
\frac{d}{dr}\Big(\frac{I_{u_0}(r)}{H_{u_0}(r)}\Big)\geq \frac{1}{r^5
  H_{u_0}(r)}\Big[\int_{S_r^+} u_0^2 G\int_{S_r^+}(Zu_0)^2
G-\Big(\int_{S_r^+} u_0(Zu_0) G\Big)^2\Big].
$$
Note here that $H_{u_0}(r)$ is never zero, since $H_{u_0}(r)\geq
r^{2\kappa'}$ for $r\leq 1$ and $H_{u_0}(r)\geq r^{-2}H_{u_0}(1)\geq
r^{-2}$.  And since we know that the above derivative must be zero it
implies that we have equality in Cauchy-Schwarz inequality
$$
\int_{S_r^+} u_0^2 G\int_{S_r^+}(Zu_0)^2 G=\Big(\int_{S_r^+} u_0(Zu_0)
G\Big)^2,
$$
which can happen only if for some constant $\kappa_0$ we have
$$
Zu_0=\kappa_0 u_0\quad\text{in }S_\infty^+,
$$
or that $u_0$ is parabolically homogeneous of degree $\kappa_0$. But
then, in this case it is straightforward to show that
\begin{align*}
  H_{u_0}(r)&=Cr^{2\kappa_0},\\
  H'_{u_0}(r)&=\frac{4}{r}I_{u_0}(r)=2\kappa_0 Cr^{2\kappa_0-1},
\end{align*}
and therefore
$$
2\frac{I_{u_0}(r)}{H_{u_0}(r)}=\kappa_0.
$$
This implies that $\kappa_0=\kappa$ and completes the proof of the
theorem.
\end{proof}

\section{Homogeneous global solutions}
\label{sec:homog-glob-solut}

In this section we study the homogeneous global solutions of the
parabolic Signorini problem, which appear as the result of the blowup
process described in Theorem~\ref{thm:exist-homogen-blowups}. One of
the conclusions of this section is that the homogeneity $\kappa$ of
the blowup is
$$
\text{either}\quad\kappa=\frac32\quad\text{or}\quad \kappa\geq 2,
$$
see Theorem~\ref{thm:min-homogen} below.  This will have two important
consequences: (i) the fact that $\kappa\geq 3/2$ will imply the
optimal $H^{3/2,3/4}_{\loc}$ regularity of solutions (see
Theorem~\ref{thm:opt-reg}) and (ii) the ``gap'' $(3/2,2)$ between
possible values of $\kappa$ will imply the relative openness of the
so-called regular set (see Proposition~\ref{prop:reg-set-rel-open}).

\medskip We start by noticing that $\kappa>1$.

\begin{proposition}\label{prop:kappa>1alpha} Let $u\in \S^f(S_1^+)$ be as in
  Theorem~\ref{thm:exist-homogen-blowups}. Then, $\kappa\geq
  1+\alpha$, where $\alpha$ is the H\"older exponent of $\nabla u$ in
  Definition~\ref{def:Sf}.
\end{proposition}

For the proof we will need the following fact.

\begin{lemma}\label{lem:Holder-Hu-above} Let $u\in\S^f(S_1^+)$. Then,
$$
H_u(r)\leq C_ur^{2(1+\alpha)},\quad 0<r<1,
$$
where $\alpha$ is the H\"older exponent of $\nabla u$ in
Definition~\ref{def:Sf}.
\end{lemma}
\begin{proof} Since $(0,0)\in \Gamma_*(u)$, we must have $|\nabla
  u(0,0)|=0$. Recalling also that $u$ has a bounded support, we obtain
  that
$$
|\nabla u|\leq C_0\|(x,t)\|^{\alpha},\quad (x,t)\in S_1^+.
$$
Let us show that for $C>0$
$$
|u|\leq C\|(x,t)\|^{1+\alpha},\quad (x,t)\in S_1^+.
$$
Because of the gradient estimate above, it will be enough to show that
$$
|u|\leq C r^{1+\alpha}\quad\text{in }Q_r^+.
$$
First, observe that since $u\geq 0$ on $Q_r'$, we readily have
$$
u\geq -C_0 r^{1+\alpha}\quad\text{in }Q_r^+.
$$
To show the estimate from above, it will be enough to establish that
$$
u(0,-r^2)\leq C_1 r^{1+\alpha}.
$$
Note that since $u$ is bounded, it is enough to show this bound for
$0<r<1/2$.  Assuming the contrary, let $r\in (0,1/2 )$ be such that
$u(0,-r^2)\geq C_1 r^\alpha$ with large enough $C_1$. Then, from the
bound on the gradient, we have $u\geq (C_1-C_0)r^{1+\alpha}$ on
$B_r\times\{-r^2\}$. In particular,
$$
u(\tfrac12r e_n,-r^2)\geq (C_1-C_0)r^{1+\alpha}.
$$
Also, let $M$ be such that $|f(x,t)|\leq M$ in $S_1^+$. Then, consider
the function
$$
\tilde u(x,t)=u(x,t)+C_0(2r)^{1+\alpha}.
$$
We will have
$$
\tilde u\geq 0,\quad |(\Delta -\partial_t) \tilde u|\leq
M\quad\text{in } Q_{2r}^+.
$$
Besides,
$$
\tilde u(\tfrac 12 r e_n,-r^2)\geq C_1r^{1+\alpha}.
$$
Then, from the parabolic Harnack inequality (see e.g.\
\cite{Lie}*{Theorems~6.17--6.18})
$$
\tilde u(\tfrac12 e_n,0)\geq C_nC_1 r^{1+\alpha}- Mr^2,
$$
or equivalently,
$$
u((\tfrac12 e_n,0)\geq (C_nC_1-C_0 2^{1+\alpha})r^{1+\alpha}-Mr^2.
$$
But then from the bound on the gradient we will have
$$
u(0,0)\geq (C_nC_1-C_02^{1+\alpha}-C_0)r^{1+\alpha}-Mr^2>0,
$$
if $C_1$ is sufficiently large, a contradiction. This implies the
claimed estimate
$$
|u(x,t)|\leq C\|(x,t)\|^{1+\alpha}.
$$
The estimate for $H_u(r)$ is then a simple corollary:
\begin{align*}
  H_u(r)&\leq \frac{C}{r^2}\int_{S_r^+}(|x|^2+|t|)^{1+\alpha} G(x,t)dxdt\\
  &=\frac{C}{r^2}\int_0^{r^2}\int_{\R^n_+}s^{1+\alpha}(y^2+1) G(s^{1/2}y,-s)s^{n/2}dyds\\
  &=\frac{C}{r^2}\int_0^{r^2}\int_{\R^n_+}s^{1+\alpha}(y^2+1) G(y,-1)dyds\\
  &=Cr^{2(1+\alpha)}.\qedhere
\end{align*}
\end{proof}

The proof of Proposition~\ref{prop:kappa>1alpha} now follows easily.

\begin{proof}[Proof of Proposition~\ref{prop:kappa>1alpha}] Let
  $\kappa'\in (\kappa, \kappa_\mu)$ be arbitrary. Then, by
  Lemma~\ref{lem:Hu-est} we have
$$
H_u(r)\geq c_u\, r^{2\kappa'},\quad 0<r<r_u.
$$
On the other hand, by Lemma~\ref{lem:Holder-Hu-above} (proved below)
we have the estimate
$$
H_{u}(r)\leq C_u r^{2(1+\alpha)},\quad 0<r<1.
$$
Hence, $\kappa'\geq 1+\alpha$. Since this is true for any $\kappa'\in
(\kappa, \kappa_\mu)$, we obtain that $\kappa\geq 1+\alpha$, which is
the sought for conclusion.
\end{proof}

We will also need the following technical fact.

\begin{lemma} Let $u_0$ be a nonzero $\kappa$-parabolically
  homogeneous solution of the Signorini problem in $S_\infty^+$, as in
  Theorem~\ref{thm:exist-homogen-blowups}(iii). Then, $\nabla u_0\in
  H^{\alpha,\alpha/2}_\loc((\R^n_+\cup\R^{n-1})\times(-\infty,0))$ for
  some $0<\alpha<1$.
\end{lemma}
\begin{remark} Note that this does not follow from
  Lemmas~\ref{lem:known-Halpha} or \ref{lem:known-Halpha-2} directly,
  since they rely on $W^{2}_\infty$-regularity of $\phi_0$ (which is
  given by the function $u_0$ itself), or $W^{1,0}_\infty$-regularity
  of $u_0$, which has to be properly justified.
\end{remark}

\begin{proof}
  Note that because of the homogeneity, it is enough to show that
  $g(x):=u_0(x,-1)$ is in
  $H^{1+\alpha}_\loc(\R^n_+\cup\R^{n-1})$. Indeed, $g\in
  W^{2}_{2}(B_R^+)$ for any $R>0$ and since $x\nabla u_0+2t\partial_t
  u_0=\kappa u_0$, we obtain that $g$ solves the Signorini problem
  \begin{align*}
    \Delta g-\frac12x\nabla g+\frac{\kappa}{2}g=0&\quad\text{in }B_R^+,\\
    g\geq 0,\quad -\partial_{x_n}g\geq 0,\quad
    g\partial_{x_n}g=0&\quad\text{on }B_R'.
  \end{align*}
  But now, the known results for the elliptic Signorini problem for
  operators with variable coefficients (see e.g.\ \cite{AU}) imply
  that $g\in H^{1+\alpha}_\loc(\R^n_+\cup \R^{n-1})$, as claimed.
\end{proof}

\begin{proposition}[Homogeneous global solutions of homogeneity
  $1<\kappa<2$]
  \label{prop:homogen-glob-sol-1-2}
  Let $u_0$ be a nonzero $\kappa$-parabolically homogeneous solution
  of the Signorini problem in $S_\infty^+=\R^n_+\times(-\infty,0]$
  with $1<\kappa<2$. Then, $\kappa=3/2$ and
$$
u_0(x,t)=C\Re(x'\cdot e+ix_n)_+^{3/2}\quad\text{in }S_\infty^+
$$
for some tangential direction $e\in\partial B_1'$.
\end{proposition}

\begin{proof} Extend $u_0$ by even symmetry in $x_n$ to the strip
  $S_\infty$, i.e., by putting
$$
u_0(x',x_n,t)=u_0(x',-x_n,t).
$$
Take any $e\in\partial B_1'$, and consider the positive and negative
parts of the directional derivative $\partial_eu_0$
$$
v_e^\pm=\max\{\pm\partial_e u_0,0\}.
$$
We claim that they satisfy the following conditions
$$
(\Delta-\partial_t)v_e^\pm\geq 0,\quad v_e^\pm\geq0,\quad v_e^+\cdot
v_e^-=0\quad\text{in }S_\infty.
$$
The last two conditions are obvious. The first one follows from the
fact that $v_e^\pm$ are continuous in $\R^n\times(-\infty,0)$ (by
Lemma~\ref{lem:Holder-Hu-above}) and caloric where positive. Hence, we
can apply Caffarelli's monotonicity formula to the pair $v_e^\pm$, see
\cite{Ca1}. Namely, the functional
$$
\phi(r)=\frac{1}{r^4}\int_{S_r} |\nabla v_e^+|^2 G\int_{S_r}|\nabla
v_e^-|^2 G,
$$
is monotone nondecreasing in $r$.  On the other hand, from the
homogeneity of $u$, it is easy to see that
$$
\phi(r)=r^{4(\kappa-2)}\phi(1),\quad r>0.
$$
Since $\kappa<2$, $\phi(r)$ can be monotone increasing if and only if
$\phi(1)=0$ and consequently $\phi(r)=0$ for all $r>0$. If fact, one
has to exclude the possibility that $\phi(1)=\infty$ as well. This can
be seen in two different ways. First, by Remark~\ref{rem:w212-alt},
one has
$$
\int_{S_{1}}|\nabla v_e^\pm|^2 G\leq C_n\int_{S_{2}}(v_e^\pm)^2 G\leq
\int_{S_4} u_0^2 G.
$$
Alternatively, from Theorem~\ref{thm:exist-homogen-blowups} (i) and
(iv) it follows that
$$
\int_{\R^n} |\nabla v_e^\pm|^2 G(\cdot,-1)dx=j^\pm<\infty.
$$
Then,
\begin{align*}
  \int_{S_1} |\nabla v_e^\pm|^2 G(x,t)dxdt&=\int_{S_1}|\nabla
  v_e^\pm(|t|^{1/2}y,t)|^2 G(|t|^{1/2}y,t)|t|^{n/2}dydt\\
  &=j^\pm\int_{-1}^0 |t|^{(\kappa-2)} dt<\infty,
\end{align*}
since $\kappa>1$.

From here it follows that one of the functions $v_e^\pm$ is
identically zero, which is equivalent to $\partial_e u_0$ being ether
nonnegative or nonpositive on the entire
$\R^n\times(-\infty,0]$. Since this is true for any tangential
direction $e\in\partial B_1$, it thus follows that $u_0$ depends only
on one tangential direction, and is monotone in that
direction. Therefore, without loss of generality we may assume that
$n=2$ and that the coincidence set at $t=-1$ is an infinite interval
$$
\Lambda_{-1}=\{(x',0)\in\R^2\mid u_0(x',0,-1)=0\}=(-\infty,
a]\times\{0\}=:\Sigma_a^-.
$$
Besides, repeating the monotonicity formula argument above for the
pair of functions $\max\{\pm w,0\}$, where
$$
w(x,t)=\begin{cases}-\partial_{x_2} u_0(x_1,x_2,t),& x_2\geq
  0\\\partial_{x_2}u_0(x_1,x_2,t),& x_2<0
\end{cases}
$$
is a caloric function in $\R^2\times(-\infty,0]\setminus\Lambda$,
parabolically homogeneous of degree $\kappa-1$, we obtain that also
$w$ does not change sign in $\R^2\times(-\infty,0]$. Noting also that
$w\geq 0$ on $\R^\times\{0\}\times (-\infty,0]$, we see that $w\geq 0$
everywhere(unless $w=0$ identically on $\R^{n-1}\times(-\infty,0]$, in
which case $w$ is a polynomial of degree $\kappa-1$, which is
impossible, since $\kappa$ is noninteger). Hence, we get
$$
\partial_{x_1}u_0\geq 0,\quad-\partial_{x_2} u_0(x_1, x_2,t)\geq
0\quad\text{in } \R^2_+\times(-\infty, 0].
$$
Further if $g_1(x)=\partial_{x_1}u_0(x,-1)$ and using that
$\partial_{x_1}u_0(x,t)$ is caloric we obtain that
$$
g_1=0\quad\text{on }\Sigma_a^-,\quad-\Delta g_1+\frac12x\nabla
g_1=\frac{\kappa-1}{2} g_1\quad\text{in }\R^2\setminus\Sigma_a^-.
$$
Since also $g_1$ is nonnegative and not identically zero, then $g_1$
is the ground state for the Ornstein-Uhlenbeck operator in
$\R^2\setminus\Sigma_a^-$ and
$$
\frac{\kappa-1}2=\lambda(\Sigma_a^-)=\inf_{v|_{\Sigma_a^-}=0}\frac{\int_{\R^2}|\nabla
  v|^2 e^{-\frac{x^2}{4}}dx}{\int_{\R^2}v^2 e^{-\frac{x^2}{4}}dx}.
$$
On the other hand, let $g_2(x)=-\partial_{x_2}u_0(x_1,x_2,-1)$ in
$x_2\geq 0$ and $g_2(x)=\partial_{x_2}u_0(x_1,x_2,-1)$ for
$x_2<0$. Then, we have
$$
g_2=0\quad\text{on }\Sigma_a^+,\quad-\Delta g_2+\frac12x\nabla
g_2=\frac{\kappa-1}{2} g_2\quad\text{in }\R^2\setminus\Sigma_a^+.
$$
with $\Sigma_a^+:=[a,\infty)\times\{0\}$.
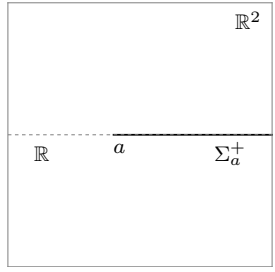
\begin{figure}[t]
  \begin{picture}(100,100)
    \put(0,0){\includegraphics[height=100pt]{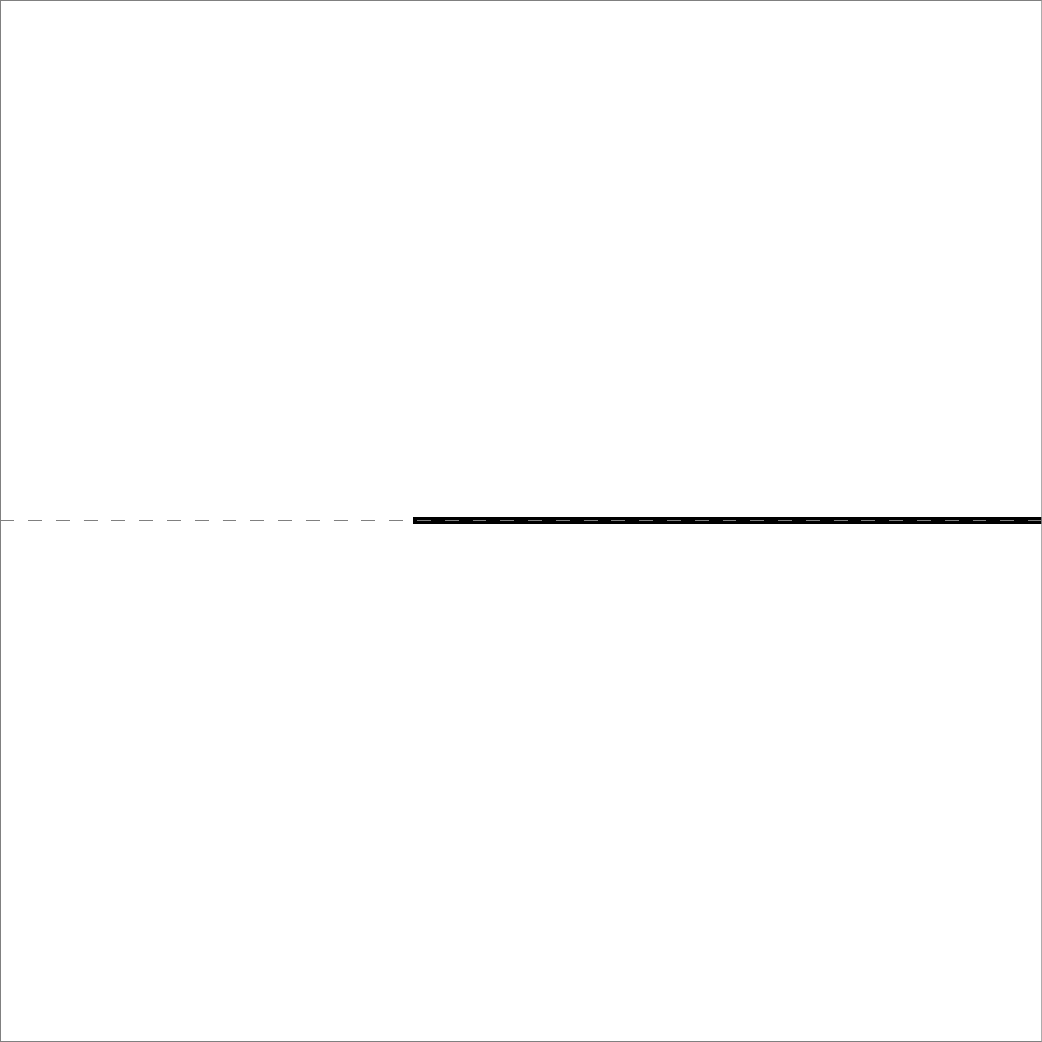}}
    \put(40,43){\footnotesize $a$} \put(78,40){\footnotesize
      $\Sigma_a^+$} \put(10,40){\footnotesize $\R$}
    \put(86,90){\footnotesize $\R^2$}
  \end{picture}
  \caption{The slit domain $\R^2\setminus \Sigma_a^+$,
    $\Sigma_a^+=[a,\infty)\times\{0\}$.}
  \label{fig:slit-dom}
\end{figure}
Thus, this time $g_2$ is the ground state for the Ornstein-Uhlenbeck
operator in $\R^2\setminus\Sigma_a^+$, and therefore
$$
\frac{\kappa-1}2=\lambda(\Sigma_a^+)=\inf_{v|_{\Sigma_a^+}=0}\frac{\int_{\R^2}|\nabla
  v|^2 e^{-\frac{x^2}{4}}dx}{\int_{\R^2}v^2 e^{-\frac{x^2}{4}}dx}.
$$
Observe now that $\lambda(\Sigma_a^+)=\lambda(\Sigma_{-a}^-)$ and
therefore from the above equations we have
$$
\lambda(\Sigma_{a}^-)=\lambda(\Sigma_{-a}^-).
$$
On the other hand, it is easy to see that the function $a\mapsto
\lambda(\Sigma_{a}^-)$ is strictly monotone and therefore the above
equality can hold only if $a=0$. In particular,
$$
\kappa=1+2\lambda(\Sigma_0^-).
$$
We now claim that $\lambda(\Sigma_0^-)=1/4$. Indeed, consider the
function
$$
v(x)=\Re(x_1+i|x_2|)^{1/2},
$$
which is harmonic and homogeneous of degree $1/2$:
$$
\Delta v=0,\quad x\nabla v-\frac12 v=0.
$$
Therefore,
$$
v=0\quad\text{on }\Sigma_0^-,\quad-\Delta v+\frac12 x\nabla v=\frac14
v\quad\text{in }\R^2\setminus\Sigma_0^-.
$$
Also, since $v$ is nonnegative, we obtain that $v$ is the ground state
of the Ornstein-Uhlenbeck operator in $\R^2\setminus\Sigma_0^-$. This
implies in particular that $\lambda(\Sigma_0^-)=\frac14$, and
consequently
$$
\kappa=3/2.
$$
Moreover, $g_1(x)=\partial_{x_1}u_0(x,-1)$ must be a multiple of the
function $v$ above and from homogeneity we obtain that
\begin{align*}
  \partial_{x_1}u_0(x,t)&=(-t)^{1/4}g_1(x/(-t)^{1/2})=C(-t)^{1/4}\Re((x_1+i
  |x_2|)/(-t)^{1/2})^{1/2}\\&=C\Re(x_1+i |x_2|)^{1/2}.
\end{align*}
From here it is now easy to see that necessarily
\[
u_0(x,t)=C\Re(x_1+i |x_2|)^{3/2}.\qedhere
\]
\end{proof}
Combining Propositions~\ref{prop:kappa>1alpha} and
\ref{prop:homogen-glob-sol-1-2} we obtain the following result.

\begin{theorem}[Minimal homogeneity]\label{thm:min-homogen}
  \pushQED{\qed} Let $u\in \S^f(S_1^+)$ and $\mu$ satisfies the
  conditions of Theorem~\ref{thm:thin-monotonicity}. Assume also
  $\kappa_\mu=\frac12\lim_{r\to 0+}r\mu'(r)/\mu(r)\geq 3/2$. Then
$$
\kappa:=\Phi_u(0+)\geq 3/2.
$$ 
More precisely, we must have
\[
\text{either}\quad \kappa=3/2\quad\text{or}\quad \kappa\geq 2.\qedhere
\]
\popQED
\end{theorem}

\begin{remark} Very little is known about the possible values of
  $\kappa$. However, we can say that the following values of $\kappa$
  do occur:
$$
\kappa=2m-1/2,\ 2m,\ 2m+1,\quad m\in\N.
$$
This can be seen from the following explicit examples of homogeneous
solutions (that are actually $t$-independent)
$$
\Re(x_1+ix_n)^{2m-1/2},\quad \Re(x_1+ix_n)^{2m},\quad
-\Im(x_1+ix_n)^{2m+1}.
$$
It is known (and easily proved) that in $t$-independent case and
dimension $n=2$, the above listed values of $\kappa$ are the only ones
possible. In all other cases, finding the set of possible values of
$\kappa$ is, to the best of our knowledge, an open problem.
\end{remark}
\section{Optimal regularity of solutions}
\label{sec:optim-regul-solut}

Using the tools developed in the previous sections we are now ready to
prove the optimal regularity of solutions of the parabolic Signorini
problem with sufficiently smooth obstacles. In fact, we will establish
our result for a slightly more general class of functions solving the
Signorini problem with nonzero obstacle and nonzero right-hand side.

\begin{theorem}[Optimal regularity in the parabolic Signorini problem]
  \label{thm:opt-reg}
  Let $\phi\in H^{2,1}(Q_1^+)$, $f\in L_\infty(Q_1^+)$. Assume that
  $v\in W^{2,1}_2(Q_1^+)$ be such that $\nabla v\in
  H^{\alpha,\alpha/2}(Q_1^+\cup Q_1')$ for some $0<\alpha<1$, and
  satisfy
  \begin{gather}
    \Delta v-\partial_t v = f\quad\text{in } Q_1^+,\\
    v-\phi\geq 0,\quad -\partial_{x_n} v\geq
    0,\quad(v-\phi)\partial_{x_n} v = 0\quad\text{on } Q'_1.
  \end{gather}
  Then, $v\in H^{3/2,3/4}(Q_{1/2}^+\cup Q_{1/2}')$ with
$$
\|v\|_ {H^{3/2,3/4}(Q_{1/2}^+\cup Q_{1/2}')}\leq
C_n\left(\|v\|_{W^{1,0}_\infty(Q_1^+)}+\|f\|_{L_\infty(Q_1^+)}+\|\phi\|_{H^{2,1}(Q_1')}\right).
$$
\end{theorem}

The proof of Theorem \ref{thm:opt-reg} will follow from the interior
parabolic estimates and the growth bound of $u$ away from the free
boundary $\Gamma(v)$.

\begin{lemma}\label{lem:HuMr3} Let $u\in \S^f(S_1^+)$ with
  $\|u\|_{L_\infty(S_1^+)}$, $\|f\|_{L_\infty(S_1^+)}\leq M$. Then,
$$
H_r(u)\leq C_{n}M^2 r^3.
$$
\end{lemma}

\begin{proof} The $L_\infty$ bound on $f$ allows us to apply Theorem
  \ref{thm:thin-monotonicity} with the following specific choice of
  $\mu$ in the generalized frequency function. Indeed, fix
  $\sigma=1/4$ and let $\mu(r)=M^2 r^{4-2\sigma}$. Then,
$$
r^{4-2\sigma}\int_{\R^n_+} f^2(\cdot, -r^2) G(\cdot, -r^2)\leq \mu(r)
$$
and therefore
$$
\Phi_u(r)=\frac12 r e^{Cr^\sigma} \frac{d}{dr}\log \max\{H_u(r),M^2
r^{4-2\sigma}\}+2(e^{C r^\sigma}-1)
$$
is monotone for $C=C_n$. By Theorem~\ref{thm:min-homogen} we have that
$\Phi_u(0+)\geq 3/2$.

Now, for $H_u(r)$ we have two alternatives: either $H_u(r)\leq
\mu(r)=M^2r^{4-2\sigma}$ or $H_u(r)>\mu(r)$. In the first case the
desired estimate is readily satisfied, so we concentrate on the latter
case. Let $(r_0,r_1)$ be a maximal interval in the open set
$\mathcal{O}= \{r\in (0,1)\mid H_u(r)>\mu(r)\}$. Then, for $r\in
(r_0,r_1)$ we have
$$
\Phi_u(r)=\frac12 r e^{Cr^\sigma}
\frac{H'_u(r)}{H_u(r)}+2(e^{Cr^\sigma}-1)\geq \Phi_u(0+)\geq \frac32.
$$
We thus have,
$$
\frac{H'_u(r)}{H_u(r)}\geq \frac{3}{r}(1-Cr^\sigma),\quad r\in
(r_0,r_1),
$$
which, after integration, implies
$$
\log\frac{H_u(r_1)}{H_u(r)}\geq\log\frac{r_1^3}{r^3}-C r_1^\sigma,
$$
and therefore
$$
H_u(r)\leq Cr^3\frac{H_u(r_1)}{r_1^3}.
$$
Now, for $r_1$ we either have $r_1=1$ or $H_u(r_1)=\mu(r_1)$. Note
that $H_u(1)\leq M^2$ from the $L_\infty$ bound on $u$, and thus in
both cases we have $H_u(r_1)\leq M^2 r_1^3$.  We thus have the desired
conclusion
$$
H_u(r)\leq CM^2r^3.
$$
\end{proof}

To apply the results of the previous sections, we will need the
following $L_\infty-L_2$ type estimates.

\begin{lemma}\label{lem:l2linf-lower}
  Let $w$ be a nonnegative function with at most polynomial growth at
  infinity in the strip $S_R$, and such that for some $\gamma> 0$
$$
\Delta w-\partial_t w\geq -M\|(x,t)\|^{\gamma-2}\quad\text{in }S_R.
$$
Then,
$$
\sup_{Q_{r/2}} w\leq C_n H_w(r)^{1/2}+C_{n,\gamma}Mr^{\gamma},\quad
0<r<R.
$$
\end{lemma}
\begin{proof} Choosing a constant $C_{n,\gamma}>0$ we can guarantee
  that
$$
\tilde w(x,t)=w(x,t)+C_{n,\gamma} M (|x|^2-t)^{\gamma/2}
$$
is still nonnegative, has a polynomial growth at infinity and
satisfies
$$
\Delta \tilde w-\partial_t\tilde w\geq 0\quad\text{in }S_R.
$$
Moreover,
$$
H_{\tilde w}(r)^{1/2}\leq
H_w(r)^{1/2}+C_{n,\gamma}M\Big(\frac1{r^2}\int_{S_r}
(|x|^2+|t|)^\gamma G(x,t)dxdt\Big)^{1/2}.
$$
From the scaling properties of $ G$ it is easily seen that
$$\frac1{r^2}\int_{S_r}
(|x|^2+|t|)^\gamma G(x,t)dxdt=C_n r^{2\gamma}
$$
and therefore we may assume that $M=0$ from the beginning.

The rest of the proof is now similar to that of
Theorem~\ref{thm:exist-homogen-blowups}(iii).

Indeed, for $(x,t)\in Q_{r/2}$ and $s\in(-r^2,-r^2/2]$ we have the sub
mean-value property
\begin{align*}
  w(x,t) &\leq\int_{\R^n} w(y,s) G(x-y,t-s)dy.
\end{align*}
This can be proved as in Theorem~\ref{thm:exist-homogen-blowups}(iii)
by combining the fact that $\int_{S_\rho\setminus S_\delta}|\nabla
w|^2 G<\infty$ for any $0<\delta<\rho<R$, with the polynomial growth
assumption on $w$, and the energy inequality. On the other hand, for
our choice of $t$ and $s$, we have $|s|/4<t-s<|s|$.  Thus, arguing as
in Claim~\ref{clm:rho-est}, we have
\begin{align*}
  G(x-y,t-s)&=\frac{1}{(4\pi
    (t-s))^{n/2}}e^{-|x-y|^2/4(t-s)}\\
  &\leq\frac{C_n}{(4\pi
    |s|)^{n/2}} e^{-|x-y|^2/4|s|}\\
  &\leq C_n G(y,s) e^{(x\cdot y)/2|s|}.
\end{align*}
Hence, we obtain that
$$
w(x,t)\leq C_n\int_{\R^n} w(y,s) e^{(x\cdot y)/2|s|} G(y,s)dy.
$$
Now, applying the Cauchy-Schwarz inequality, we will have
\begin{align*}
  w(x,t) &\leq C_n\Big(\int_{\R^n} w(y,s)^2 G(y,s) dy\Big)^{1/2}
  \Big(\int_{\R^n} e^{(x\cdot y)/|s|} G(y,s) dy\Big)^{1/2}\\
  & \leq C_n\Big(\int_{\R^n} w(y,s)^2 G(y,s) dy\Big)^{1/2}
  e^{2|x|^2/|s|}\\
  &\leq C_n\Big(\int_{\R^n} w(y,s)^2 G(y,s) dy\Big)^{1/2},
\end{align*}
where in the last step we have used that $|x|\leq r$ and $|s|\geq
r^2/2$. Integrating over $s\in [-r^2,-r^2/4]$, we obtain
\begin{align*}
  w(x,t) &\leq\frac{C_n}{r^2}\int_{-r^2}^{-r^2/2}\Big(\int_{\R^n}
  w(y,s)^2 G(y,s) dy\Big)^{1/2}ds\\
  &\leq C_n\Big(\frac{1}{r^2}\int_{-r^2}^{-r^2/2}\int_{\R^n}
  w(y,s)^2 G(y,s) dyds\Big)^{1/2}\\
  &\leq C_n H_w(r)^{1/2}.
\end{align*}
This completes the proof.
\end{proof}

\begin{lemma}\label{lem:l2linf-upper} Let $w$ be a nonnegative bounded
  function in $S_R$ satisfying
$$
0\leq w\leq M,\quad \Delta w -\partial_t w\geq -M\quad\text{in }S_R,
$$
and let $(x_0,t_0)\in S_R'$, $r_0>0$ be such that for some constant
$A$
$$
\sup_{Q_r(x_0,t_0)}w\leq A r^{3/2},\quad 0<r<r_0.
$$
Then,
$$
w(x,t)\leq C \|(x-x_0,t-t_0)\|^{3/2},\quad(x,t)\in S_R,
$$
with $C=C_{n,r_0,R}(A+M)$.
\end{lemma}
\begin{proof} Without loss of generality we may assume that
  $x_0=0$. Then, from the assumptions on $w$ we have that
$$
w(x,t_0)\leq C|x|^{3/2},\quad x\in\R^n,
$$
and more generally
$$
w(x,t)\leq C\|(x,t-t_0)\|^{3/2},\quad x\in\R^n,\ t\in (-R^2, t_0],
$$
where $C=C_{n,r_0,R}(A+M)$. To propagate the estimate to $t\in
(t_0,0)$ we use the sub-mean value property. Namely, for $x\in\R^n$
and $t_0<t<0$ we have
\begin{align*}
  w(x,t)&\leq \int_{\R^n}w(y,t_0)  G(x-y,t-t_0)dy+M(t-t_0)\\
  &\leq C\int_{\R^n} |y|^{3/2} G(x-y,t-t_0)dy+M(t-t_0).
\end{align*}
To obtain the desired conclusion, we notice that the integral
$$
V(x,s)=\int_{\R^n} |y|^{3/2} G(x-y,s)dy,\quad x\in\R^n, s>0
$$
is parabolically homogeneous of degree $3/2$ in the sense that
$$
V(\lambda x, \lambda^2s)=\lambda^{3/2}V(x,s),\quad \lambda>0
$$
and therefore we immediately obtain that
$$
V(x,s)\leq C_n\|(x,s)\|^{3/2},\quad x\in\R^n, \ s>0.
$$
Consequently, this yields
$$
w(x,t)\leq C_nC\|(x,t-t_0)\|^{3/2}+M(t-t_0),\quad x\in\R^n,\
t\in(t_0,0],
$$
which implies the statement of the lemma.
\end{proof}

\begin{lemma}\label{lem:growth-est} Let $u\in \S^f(S_1^+)$ and $(x_0,t_0)\in
  \Gamma_*(u)\cap Q_{3/4}'$. Then,
$$
|u(x,t)|\leq C\|(x-x_0,t-t_0)\|^{3/2},\quad (x,t)\in S_1^+,
$$
with
$$
C=C_n\big(\|u\|_{L_\infty(S_1^+)}+\|f\|_{L_\infty(S_1^+)}\big).
$$
\end{lemma}
\begin{proof} This is simply a combination of Lemma~\ref{lem:HuMr3}
  for $u$ and Lemmas~\ref{lem:l2linf-lower}--\ref{lem:l2linf-upper}
  applied for $w=u^\pm$.
\end{proof}

\begin{proof}[Proof of Theorem~\ref{thm:opt-reg}]
  As in Section~\ref{sec:classes-solutions}, let $\psi\in
  C^\infty_0(\R^n)$ be a cutoff function satisfying
  \eqref{eq:psi-1}--\eqref{eq:psi-2}, and consider
$$
u(x,t)=[v(x,t)-\phi(x',t)]\psi(x).
$$
We may also assume that $|\nabla\psi|\leq C_n$. We will thus have
$u\in\S^g(S_1^+)$ with
$$
\|g\|_{L_\infty(S_1^+)}\leq
C_n\left(\|v\|_{W^{1,0}_\infty(Q_1^+)}+\|f\|_{L_\infty(Q_1^+)}+\|\phi\|_{H^{2,1}(Q_1')}\right).
$$
In the remaining part of the proof, $C$ will denote a generic constant
that has the same form as the right-hand side in the above inequality.
For $(x,t)\in Q_{1/2}$, and
$$
Q_r^*(x,t):=\tilde Q_r(x,t)\cap S_\infty=B_r(x)\times
((t-r^2,t+r^2)\cap(-\infty, 0]),
$$ 
let
$$
d=d(x,t)=\sup\{r>0\mid Q^*_r(x,t)\subset Q_1\setminus \Gamma_*(u)\}.
$$
We first claim that
\begin{equation}\label{eq:ud32}
  |u|\leq C d^{3/2}\quad\text{in }Q_{d}^*(x,t).
\end{equation}
Indeed, if $d>1/4$, this follows from boundedness of $u$. If instead
$d\leq 1/4$, then there exist $(x_0,t_0)\in Q_{3/4}'\cap
\Gamma_*(u)\cap \partial Q_{d}^*(x,t)$ and by
Lemma~\ref{lem:growth-est} we have the desired estimate.  Next, we
claim that
\begin{equation}\label{eq:gradud12}
  |\nabla u|\leq C d^{1/2} \quad\text{in }Q_{d/2}^*(x,t)
\end{equation}
This will follow from the interior gradient estimates, applied to the
even or odd extension of $u$ in $x_n$ variable. More specifically,
consider the intersection $Q_{d}^*(x,t)\cap Q_1'$. Since there are no
points of $\Gamma_*(u)$ in this set, we have a dichotomy: either (i)
$u>0$ on $Q_{d}^*(x,t)\cap Q_1'$, or (ii) $u=0$ on $Q_{d}^*(x,t)\cap
Q_1'$.  Accordingly, we define
\begin{align*}
  \tilde u(x',x_n,t)&=\begin{cases} u(x',x_n,t), & x_n\geq 0\\
    u(x',-x_n,t), &x_n\leq 0
  \end{cases}\quad\text{in case (i)},\\
  \tilde u(x',x_n,t)&=\begin{cases} u(x',x_n,t), & x_n\geq 0\\
    -u(x',-x_n,t), &x_n\leq 0\quad\text{in case (ii)}.
  \end{cases}
\end{align*}
In either case $\tilde u$ satisfies a nonhomogeneous heat equation
$$
(\Delta-\partial_t)\tilde u=\tilde g\quad\text{in }Q_{d}^*(x,t),
$$
for an appropriately defined $\tilde g$. The claimed estimate for
$|\nabla u|=|\nabla \tilde u|$ now follows from parabolic interior
gradient estimates, see e.g.\ \cite{LSU}*{Chapter III,
  Theorem~11.1}. Moreover, by \cite{LSU}*{Chapter IV, Theorem~9.1},
one also has that $\tilde u\in W^{2,1}_q(Q_{d/2}^*(x,t))$ for any
$3/2<q<\infty$. To be more precise, we apply the latter theorem to
$\tilde u\zeta$, where $\zeta$ is a cutoff function supported in
$Q^*_d(x,t)$, such $\zeta=1$ on $Q^*_{d/2}(x,t)$, $0\leq \zeta\leq 1$,
$|\nabla\zeta|\leq C_n/d$, $|\partial_t\zeta|\leq C_n/d^2$. From the
estimates on $\tilde u$ and $|\nabla \tilde u|$, we thus have
$$
|(\Delta -\partial_t)(\tilde u\zeta)|\leq C d^{-1/2},
$$
which provides the estimate
$$
\|D^2\tilde u\|_{L_q(Q^*_{d/2}(x,t))}+\|\partial_t \tilde
u\|_{L_q(Q^*_{d/2}(x,t))}\leq C d^{-1/2} d^{(n+2)/q}.
$$
Then, from the Sobolev embedding of $W^{2,1}_q$ into
$H^{2-\frac{n+2}q,1-\frac{n+2}{2q}}$ when $q>n+2$, we obtain the
estimates for H\"older seminorms
$$
\langle \nabla \tilde u\rangle^{(\alpha)}_{Q^*_{d/2}(x,t)}+ \langle
\tilde u\rangle^{((1+\alpha)/2)}_{t,Q^*_{d/2}(x,t)}\leq C
d^{1/2-\alpha},
$$
for any $0<\alpha<1$, see \cite{LSU}*{Chapter II, Lemma~3.3}.  In
particular, we have
\begin{equation}\label{eq:graduholder12}
  \langle \nabla \tilde u\rangle^{(1/2)}_{Q^*_{d/2}(x,t)}+
  \langle \tilde u\rangle^{(3/4)}_{t,Q^*_{d/2}(x,t)}\leq C.
\end{equation}
Now take two points $(x^i,t^i)\in Q_{1/2}^+$, $i=1,2$, and let
$d^i=d(x^i,t^i)$. Without loss of generality we may assume $d_1\geq
d_2$. Let also $\delta=(|x^1-x^2|^2+|t^1-t^2|)^{1/2}$.

Consider two cases:

1) $\delta> \frac12 d_1$. In this case, we have by \eqref{eq:gradud12}
\begin{align*}
  |\nabla u(x^1,t^1)-\nabla u(x^2,t^2)|&\leq |\nabla
  u(x^1,t^1)|+|\nabla u(x^2,t^2)|\\&\leq C (d^1)^{1/2}+C
  (d^2)^{1/2}\leq C \delta^{1/2}.
\end{align*}

2) $\delta< \frac12 d_1$. In this case, both $(x^i,t^i)\in
Q_{d^1/2}^*(x^1,t^1)$, and therefore by \eqref{eq:graduholder12}
\begin{align*}
  |\nabla u(x^1,t^1)-\nabla u(x^2,t^2)|&\leq C \delta^{1/2}
\end{align*}
This gives the desired estimate for the seminorm $\langle \nabla
u\rangle^{1/2}_{Q_{1/2}^+}$. Arguing analogously, we can also prove a
similar estimate for $\langle u\rangle^{(3/4)}_{t,Q_{1/2}^+}$, thus
completing the proof of the theorem.
\end{proof}

\section{Classification of free boundary points}
\label{sec:class-free-bound}

After establishing the optimal regularity of the solutions, we are now
able to undertake the study of the free boundary
$$
\Gamma(v)=\partial\{(x',t)\mid v(x',0,t)>\phi(x',t)\}.
$$
We start with classifying the free boundary points and more generally
points in
$$
\Gamma_*(v)=\{(x',t)\mid
v(x',0,t)=\phi(x',t),\ \partial_{x_n}v(x',0,t)=0\}.
$$
As we will see the higher is the regularity of $\phi$, the finer is
going to be the classification.

Let $v\in \S_\phi(Q_1^+)$ with $\phi\in H^{\ell,\ell/2}(Q_1')$,
$\ell=k+\gamma\geq 2$, $k\in \N$, $0<\gamma\leq1$ and $u_k\in
\S^{f_k}(S_1^+)$ be as constructed in Proposition~\ref{prop:uk-def}.
In particular, $f_k$ satisfies
$$
|f_k(x,t)|\leq M\|(x,t)\|^{\ell-2},\quad (x,t)\in S_1^+.
$$
This implies that
\begin{align*}
  \int_{\R^n_+}f_k(x,-r^2)^2 G(x,-r^2)dx&\leq
  M^2\int_{\R^n_+}(|x|^2+r^2)^{\ell-2} G(x,-r^2)dx\\
  &=M^2r^{2\ell-4}\int_{\R^n_+}(|y|^2+1)^\ell G(y,-1)dy\\
  &=C_\ell M^2 r^{2\ell-4}.
\end{align*}
Thus, if we choose
$$
\mu(r)=r^{2\ell_0},\quad\text{with }k\leq \ell_0<\ell\text{ and
}\sigma\leq\ell-\ell_0,
$$
then $\mu$, $f_k$ and $u_k$ will satisfy the conditions of
Theorem~\ref{thm:thin-monotonicity}. In particular, we will have that
\begin{equation}\label{eq:Philo}
  \Phi_{u_k}^{(\ell_0)}(r) :=  \frac 12 r e^{C
    r^{\sigma}}\frac{d}{dr}\log\max\{H_{u_k}(r),r^{2\ell_0}\}+2
  (e^{C r^{\sigma}}-1)
\end{equation}
is monotone increasing in $r\in(0,1)$ and consequently there exists
the limit (see \eqref{eq:kappa} above)
$$
\kappa=\Phi_{u_k}^{(\ell_0)}(0+).
$$
Recalling the definition \eqref{eq:kappamu} of $\kappa_\mu$, we note
that in the present case we have
$$
\kappa_\mu=\ell_0.
$$
Therefore, by Lemma~\ref{lem:Hu-mu} we infer that
$$
\kappa\leq \ell_0<\ell.
$$ 
Generally speaking, the value of $\kappa$ may depend on the cutoff
function that we have chosen to construct $u_k$. However, as the next
result proves, it is relatively straightforward to check that this is
not the case.

\begin{lemma}\label{lem:kappa-indep-psi} The limit $\kappa=\Phi_{u_k}^{(\ell_0)}(0+)$ does
  not depend on the choice of the cutoff function $\psi$ in the
  definition of $u_k$.
\end{lemma}
\begin{proof}
  Indeed, if we choose a different cutoff function $\psi'$, satisfying
  \eqref{eq:psi-1}--\eqref{eq:psi-2} and denote by $u_k'$ the function
  corresponding to $u_k$ in the construction above, and by $\kappa'$
  the corresponding value as in \eqref{eq:kappa}, then by simply using
  the fact that $u_k=u_k'$ on $B_{1/2}^+\times(-1,0]$ and that $|
  G(x,t)|\leq e^{-c_n/r^2}$ for $|x|\geq 1/2$ and $-r^2<t\leq0$, we
  have
$$
|H_{u_k}(r)-H_{u_k'}(r)|\leq C e^{-c_n/r^2}.
$$
To show now that $\kappa=\kappa'$, we consider several cases.

1) If $\kappa=\kappa'=\ell_0$, then we are done.

2) If $\kappa<\ell_0$, then Lemma~\ref{lem:Hu-mu} implies that
$$
2\kappa-\epsilon \leq r\frac{H_{u_k}'(r)}{H_{u_k}(r)}\leq
2\kappa+\epsilon,\quad 0<r<r_\epsilon.
$$
Integrating these inequalities we obtain
$$
c_\epsilon r^{2\kappa+\epsilon}\leq H_{u_k}(r)\leq C_\epsilon
r^{2\kappa-\epsilon},\quad 0<r<r_\epsilon,
$$
for some (generic) positive constants $c_\epsilon$, $C_\epsilon$. This
will also imply
$$
c_\epsilon r^{2\kappa+\epsilon}\leq H_{u_k'}(r)\leq C_\epsilon
r^{2\kappa-\epsilon},\quad 0<r<r_\epsilon.
$$
Now, if $\epsilon$ is so small that $2\kappa+\epsilon<\ell_0$, we will
have that $H_{u_k'}(r)> \mu(r)=r^{\ell_0}$ for $0<r<r_\epsilon$. But
then, we also have $\kappa'=\lim_{r\to 0}r H'_{u_k'}(r)/H_{u_k'}(r)$,
and therefore
$$
c_\epsilon r^{2\kappa'+\epsilon}\leq H_{u_k'}(r)\leq C_\epsilon
r^{2\kappa'-\epsilon},\quad 0<r<r_\epsilon,
$$
for arbitrarily small $\epsilon>0$. Obviously, the above estimates
imply that $\kappa'=\kappa$.

3) If $\kappa'<\ell_0$, we argue as in 2) above.
\end{proof}

\begin{definition}[Truncated homogeneity] To stress in the above
  construction the dependence only of the function $v$, we will denote
  the quantity $\kappa=\Phi_{u_k}^{(\ell_0)}(0+)$ by
$$
\kappa^{(\ell_0)}_v(0,0).
$$
More generally, for $(x_0,t_0)\in\Gamma_*(v)$ we let
$$
v^{(x_0,t_0)}(x,t):=v(x_0+x,t_0+t),
$$
which translates $(x_0,t_0)$ to the origin. Then, $v^{(x_0,t_0)}\in
\S^{f^{(x_0,t_0)}}(Q_r^+)$ for some small $r>0$. The construction
above has been carried out in $Q_1^+$, rather than $Q_r^+$. However, a
simple rescaling argument generalizes it to any $r>0$. Thus, we can
define
$$
\kappa_v^{(\ell_0)}(x_0,t_0)=\kappa_{v^{(x_0,t_0)}}^{(\ell_0)}(0,0),
$$ 
which we will call the \emph{truncated homogeneity} of $v$ at an
extended free boundary point $(x_0,t_0)$.
\end{definition}

Suppose now for a moment that the thin obstacle $\phi$, that was
assumed to belong to the class $H^{\ell,\ell/2}(Q_1')$, has a higher
regularity. To fix the ideas, suppose $\phi\in H^{\tilde \ell,\tilde
  \ell/2}(Q_1')$, for some $\tilde \ell\geq\ell\geq 2$, with $\tilde
\ell=\tilde k+\tilde \gamma$, $\tilde k\in \N$, $\tilde k\geq k$,
$0<\tilde \gamma\leq 1$. We may thus define
$$
\kappa_v^{(\tilde\ell_0)}(x_0,t_0),
$$
for any $\tilde k\leq \tilde\ell_0<\tilde\ell$. It is natural to ask
about the relation between $\kappa_v^{(\tilde\ell_0)}(x_0,t_0)$ and
$\kappa_v^{(\ell_0)}(x_0,t_0)$. The following proposition provides an
answer to this question.

\begin{proposition}[Consistency of truncated homogeneities]
  \label{prop:kappa-compat}
  If $\ell\leq\tilde\ell$, $\ell_0\leq\tilde\ell_0$ are as above, then
$$
\kappa^{(\ell_0)}_v(x_0,t_0)=\min\{\kappa^{(\tilde\ell_0)}_v(x_0,t_0),\ell_0\}.
$$
\end{proposition}
This proposition essentially says that $\kappa^{(\ell_0)}_v(x_0,t_0)$
is the truncation of $\kappa^{(\tilde \ell_0)}_v(x_0,t_0)$ by the
value $\ell_0$.
\begin{proof} It will be sufficient to prove the statement for
  $(x_0,t_0)=(0,0)$. To simplify the notation in the proof, we are
  going to denote
$$
\kappa=\kappa^{(\ell_0)}_v(0,0),\quad \tilde\kappa=\kappa^{(\tilde
  \ell_0)}_v(0,0),
$$
so we will need to show that
$$
\kappa=\min\{\tilde\kappa,\ell_0\}.
$$
First, we fix a cutoff function $\psi$ is the definition of the
functions $u_k$ and $u_{\tilde k}$, and note that
$$
|u_k(x,t)-u_{\tilde k}(x,t)|\leq C \|(x,t)\|^{\ell}.
$$
This implies that
$$
|H_{u_{k}}(r)-H_{u_{\tilde k}}(r)|\leq C r^{2\ell}.
$$
Arguing as in the proof of Lemma~\ref{lem:kappa-indep-psi}, we obtain
that, in fact,
$$
\kappa=\Phi_{u_k}^{(\ell_0)}(0+)=\Phi_{u_{\tilde k}}^{(\ell_0)}(0+).
$$
Using this information, form now on in this proof we will abbreviate
$H_{u_{\tilde k}}(r)$ with $H(r)$.

We consider two cases.

\setcounter{step}{0}

\step{1-trunc-kappa} Assume first that $\tilde \kappa<\ell_0$. In this
case we need to show that $\kappa=\tilde \kappa$.

From the assumption we will have that $\tilde \kappa<\tilde \ell_0$
and by Lemma~\ref{lem:Hu-mu}
$$
\tilde \kappa=\frac12\lim_{r\to 0+}\frac{r H'(r)}{H(r)}.
$$
Therefore, for any $\epsilon>0$ we obtain
$$
r\frac{H'(r)}{H(r)}\leq 2\tilde\kappa+\epsilon,\quad 0<r<r_\epsilon.
$$
Integrating, we find
$$
H(r)\geq\frac{H(r_\epsilon)}{r_\epsilon^{2\tilde\kappa+\epsilon}}
r^{2\tilde \kappa+\epsilon},\quad 0<r<r_\epsilon.
$$
In particular, if $\epsilon>0$ is so small that
$2\tilde\kappa+\epsilon<2\ell_0$, then $H(r)>r^{2\ell_0}$ and
therefore
$$
\kappa=\Phi^{(\ell_0)}(0+)=\frac12\lim_{r\to 0+}\frac{r
  H'(r)}{H(r)}=\tilde \kappa.
$$

\step{2-trunc-kappa} Assume now that $\tilde \kappa\geq \ell_0$. We
need to show in this case that $\kappa=\ell_0$. In general, we know
that $\kappa\leq \ell_0$, so arguing by contradiction, assume
$\kappa<\ell_0$. We thus know by Lemma~\ref{lem:Hu-mu} that $H(r)\geq
r^{2\ell_0}$ for $0<r<r_0$, and
$$
\kappa=\frac12\lim_{r\to 0+}\frac{r H'(r)}{H(r)}.
$$
But then, we also have $H(r)\geq r^{2\tilde\ell_0}$ for $0<r<r_0$ and
therefore
$$
\tilde \kappa=\Phi^{(\tilde\ell_0)}(0+)=\frac12\lim_{r\to 0+}\frac{r
  H'(r)}{H(r)}=\kappa<\ell_0,
$$
contrary to the assumption.
\end{proof}
\begin{definition}[Truncated homogeneity, part II]
  \label{def:trunc-hom-II}
  In view of Proposition~\ref{prop:kappa-compat}, if $\phi\in
  H^{\ell,\ell/2}(Q_1')$ we can push $\ell_0$ in the definition of the
  truncated homogeneity up to $\ell$ by setting
$$
\kappa^{(\ell)}_v(x_0,t_0)=\sup_{\ell_0<\ell}\kappa^{(\ell_0)}_v(x_0,t_0).
$$
Indeed, Proposition~\ref{prop:kappa-compat} guarantees that
$\ell_0\mapsto \kappa^{(\ell_0)}_v$ is monotone increasing. Moreover,
we have
$$
\kappa^{(\ell_0)}_v=\min\{\kappa^{(\ell)}_v,\ell_0\}.
$$
\end{definition}

\begin{lemma}\label{lem:kappa-upp-semi} The function $(x,t)\mapsto \kappa^{(\ell)}_v(x,t)$
  is upper semicontinuous on $\Gamma_*(v)$ (with respect to Euclidean
  or, equivalently, parabolic distance), i.e., for any
  $(x_0,t_0)\in\Gamma_*(v)$ one has
$$
\lim_{\delta\to 0} \sup_{\tilde Q_\delta'(x_0,t_0)\cap \Gamma_*(v)}
\kappa^{(\ell)}_v\leq \kappa_v^{(\ell)}(x_0,t_0).
$$
\end{lemma}
\begin{proof} Suppose first $\kappa=\kappa_v^{(\ell)}(x_0,t_0)<\ell$
  and fix $\ell_0\in (\kappa,\ell)$. Then, for any $\epsilon>0$ there
  exists $r_\epsilon>0$ such that
  $\Phi_{u}^{(\ell_0)}(r_\epsilon)<\kappa+\epsilon<\ell_0$, where
  $u=u_k^{(x_0,t_0)}$. This implies that that $H_{u}(r)\geq C
  r^{2(\kappa+\epsilon)}$ for $0<r<r_\epsilon$. Since the mapping
  $(x,t)\mapsto H_{u_k^{(x,t)}}(r_\epsilon)$ is continuous on
  $\Gamma_*(v)$, we will have
$$
H_{u^{(x,t)}_k}(r_\epsilon)\geq (C/2) r_\epsilon^{2(\kappa+\epsilon)}>
r_\epsilon^{2\ell_0},
$$
if $|x-x_0|^2+|t-t_0|<\eta_\epsilon^2$, $(x,t)\in \Gamma_*(v)$,
provided $r_\epsilon$ and $\eta_\epsilon>0$ are small enough. In
particular, this implies the explicit formula
$$
\Phi_{u_k^{(x,t)}}^{(\ell_0)}(r_\epsilon)=\frac12r_\epsilon
e^{Cr_\epsilon^\sigma}\frac{H_{u_k^{(x,t)}}'(r_\epsilon)}{H_{u_k^{(x,t)}}(r_\epsilon)}+2(e^{Cr_\epsilon^\sigma}-1).
$$
Therefore, taking $\eta_\epsilon>0$ small, we can guarantee
$$
|\Phi_{u_k^{(x,t)}}^{(\ell_0)}(r_\epsilon)-\Phi_{u_k^{(x_0,t_0)}}^{(\ell_0)}(r_\epsilon)|\leq\epsilon,
$$
if $|x-x_0|^2+|t-t_0|<\eta_\epsilon^2$, $(x,t)\in \Gamma_*(v)$. It
follows that, for such $(x,t)$, one has
$$
\kappa_v^{(\ell_0)}(x,t)\leq \Phi_{u_k^{(x,t)}}(r_\epsilon)\leq
\kappa+2\epsilon,
$$
which implies the upper semicontinuity of $\kappa_v^{(\ell_0)}$ and
$\kappa_v^{(\ell)}$ at $(x_0,t_0)$.

If $\kappa^{(\ell)}_v(x_0,t_0)=\ell$, the upper continuity follows
immediately since $\kappa^{(\ell)}_v\leq \ell$.
\end{proof}

The truncated homogeneity $\kappa_v^{(\ell)}$ gives a natural
classification of extended free boundary points.

\begin{definition}[Classification of free boundary points]
  Let $v\in \S_\phi(Q_1^+)$, with $\phi\in H^{\ell,\ell/2}(Q_1')$. For
  $\kappa\in [3/2,\ell]$, we define
$$
\Gamma_\kappa^{(\ell)}(v):=\{(x,t)\in\Gamma_*(v)\mid
\kappa_v^{(\ell)}(x,t)=\kappa\}.
$$
\end{definition}
As a direct corollary of Proposition~\ref{prop:kappa-compat}, we have
the following consistency for the above definition.

\begin{proposition}[Consistency of classification]
  \pushQED{\qed} If $\phi\in H^{\tilde\ell,\tilde\ell/2}(Q_1')$ with
  $\tilde \ell\geq \ell\geq 2$, then
  \begin{align*}
    \Gamma_\kappa^{(\ell)}(v)&=\Gamma_\kappa^{(\tilde\ell)}(v),\quad\text{if
    }\kappa<\ell,\\
    \Gamma_\ell^{(\ell)}(v)&=\bigcup_{\ell\leq \kappa\leq
      \tilde\ell}\Gamma_\kappa^{(\tilde\ell)}(v).\qedhere
  \end{align*}
  \popQED
\end{proposition}

The latter identity essentially means that, if $\phi$ is more regular
than $H^{\ell,\ell/2}$, then $\Gamma^{(\ell)}_\ell(v)$ is an
``aggregate'' of points with higher homogeneities $\kappa$.

We conclude this section with the following description of the free
boundary, based on the fact that the function $\kappa_v^{(\ell)}$
never takes certain values. We also characterize the points that are
in the extended free boundary $\Gamma_*(v)$, but not in the free
boundary $\Gamma(v)$.

\begin{proposition}\label{prop:free-bound-struct} If $v\in \S_\phi(Q_1^+)$ with $\phi\in
  H^{\ell,\ell/2}(Q_1')$, $\ell\geq 2$, then for any $(x_0,t_0)\in
  \Gamma_*(v)$, either we have
$$
\kappa^{(\ell)}_v(x_0,t_0)=\frac32,\quad\text{or}\quad2\leq
\kappa^{(\ell)}_v(x_0,t_0)\leq \ell.
$$
As a consequence,
\[
\Gamma_*(v)=\Gamma^{(\ell)}_{3/2}(v)\cup\bigcup_{2\leq\kappa\leq\ell}\Gamma_{\kappa}^{(\ell)}(v).
\]
Moreover,
\[
\Gamma_*(v)\setminus \Gamma(v)\subset
\Gamma^{(\ell)}_\ell(v)\cup\bigcup_{m\in\N}\Gamma_{2m+1}^{(\ell)}(v).
\]
\end{proposition}
\begin{proof} The first part is nothing but
  Theorem~\ref{thm:min-homogen}.

  Suppose now $(x_0,t_0)\in \Gamma_*(v)\setminus \Gamma(v)$ and that
  $\kappa=\kappa_v^{(\ell)}(x_0,t_0)<\ell$. Then, there exists a small
  $\delta>0$ such that $v=\phi$ on $Q_\delta'(x_0,t_0)$. Next,
  consider the translate $v^{(x_0,t_0)}=v(x_0+\cdot,t_0+\cdot)$ and
  let $u=u_k^{(x_0,t_0)}$ be obtained from $v^{(x_0,t_0)}$ as in
  Proposition~\ref{prop:uk-def}. Since $\kappa<\ell$, by
  Theorem~\ref{thm:exist-homogen-blowups}, there exists a blowup $u_0$
  of $u$ over some sequence $r=r_j\to 0+$. Since $u=0$ on $Q_\delta'$,
  $u_0$ will vanish on $S'_\infty$. Hence, extending it as an odd
  function $\tilde u_0$ of $x_n$ from $S_\infty^+$ to $S_\infty$, we
  will obtain a homogeneous caloric function in $S_\infty$. Then, by
  the Liouville theorem, $\tilde u_0$ must be a caloric polynomial of
  degree $\kappa$. Thus, $\kappa$ is an integer. We further claim that
  $\kappa$ is odd. Indeed, $\tilde u_0$ solves the Signorini problem
  in $S_\infty^+$ and therefore we must have that $-\partial_{x_n}
  u_0(x',0,t)$ is a nonnegative polynomial on $S_\infty'$ of
  homogeneity $\kappa-1$. The latter is possible when either
  $\kappa-1$ is even or if $-\partial_{x_n} u_0$ vanishes on
  $S_\infty'$. However, the latter case is impossible, since otherwise
  Holmgren's uniqueness theorem would imply that $u_0$ is identically
  zero, contrary to Theorem~\ref{thm:exist-homogen-blowups}. Thus, the
  only possibility is that $\kappa-1$ is even, or equivalently,
  $\kappa$ is odd. Since we also have $\kappa\geq 3/2>1$, we obtain
  that $\kappa\in\{2m+1\mid m\in\N\}$.
\end{proof}
\begin{remark} It is easy to construct $v\in\S_0(Q_1^+)$ such that
  $\Gamma(v)=\emptyset$ and
  $\Gamma_*(v)=\Gamma_{2m+1}^{(\ell)}(v)\not= \emptyset$.  The
  simplest example is perhaps
$$
v(x,t)=-\Im(x_1+ix_n)^{2m+1}.
$$
It is easy to verify that $v\in \S_0(Q_1^+)$, and $v=0$ on
$Q_1'$. Thus, $\Gamma(v)=\emptyset$. However,
$\Gamma_*(v)=\{0\}\times\{0\}\times(-1,0]$, and because of the
$(2m+1)$-homogeneity of $v$ with respect to any point on
$\Gamma_*(v)$, if we choose $\ell>2m+1$, we have that
$\Gamma_*(v)=\Gamma_{2m+1}^{(\ell)}(v)$.
\end{remark}

\section{Free boundary: Regular set}
\label{sec:free-bound-regul}
In this section we study a special subset $\cR(v)$ of the extended
free boundary. Namely, the collection of those points having minimal
frequency $\kappa=3/2$.

\begin{definition}[Regular set] Let $v\in \S_\phi(Q_1^+)$ with
  $\phi\in H^{\ell,\ell/2}(Q_1')$, $\ell\geq 2$. We say that
  $(x_0,t_0)\in \Gamma_*(v)$ is a \emph{regular} free boundary point
  if it has a minimal homogeneity $\kappa=3/2$, or equivalently
  $\kappa_{v}^{(\ell)}(x_0,t_0)=3/2$. The set
$$
\cR(v):=\Gamma_{3/2}^{(\ell)}(v)
$$
will be called the \emph{regular set} of $v$.
\end{definition}

We have the following basic fact about $\cR(v)$.

\begin{proposition}\label{prop:reg-set-rel-open} The regular set
  $\cR(v)$ is a relatively open subset of $\Gamma(v)$. In particular,
  for any $(x_0,t_0)\in \cR(v)$ there exists $\delta_0>0$ such that
$$
\Gamma(v)\cap Q_{\delta_0}'(x_0,t_0)=\cR(v)\cap
Q_{\delta_0}'(x_0,t_0).
$$

\end{proposition}
\begin{proof} First note that, by
  Proposition~\ref{prop:free-bound-struct}, we have $\cR(v)\subset
  \Gamma(v)$. The relative openness of $\cR(v)$ follows from the upper
  semicontinuity of the function $\kappa^{(\ell)}_v$ and form the fact
  that it does not take any values between $3/2$ and $2$.
\end{proof}

We will show in this section that, if the thin obstacle $\phi$ is
sufficiently smooth, then the regular set can be represented locally
as a $(n-2)$-dimensional graph of a parabolically Lipschitz
function. Further, such function can be shown to have H\"older
continuous spatial derivatives. We begin with the following basic
result.

\begin{theorem}[Lipschitz regularity of $\cR(v)$]
  \label{thm:lip-reg-reg-set}
  Let $v\in \S_\phi(Q_1^+)$ with $\phi\in H^{\ell,\ell/2}(Q_1')$,
  $\ell\geq 3$ and that $(0,0)\in\cR(v)$. Then, there exist
  $\delta=\delta_v>0$, and $g\in H^{1,1/2}(Q_\delta'')$ (i.e., $g$ is
  a parabolically Lipschitz function), such that possibly after a
  rotation in $\R^{n-1}$, one has
  \begin{align*}
    \Gamma(v)\cap Q_\delta'=\cR(v)\cap Q_\delta'&=\{(x',t)\in
    Q'_\delta\mid x_{n-1}=g(x'',t)\},\\
    \Lambda(v)\cap Q_\delta'&=\{(x',t)\in Q'_\delta\mid x_{n-1}\leq
    g(x'',t)\},
  \end{align*}
\end{theorem}
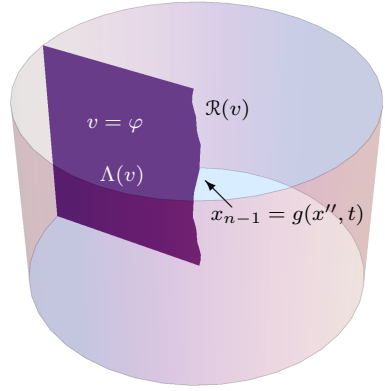
\begin{figure}
  \begin{picture}(150,150)
    \put(12,0){\includegraphics[height=150pt]{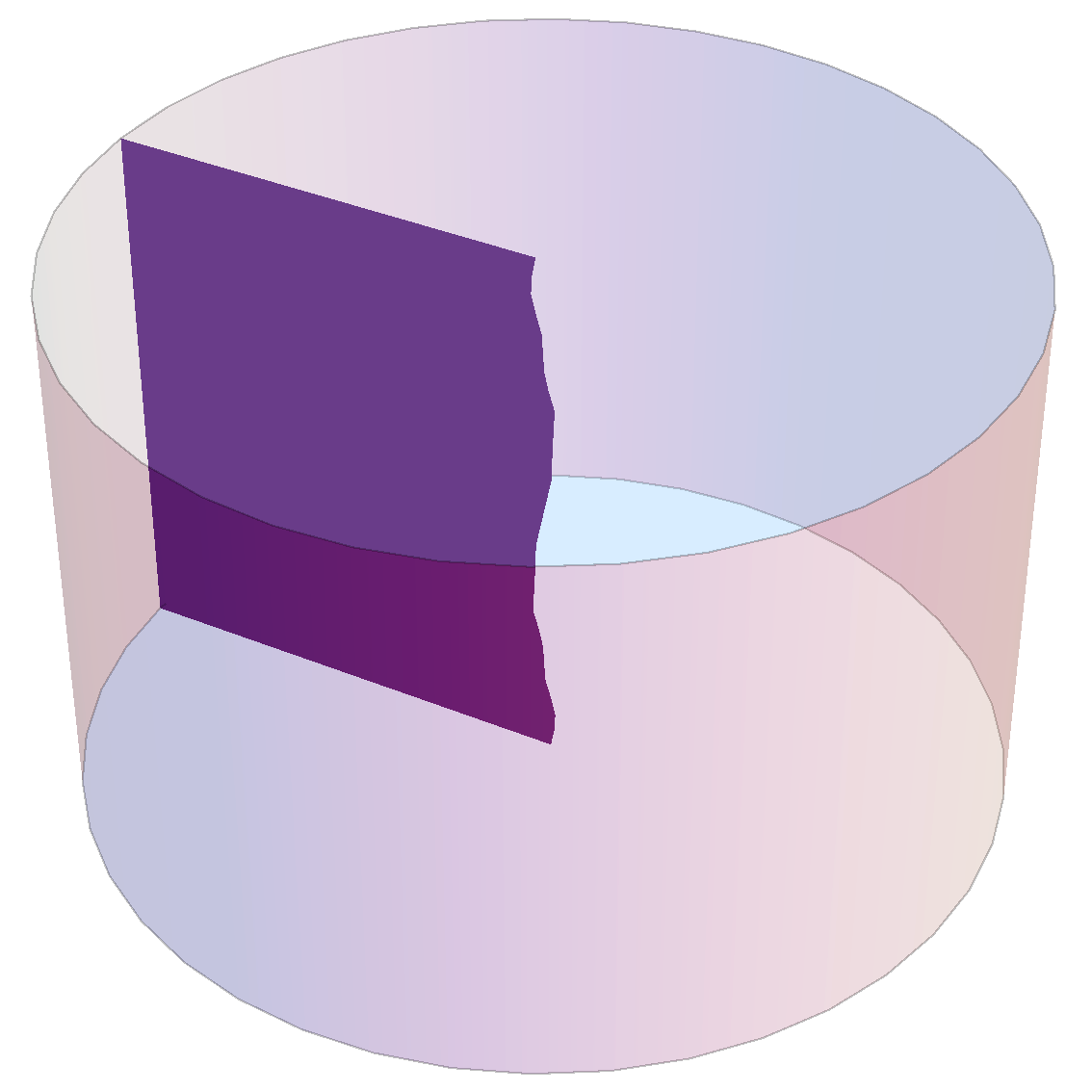}}
    \put(50,80){\footnotesize \color{white}{$\Lambda(v)$}}
    \put(45,100){\footnotesize \color{white}{$v=\phi$}}
    \put(92,65){\footnotesize $x_{n-1}=g(x'',t)$}
    \put(90,105){\footnotesize $\cR(v)$}
    \put(100,70){\vector(-1,1){10}}
  \end{picture}
  \caption{The regular set $\cR(v)$ in $Q_\delta'$, given by the graph
    $x_{n-1}=g(x'')$ with $g\in H^{1,1/2}(Q_\delta'')$ and
    $\nabla''g\in H^{\alpha,\alpha/2}(Q_\delta'')$ by
    Theorems~\ref{thm:lip-reg-reg-set} and \ref{thm:h1alpha-reg-set}}
  \label{fig:reg-set}
\end{figure}

For an illustration, see Fig~\ref{fig:reg-set}.  Following the
well-known approach in the classical obstacle problem, see e.g.\
\cite{PSU}*{Chapter~4}, the idea of the proof is to show that there is
a cone of directions in the thin space, along which $v-\phi$ is
increasing. This approach was successfully used in the elliptic
Signorini problem in \cite{ACS}, \cite{CSS}, see also
\cite{PSU}*{Chapter~9}, and in the arguments below we generalize the
constructions in these papers to the parabolic case. This will
establish the Lipschitz regularity in the space variables. To show the
$1/2$-H\"older regularity in $t$ (actually better than that), we will
use the fact that the $3/2$-homogeneous solutions of the parabolic
Signorini problem are $t$-independent (see
Proposition~\ref{prop:homogen-glob-sol-1-2}).

However, in order to carry out the program outlined above, in addition
to (i) and (ii) in Theorem~\ref{thm:exist-homogen-blowups} above, we
will need a stronger convergence of the rescalings $u_r$ to the
blowups $u_0$. This will be achieved by assuming a slight increase in
the regularity assumptions on the thin obstacle $\phi$, and,
consequently, on the regularity of the right-hand side $f$ in the
construction of Proposition~\ref{prop:uk-def}.

\begin{lemma}\label{lem:blowups-H1alpha-conv} Let $u\in \S^f(S_1^+)$,
  and suppose that for some $\ell_0\geq 2$
  \begin{alignat*}{2}
    |f(x,t)|&\leq M\|(x,t)\|^{\ell_0-2}&\quad&\text{in }S_1^+,\\
    |\nabla f(x,t)|&\leq L\|(x,t)\|^{(\ell_0-3)^+}&\quad&\text{in
    }Q_{{1/2}}^+,
  \end{alignat*}
  and
$$
H_u(r)\geq r^{2\ell_0},\quad\text{for }0<r<r_0.
$$
Then, for the family of rescalings $\{u_r\}_{0<r<r_{0}}$ we have the
uniform bounds
$$
\|u_r\|_{H^{3/2,3/4}(Q_R^+\cup Q_R')}\leq C_u,\quad 0<r<r_{R,u}.
$$
In particular, if the sequence of rescalings $u_{r_j}$ converges to
$u_0$ as in Theorem~\ref{thm:exist-homogen-blowups}, then over a
subsequence
$$
u_{r_j}\to u_0, \quad\nabla u_{r_j}\to \nabla u_0\quad\text{in }
H^{\alpha,\alpha/2}(Q_R^+\cup Q_R'),
$$
for any $0<\alpha<1/2$ and $R>0$.

\end{lemma}
\begin{proof} Because of Theorem~\ref{thm:opt-reg}, it is enough to
  show that $u_r$, $|\nabla u_r|$, and $f_r$ are bounded in $Q_R^+$.
  We have
  \begin{alignat*}{2}
    |f_r(x,t)|&=\frac{r^2|f(rx,r^2t)|}{H_u(r)^{1/2}}\\&\leq
    \frac{Mr^{\ell_0}\|(x,t)\|^{\ell_0-2}}{H_{u}(r)^{1/2}}\leq M \|(x,t)\|^{\ell_0-2}, &\quad&(x,t)\in S_R^+.\\
    \intertext{Besides, we have that} |\nabla
    f_r(x,t)|&=\frac{r^3|\nabla f(rx,r^2t)|}{H_u(r)^{1/2}}\\&\leq
    \frac{Lr^{\max\{\ell_0,3\}}\|(x,t)\|^{(\ell_0-3)_+}}{H_{u}(r)^{1/2}}\leq
    L \|(x,t)\|^{(\ell_0-3)_+},&\quad&(x,t)\in Q_R^+.
  \end{alignat*}
  Then, the functions
$$
w_\pm=(u_r)_\pm \quad\text{(evenly reflected to $S_R^-$)}
$$
satisfy
$$
\Delta w_\pm-\partial_t w_\pm\geq -M\|(x,t)\|^{\ell_0-2}\quad\text{in
}S_R.
$$
By Lemma~\ref{lem:l2linf-lower} we thus obtain
$$
\sup_{Q_{R/2}}|u_r|\leq C( H_{u_r}(R)^{1/2}+M R^{\ell_0})\leq C
R^{\ell_0}(1+M),
$$
for small $r$. Then, by the energy inequality for $w_\pm$ in
$Q_{R/2}$, we have
$$
\frac{1}{R^{n+2}}\int_{Q_{R/4}} |\nabla u_r|^2\leq C
R^{-2}R^{2\ell_0}(1+M)^2+CR^2 R^{2(\ell_0-2)}M^2\leq C
R^{2\ell_0-2}(1+M)^2.
$$
On the other hand, using that for $i=1,\ldots, n$,
$$
(w_i)_\pm=(\partial_{x_i}u_r)_\pm\quad\text{(evenly reflected to
  $S_R^-$)}
$$
satisfy
$$
\Delta (w_i)_\pm-\partial_t (w_i)_\pm\geq -L
R^{(\ell_0-3)_+}\quad\text{in }Q_R,
$$
then from $L_\infty-L_2$ estimate for subcaloric functions, we obtain
$$
\sup_{Q_{R/8}} |\nabla u_r|\leq C_nR^{\max\{\ell_0-1,2\}}(1+M+L).
$$
Thus, $u_r$, $|\nabla u_r|$ and $f_r$ are uniformly bounded in
$Q_{R/8}^+$ for small $r<r_{R,u}$, and this completes the proof of the
lemma.
\end{proof}

The next lemma will allow to deduce the monotonicity of the solution
$u$ in a cone of directions in the thin space, from that of the
blowup. It is the parabolic counterpart of \cite{ACS}*{Lemma~4} and
\cite{CSS}*{Lemma~7.2}.

\begin{lemma}\label{lem:glob-local} Let $\Lambda$ be a closed subset of
  $\R^{n-1}\times(-\infty,0]$, and $h(x,t)$ a continuous function in
  $Q_{1}$. For any $\delta_0>0$ there exists $\epsilon_0>0$, depending
  only on $\delta_0$ and $n$, such that if
  \begin{enumerate}[label=\textup{\roman*)}]
  \item $h\geq 0$ on $Q_1\cap\Lambda$,
  \item $(\Delta-\partial_t)h\leq \epsilon_0$ in
    $Q_1\setminus\Lambda$,
  \item $h\geq-\epsilon_0$ in $Q_1$,
  \item $h\geq\delta_0$ in $Q_1\cap\{|x_n|\geq c_n\}$,
    $c_n=1/(32\sqrt{n-1})$,
  \end{enumerate}
  then $h\geq 0$ on $Q_{1/2}$.
\end{lemma}

\begin{proof} 
  It is enough to show that $h\geq 0$ on $Q_{1/2}\cap\{|x_n|\leq
  c_n\}$. Arguing by contradiction, let $(x_0,t_0)\in
  Q_{1/2}\cap\{|x_n|\leq c_n\}$ be such that $h(x_0,t_0)<0$. Consider
  the auxiliary function
$$
w(x,t)=h(x,t)+\frac{\alpha_0}{2(n-1)}|x'-x_0'|^2+\alpha_0(t_0-t)-\left(\alpha_0+\frac{\epsilon_0}2\right)x_n^2,
$$
where $\alpha_0=\delta_0/2c_n^2$.  It is immediate to check that
$$
w(x_0,t_0)<0,\quad(\Delta-\partial_t)w\leq 0\quad\text{in
}Q_{1}\setminus\Lambda.
$$
Now, consider the function $w$ in the set $U=(Q_{3/4}\cap \{|x_n|\leq
c_n, t\leq t_0\})\setminus\Lambda$. By the maximum principle, we must
have
$$
\inf_{\partial_p U} w< 0.
$$
Analyzing the different parts of $\partial_p U$ we show that this
inequality cannot hold:

1) On $\Lambda\cap \partial_p U$ we have $w\geq 0$.

2) On $\{|x_n|=c_n\}\cap\partial_p U$ we have
$$
w(x,t)\geq h(x,t)-2\alpha_0 x_n^2\geq\delta_0-2\alpha_0 c_n^2\geq 0,
$$
if $\epsilon_0\leq 2\alpha_0$.

3) On $\{|x_n|<c_n\}\cap\partial_p U$ we have
\begin{align*}
  w(x,t)&\geq-\epsilon_0+\frac{\alpha_0}{2(n-1)}|x'-x_0'|^2-2\alpha_0
  x_n^2\\
  &\geq-\epsilon_0+\alpha_0\epsilon_n,
\end{align*}
with
$$
\epsilon_n=\frac{1}{128(n-1)}-2c_n^2=\frac{3}{512(n-1)}>0.
$$
If we choose $\epsilon_0<\alpha_0\epsilon_n$, we conclude that $w\geq
0$ on this portion of $\partial_p U$.

4) On $t=-9/16$ we have
$$
w(x,t)\geq-\epsilon_0+\alpha_0\frac{5}{16}-2\alpha_0 c_n^2\geq
-\epsilon_0+\alpha_0\left(\frac{5}{16}-2c_n^2\right)\geq0,
$$
for $\epsilon_0<\alpha_0/4$.

In conclusion, if $\epsilon_0$ is sufficiently small, we see that we
must have $\inf_{\partial_p U} w \ge 0$, thus arriving at a
contradiction with the assumption that $h(x_0,t_0)<0$. This completes
the proof.
\end{proof}

\begin{proof}[Proof of Theorem~\ref{thm:lip-reg-reg-set}] Let $u=u_k$
  and $f=f_k$ be as in Proposition~\ref{prop:uk-def}. From the
  assumption $\ell\geq 3$, we have that $|f|\leq M$ in $S_1^+$ and
  $|\nabla f|\leq L$ in $Q_{1/2}^+$. We also choose $\ell_0=2$. We
  thus conclude that $H_u(r)\geq r^{2\ell_0}$ for $0<r<r_u$. In view
  of Theorem~\ref{thm:exist-homogen-blowups}, the rescalings $u_{r_j}$
  converge (over a sequence $r=r_j\to 0+$) to a homogeneous global
  solution $u_0$ of degree $3/2$. Furthermore, we note that
  Lemma~\ref{lem:blowups-H1alpha-conv} is also applicable here. In
  view of Proposition~\ref{prop:homogen-glob-sol-1-2}, after a
  possible rotation in $\R^{n-1}$, we may assume that
$$
u_0(x,t)=C_n\Re(x_{n-1}+i x_n)^{3/2}.
$$ 
It can be directly calculated that for any $e\in \partial B_1'$
\begin{align*}
  \partial_{e} u_0(x,t)&=\frac32C_n(e\cdot e_{n-1})\Re(x_{n-1}+ix_n)^{1/2}\\
  &=\frac3{2\sqrt2}C_n(e\cdot
  e_{n-1})\sqrt{\sqrt{x_{n-1}^2+x_n^2}+x_{n-1}}.
\end{align*}
Thus, if for any given $\eta>0$ we consider the thin cone around
$e_{n-1}$
$$
\C_\eta':=\{x'=(x'',x_{n-1})\in\R^{n-1}\mid x_{n-1}\geq \eta |x''|\},
$$
then it is immediate to conclude that for any $e\in\C_\eta'$, $|e|=1$,
\begin{alignat*}{2}
  &\partial_e u_0\geq 0, &\quad&\text{in } Q_1^+\\
  &\partial_e u_0\geq \delta_{n,\eta}>0,&\quad&\text{in }
  Q_1^+\cap\{x_n\geq c_n\},
\end{alignat*}
where $c_n=1/(32\sqrt{n-1})$ is the dimensional constant in
Lemma~\ref{lem:glob-local}.  We next observe that, by
Lemma~\ref{lem:blowups-H1alpha-conv}, for any given $\epsilon>0$ we
will have for all directions $e\in \partial B_1'$
$$
|\partial_{e}u_{r_j}-\partial_e u_0|<\epsilon\quad\text{on }Q_1^+,
$$
provided $j$ is sufficiently large.  Moreover, note that in view of
Proposition~\ref{prop:uk-def} we can estimate
\begin{align*}
  |(\Delta-\partial_t) \partial_e u_{r_j}|=\frac{C
    r_j^\ell\|(x,t)\|^{\ell-3}}{H_u(r_j)^{1/2}}\leq C
  r_j^{\ell-\ell_0}\to 0\quad\text{uniformly in }Q_1^+.
\end{align*}
Thus, the function $h=\partial_{e}u_{r_j}$ (evenly reflected to $Q_1$)
will satisfy the conditions of Lemma~\ref{lem:glob-local}, and
therefore we conclude that
$$
\partial_e u_{r_{j}}\geq 0\quad\text{in }Q_{1/2}^+,\quad\text{for any
}e\in \C_\eta',\ |e|=1,
$$
for $j\geq j_\eta$.  Scaling back, we obtain that
$$
\partial_e u\geq 0\quad\text{in }Q_{r_{\eta}}^+,\quad\text{for any
}e\in \C_\eta',\ |e|=1,
$$
where $r_\eta=r_{j_\eta}/2$.  Now a standard argument (see
\cite{PSU}*{Chapter~4, Exercise~4.1}) implies that
$$
\{u(x',0,t)>0\}\cap Q'_{r_\eta}=\{ (x',t)\in Q'_{r_\eta}\mid x_{n-1}>
g(x'',t)\},
$$
where, for every fixed $t\in (-r_\eta^2, 0]$, $x''\mapsto g(x'',t)$ is
a Lipschitz continuous function with
$$
|\nabla'' g|\leq \eta.
$$
We are now left with showing that $g$ is $(1/2)$-H\"older continuous
in $t$.  In fact, we are going to show that
$|g(x,t)-g(x,s)|=o(|t-s|^{1/2})$, uniformly in $Q''_{r_\eta/2}$.

Suppose, towards a contradiction, that for $x_j''\in B''_{r_\eta/2}$,
$-r_\eta^2/4\leq s_j<t_j\leq 0$, $t_j-s_j\to 0$, we have for some
$C>0$
\begin{equation}\label{eq:awayfromzero}
  |g(x_j'',t_j)-g(x_j'',s_j)|\geq C |t_j-s_j|^{1/2}.
\end{equation}
Let
$$
x_j'=(x_j'',g(x_j'',t_j)),\quad y_j'=(x_j'',g(x_j'',s_j))
$$
and
$$
\delta_j=\max\{|g(x_j'',t_j)-g(x_j'',s_j)|, |t_j-s_j|^{1/2}\}.
$$
Let also
$$
\xi_j'=\frac{y_j'-x_j'}{\delta_j},\quad
\tau_j=\frac{s_j-t_j}{\delta_j^2}.
$$
Note that
$$
\xi_j'=|\xi_j'|e_{n-1},\quad (\xi_j', \tau_j)\in \partial_p Q_1'.
$$
Moreover, we claim that $\delta_j\to 0$. Indeed, we may assume that
the sequences $x_j'$, $y_j'$, $t_j$, $\delta_j$ converge to some $x'$,
$y'$, $t$, $\delta$ respectively.  If $\delta>0$ then we obtain that
$(x',t), (y',t)\in \Gamma(v)$. But $y'-x'=|y'-x'|e_{n-1}$, which
cannot happen since $\Gamma(v)$ is given as a graph
$\{x_{n-1}=g(x'',t)\}$ in $Q'_{r_\eta}$. Thus, $\delta_j\to 0$.

Consider now the rescalings of $u$ at $(x_j,t_j)$ by the factor of
$\delta_j$:
\begin{equation}\label{eq:wj}
  w_j(x,t)
  =\frac{u(x_j+\delta_jx, t_j+\delta_j^2t)}{H_{u^{(x_j,t_j)}}(\delta_j)^{1/2}}.
\end{equation}
We want to show that the sequence $w_j$ converges to a homogeneous
global solution in $S_\infty$, of homogeneity $3/2$.  For that
purpose, we first assume that $r_\eta$ is so small that
$$
\Gamma(u)\cap Q'_{r_\eta}=\Gamma_{3/2}^{(2)}(u)\cap Q'_{r_\eta}.
$$
This is possible by the upper semicontinuity of the mapping
$(x,t)\mapsto \kappa^{(2)}_u(x,t)=\Phi_{u^{(x,t)}}^{(2)}(0+)$ on
$\Gamma_*(u)$ as in Lemma~\ref{lem:kappa-upp-semi}, the equality
$\kappa^{(2)}_u(0,0)=\kappa_v^{(2)}(0,0)=3/2$, and
Theorem~\ref{thm:min-homogen}. Moreover, arguing as in the proof of
Lemma~\ref{lem:kappa-upp-semi}, we may assume that
$$
\Phi_{u^{(x,t)}}^{(2)}(r)<7/4,\quad\text{if } r<r_0,\
(x,t)\in\Gamma(u)\cap Q'_{r_\eta}.
$$
This assumption implies
$$
H_{u^{(x,t)}}(r)\geq r^4,\quad\text{if } r<r_0,\ (x,t)\in\Gamma(u)\cap
Q'_{r_\eta}.
$$
Otherwise, we would have $\Phi_{u^{(x,t)}}^{(2)}(r)\geq 2$, a
contradiction. As a consequence, the functions
$$
\phi_r(x,t)=\Phi_{u^{(x,t)}}^{(2)}(r),\quad (x,t)\in\Gamma(u)\cap
Q'_{r_\eta}
$$
will have an explicit representation through $H_{u^{(x,t)}}(r)$ and
its derivatives, and therefore will be continuous.  We thus have a
monotone family of continuous functions $\{\phi_r\}$ on a compact set
$K=\Gamma(u)\cap\overline{Q'_{r_\eta/2}}$ such that
$$
\phi_r\searrow 3/2\quad\text{on }K\quad\text{as }r\searrow 0.
$$
By the theorem of Dini the convergence $\phi_r\to 3/2$ is uniform on
$K$. This implies that
$$
\phi_{r_j}(x_j,t_j)\to 3/2\quad\text{for any }(x_j,
t_j)\in\Gamma(u)\cap\overline{Q'_{r_\eta/2}},\ r_j\to 0.
$$
For the functions $w_j$ defined in \eqref{eq:wj} above this implies
$$
\Phi_{w_j}^{(2)}(r)\to 3/2\quad\text{as }j\to \infty,
$$
for any $r>0$.  Now, analyzing the proof of
Theorem~\ref{thm:exist-homogen-blowups}, we realize that the same
conclusions can be drawn about the sequence $w_j$ as for the sequence
of rescalings $u_{r_j}$. In particular, over a subsequence, we have
$w_j\to w_0$ in $L_{2,\loc}(S_\infty)$, where $w_0$ is a
$3/2$-homogeneous global solution of the Signorini problem.  By
Proposition~\ref{prop:homogen-glob-sol-1-2} we conclude that for some
direction $e_0\in\R^{n-1}$ it must be
$$
w_0(x,t)=C_n\Re (x'\cdot e_0+i x_n)^{3/2}.
$$
Further, since $\partial_e u\geq 0$ for unit $e\in \C_\eta'$, we must
have
$$
\partial_e w_0\geq 0,\quad\text{in }S_\infty^+,\quad e\in \C_\eta'.
$$
Therefore,
$$
e_0\cdot e\geq 0\quad\text{for any } e\in\C_\eta' \Rightarrow e_0\cdot
e_{n-1}>0.
$$
Further, note that since $u\in\S^f(S_1^+)$ with $|f|\leq M$ in
$S_1^+$, and $|\nabla f|\leq L$ in $Q_{1/2}^+$, we can repeat the
arguments in the proof of Lemma~\ref{lem:blowups-H1alpha-conv} (with
$\ell_0=2$) to obtain for any $R>0$
\begin{equation}\label{eq:wjw0Ha}
  w_j\to w_0,\quad \nabla w_j\to \nabla w_0\quad\text{in }H^{\alpha,\alpha/2}(Q_R^+\cup Q'_R).
\end{equation}
Going back to the construction of the functions $w_j$, note that
$(\xi_j',\tau_j)\in \Gamma(w_j)$, in addition to
$(\xi_j',\tau_j)\in\partial_p Q_1'$.  Without loss of generality we
may assume that $(\xi_j',\tau_j)\to(\xi_0',\tau_0)\in \partial_p
Q_1'$. But then the convergence \eqref{eq:wjw0Ha} implies that
$w_0(\xi_0',\tau_0)=0$ and $\nabla w_0(\xi_0',\tau_0)=0$. From the
explicit formula for $w_0$ it follows that
$$
(\xi_0',\tau_0)\in \{(x',t)\in S_\infty'\mid x'\cdot e_0=0\},
$$
or equivalently, $\xi_0'\cdot e_0=|\xi_0'|e_{n-1}\cdot e_0=0$.  Since
$e_{n-1}\cdot e_0>0$, we must have $\xi_0'=0$.  Thus, we have proved
that
$$
|\xi_j'|=\frac{|g(x_j'',s_j)-g(x_j'',t_j)|}{\max\{|g(x_j'',s_j)-g(x_j'',t_j)|,
  |t_j-s_j|^{1/2}\}}\to 0,
$$
which is equivalent to
$$
\frac{|g(x_j'',s_j)-g(x_j'',t_j)|}{|t_j-s_j|^{1/2}}\to 0,
$$
contrary to our assumption \eqref{eq:awayfromzero}.
\end{proof}

We next show that, following an idea in \cite{AC0}, the regularity of
the function $g$ can be improved with an application of a boundary
Harnack principle.

\begin{theorem}[H\"older regularity of $\nabla''g$]
  \label{thm:h1alpha-reg-set}

  In the conclusion of Theorem~\ref{thm:lip-reg-reg-set}, one can take
  $\delta>0$ so that $\nabla''g\in
  H^{\alpha,\frac{\alpha}2}(Q_\delta'')$ for some $\alpha>0$.
\end{theorem}

To prove this theorem we first show the following nondegeneracy
property of $\partial_e u$.

\begin{proposition}[Nondenegeracy of $\partial_e u$]\label{prop:nondeg-deu} Let $v\in\S_\phi(Q_1^+)$ and $u=u_k$ be as
  Theorem~\ref{thm:lip-reg-reg-set}. Then, for any $\eta>0$ there
  exist $\delta>0$ and $c>0$ such that
$$
\partial_e u\geq c\, d(x,t)\quad \text{in }Q_\delta^+,\quad\text{for
  any } e\in \C_\eta',\ |e|=1,
$$
where
$$
d(x,t)=\sup\{r\mid \tilde Q_r(x,t)\cap Q_\delta \subset
Q_\delta\setminus \Lambda(v)\}
$$
is the parabolic distance from the point $(x,t)$ to the coincidence
set $\Lambda(v)\cap Q_\delta'$.
\end{proposition}
The proof is based on the following improvement on
Lemma~\ref{lem:glob-local}, which is the parabolic counterpart of
\cite{CSS}*{Lemma~7.3}.

\begin{lemma}\label{lem:h-nondeg}
  For any $\delta_0>0$ there exist $\epsilon_0>0$ and $c_0>0$,
  depending only on $\delta_0$ and $n$, such that if $h$ is a
  continuous function on $Q_1\cap\{0\leq x_n\leq c_n\}$,
  $c_n=1/(32\sqrt{n-1})$, satisfying
  \begin{enumerate}[label=\textup{\roman*)}]
  \item $(\Delta -\partial_t)h\leq \epsilon_0$ in
    $Q_{1}\cap\{0<x_n<c_n\}$,
  \item $h\geq 0$ in $Q_{1}\cap\{0<x_n<c_n\}$,
  \item $h\geq \delta_0$, on $Q_1\cap\{x_n=c_n\}$,
  \end{enumerate}
  then
$$
h(x,t)\geq c_0 x_n\quad\text{in }Q_{1/2}\cap\{0<x_n<c_n\}.
$$
\end{lemma}
\begin{proof} The proof is very similar to that of
  Lemma~\ref{lem:glob-local}. Let $(x_0,t_0)\in
  Q_{1/2}\cap\{0<x_n<c_n\}$, and consider the auxiliary function
$$
w(x,t)=h(x,t)+\frac{\alpha_0}{2(n-1)}|x'-x_0'|^2+\alpha_0(t_0-t)-\left(\alpha_0+\frac{\epsilon_0}2\right)x_n^2-c_0x_n,
$$
with $\alpha_0=\delta_0/2c_n^2$.

As before, we have $(\Delta-\partial_t)w\leq 0$ in
$Q_{1}\cap\{0<x_n<c_n\}$. We now claim that $w\geq 0$ on
$U^+=Q_{3/4}\cap\{0<x_n<c_n, t<t_0\}$. This will follow once we verify
that $w\geq 0$ on $\partial_p U^+$.

We consider several cases:

1) On $\{x_n=0\}\cap \partial_p U^+$, we clearly have $w\geq 0$.

2) On $\{x_n=c_n\}\cap \partial_p U^+$, one has
$$
w\geq \delta_0-\frac32\alpha_0 c_n^2-c_0 c_n= \frac12\alpha_0
c_n^2-c_0 c_n\geq 0,
$$
provided $\epsilon_0\leq \alpha_0$ and $c_0\leq c_n\alpha_0/2$.

3) On $\{0<x_n<c_n\}\cap \partial_p U^+$ we have
$$
w\geq \frac{\alpha_0}{128(n-1)}-2\alpha_0c_n^2-c_0c_n\geq
\frac{3\alpha_0}{512(n-1)}-c_0c_n\geq 0,
$$
provided $c_0<3\alpha_0/(512(n-1)c_n)$.

4) On $\{t=-9/16\}\cap \partial_p U^+$ we have
$$
w\geq \frac{5\alpha_0}{16}-2\alpha_0c_n^2-c_0c_n\geq
\frac{\alpha_0}{4}-c_0c_n\geq 0,
$$
provided $c_0< \alpha_0/(4 c_n)$.

In conclusion, for small enough $\epsilon_0$ and $c_0$, we have $w\geq
0$ in $U$, and in particular $w(x_0,t_0)\geq 0$. This implies that
$h(x_0,t_0)> c_0 (x_0)_n$, as claimed.
\end{proof}

\begin{proof}[Proof of Proposition~\ref{prop:nondeg-deu}] Considering
  the rescalings $u_r$ as in the proof of
  Theorem~\ref{thm:lip-reg-reg-set}, and applying
  Lemma~\ref{lem:h-nondeg}, we obtain
$$
\partial_e u_{r}\geq c_\eta |x_n|\quad\text{in }Q_{1/2},\quad e\in
\C'_\eta,
$$
for $0<r=r_\eta$ small. (Here, we assume that $u$ has been extended by
even symmetry in $x_n$ variable to $Q_1$.)  Besides, by choosing $r$
small, we can also make
$$
|(\Delta-\partial_t)\partial_e u_r|\leq Cr^{\ell-\ell_0}\leq
c_\eta\epsilon_n\quad\text{in }Q_{1/2}\setminus\Lambda(u_r),
$$
for a dimensional constant $\epsilon_n>0$ to be specified below. Let
now $(x,t)\in Q_{1/4}^+$ and $d=d_{r}(x,t)$ be the parabolic distance
from $(x,t)$ to $\Lambda(u_{r})\cap Q_1'$. Consider the lowest
rightmost point on the boundary $\partial \tilde Q_d(x,t)$
$$
(x_*,t_*)=(x+e_n d, t-d^2).
$$
We have
$$
\partial_e u_{r}(x_*,t_*)\geq c_0 d.
$$
By the parabolic Harnack inequality (see, e.g.,\ \cite{Lie}*{Theorems
  6.17--6.18})
$$
\partial_e u_{r}(x,t)\geq c_n c_0 d- C_nc_0\epsilon_n d^2\geq c\,d,
$$
if we take $\epsilon_n$ sufficiently small.  Scaling back to $u$, we
complete the proof of the proposition.
\end{proof}

A key ingredient in the proof of Theorem~\ref{thm:h1alpha-reg-set} is
the following version of the parabolic boundary Harnack principle for
domains with thin Lipschitz complements established in
\cite{Shi}*{Section~7}.  To state the result, we will need the
following notations. For a given $L\geq 1$ and $r>0$ denote
\begin{align*}
  \Theta''_r&=\{(x'',t)\in\R^{n-2}\times\R\mid |x_i|<r,
  i=1,\ldots,n-2, -r^2<t\leq 0\},\\
  \Theta'_r&=\{(x',t)\in\R^{n-1}\times\R\mid (x'',t)\in\Theta''_r, |x_{n-1}|<4nLr\},\\
  \Theta_r&=\{(x,t)\in\R^n\times\R\mid (x'',t)\in\Theta'_r,
  |x_{n}|<r\}.
\end{align*}

\begin{lemma}[Boundary Harnack principle]\label{lem:BHP} Let
  $$
  \Lambda=\{(x',t)\in \Theta_1'\mid x_{n-1}\leq g(x'',t)\}
$$ 
for a parabolically Lipschitz function $g$ in $\Theta_1''$ with
Lipschitz constant $L\geq 1$ such that $g(0,0)=0$. Let $u_1$, $u_2$ be
two continuous nonnegative functions in $\Theta_1$ such that for some
positive constants $c_0$, $C_0$, $M$, and $i=1,2$,
\begin{enumerate}[label=\textup{\roman*)}]
\item $0\leq u_i\leq M$ in $\Theta_{1}$ and $u_i=0$ on $\Lambda$,
\item $|(\Delta-\partial_t) u_i|\leq C_0$ in $\Theta_1\setminus
  \Lambda$,
\item $u_i(x,t)\geq c_0\,d(x,t)$ in $\Theta_1\setminus\Lambda$, where
  $d(x,t)=\sup\{r\mid \tilde \Theta_r(x,t)\cap \Lambda=\emptyset\}$.
\end{enumerate}
Assume additionally that $u_1$ and $u_2$ are symmetric in $x_n$.
Then, there exists $\alpha\in (0,1)$ such that
$$
\frac{u_1}{u_2}\in H^{\alpha,\alpha/2}(\Theta_{1/2}).
$$
Furthermore, $\alpha$ and the bound on the corresponding norm
$\|u_1/u_2\|_{H^{\alpha,\alpha/2}(\Theta_{1/2})}$ depend only on $n$,
$L$, $c_0$, $C_0$, and $M$.\qed
\end{lemma}

\begin{remark}
  We note that, unlike the elliptic case, Lemma \ref{lem:BHP} above
  cannot be reduced to the other known results in the parabolic
  setting (see, e.g., \cite{Kem}, \cite{FGS} for parabolically
  Lipschitz domains, or \cite{HLN} for parabolically NTA domains with
  Reifenberg flat boundary). We also note that this version of the
  boundary Harnack is for functions with nonzero right-hand side and
  therefore the nondegeneracy condition as in iii) is necessary. The
  elliptic version of this result has been established in \cite{CSS}.
\end{remark}

We are now ready to prove Theorem~\ref{thm:h1alpha-reg-set}.

\begin{proof}[Proof of Theorem~\ref{thm:h1alpha-reg-set}]
  Fix $\eta>0$ and let $\theta=\theta_\eta$ be such that
  $e=(\cos\theta)e_{n-1}+(\sin\theta)e_j\in \C'_\eta$ for
  $j=1,\ldots,n-2$. Consider two functions
$$
u_1=\partial_{e}u\quad\text{and}\quad u_2=\partial_{e_{n-1}}u.
$$
Then, by Proposition~\ref{prop:nondeg-deu}, the conditions of
Lemma~\ref{lem:BHP} are satisfied for some rescalings of $u_1$ and
$u_2$. Hence, applying Lemma~\ref{lem:BHP} and scaling back, we obtain
that for a small $\delta>0$, and $\alpha\in (0,1)$,
$$
\frac{\partial_e u}{\partial_{e_{n-1}}u}\in
H^{\alpha,\alpha/2}(\Theta_{\delta}).
$$
This gives
$$
\frac{\partial_{e_j} u}{\partial_{e_{n-1}}u}\in
H^{\alpha,\alpha/2}(\Theta_{\delta}),\quad j=1,\ldots, n-2.
$$
Hence, the level surfaces $\{u=\epsilon\}\cap \Theta'_{\delta}$ are
given as graphs
$$
x_{n-1}=g_\epsilon(x'',t),\quad x''\in \Theta_{\delta}'',
$$
with uniform in $\epsilon$ estimates on
$\langle\nabla''g_\epsilon\rangle^{(\alpha)}_{\Theta_{\delta}''}$ for
small $\epsilon>0$. Consequently, we obtain
\[
\nabla'' g\in H^{\alpha,\alpha/2}(\Theta_{\delta}''),
\]
and this completes the proof of the theorem.
\end{proof}

\section{Free boundary: Singular set}
\label{sec:free-bound-sing}

We now turn to the study of a special class of free boundary points,
called \emph{singular points}, that are characterized by the property
that the coincidence set has a zero density at those points with
respect to the $\mathcal{H}^{n}$ measure in the thin space. In the
time-independent Signorini problem the analysis of the singular set
was carried in the paper \cite{GP}.

\begin{definition}[Singular points]\label{def:parab-sing-points} Let $v\in \S_\phi(Q_1^+)$ with $\phi\in H^{\ell,\ell/2}(Q_1')$,
  $\ell\geq 2$. We say that $(x_0,t_0)\in \Gamma_*(v)$ is
  \emph{singular} if
$$
\lim_{r\to 0+}\frac{\mathcal{H}^{n}(\Lambda(v)\cap
  Q_r'(x_0,t_0))}{\mathcal{H}^n(Q_r')}=0.
$$
We will denote the set of singular points by $\Sigma(u)$ and call it
the \emph{singular set}. We can further classify singular points
according to the homogeneity of their blowup, by defining
$$
\Sigma_\kappa(v):=\Sigma(v)\cap \Gamma_\kappa^{(\ell)}(v),\quad
\kappa\leq\ell.
$$  
\end{definition}
Since we are going to work with the blowups, we will write
$\ell=k+\gamma$, for $k\in\N$, $0<\gamma\leq 1$, and construct the
functions $u_k\in \S^{f_k}(S_1^+)$ as in
Proposition~\ref{prop:uk-def}. By abusing the notation, we will write
$0\in \Sigma_\kappa(u_k)$ whenever $0\in \Sigma_\kappa(v)$.  Also, for
technical reasons, similarly to what we did in the study of the
regular set, we will assume that $\ell\geq 3$ for most of the results
in this section.

The following proposition gives a complete characterization of the
singular points in terms of the blowups and the generalized
frequency. In particular, it establishes that
$$
\Sigma_\kappa(v)=\Gamma_\kappa(v)\quad\text{for }\kappa=2m<\ell,\
m\in\N.
$$

\begin{proposition}[Characterization of singular points]
  \label{prop:char-sing-point}
  Let $u\in\S^f(S_1^+)$ with $|f(x,t)|\leq M \|(x,t)\|^{\ell-2}$ in
  $S_1^+$, $|\nabla f(x,t)|\leq L\|(x,t)\|^{\ell-3}$ in $Q_{1/2}^+$,
  $\ell\geq 3$ and $0\in \Gamma^{(\ell)}_\kappa(u)$ with
  $\kappa<\ell$. Then, the following statements are equivalent:
  \begin{enumerate}[label=\textup{(\roman*)}]
  \item $0\in \Sigma_\kappa(u)$.
  \item any blowup of $u$ at the origin is a nonzero parabolically
    $\kappa$-homogeneous polynomial $p_\kappa$ in $S_\infty$
    satisfying
$$
\Delta p_\kappa-\partial_t p_\kappa=0,\quad p_\kappa(x',0,t)\geq
0,\quad p_\kappa(x',-x_n, t)=p_\kappa(x',x_n,t).
$$
We denote this class by $\P_\kappa$.
\item $\kappa=2m$, $m\in\N$.
\end{enumerate}
\end{proposition}
\begin{proof} (i) $\Rightarrow$ (ii) Recall that the rescalings $u_r$
  satisfy
  $$
  \Delta u_r-\partial_t u_r=f_r+2(\partial_{x_n}^+ u_r)
  \mathcal{H}^{n}\big|_{\Lambda(u_r)}\quad\text{in }S_{1/r},
  $$
  in the sense of distributions, after an even reflection in $x_n$
  variable. Since $u_r$ are uniformly bounded in $W^{2,1}_2(Q_{2R}^+)$
  for small $r$ by Theorem~\ref{thm:exist-homogen-blowups},
  $\partial_{x_n}^+ u_r$ are uniformly bounded in $L_2(Q_R')$. On the
  other hand, if $0\in \Sigma(u)$, then
$$
\frac{\mathcal{H}^n(\Lambda(u_r)\cap
  Q_R')}{R^n}=\frac{\mathcal{H}^n(\Lambda(u)\cap Q_{Rr})}{(Rr)^n}\to
0\quad\text{as }r\to 0,
$$
and therefore
$$
(\partial_{x_n}^+ u_r) \mathcal{H}^{n}\big|_{\Lambda(u_r)}\to
0\quad\text{in }Q_R.
$$
in the sense of distributions. Further, the bound $|f(x,t)|\leq M
\|(x,t)\|^{\ell-2}$ implies that
\begin{align*}
  |f_r(x,t)|&=\frac{r^2|f(rx,r^2t)|}{H_u(r)^{1/2}}\leq \frac{M
    r^{\ell}}{H_u(r)^{1/2}}\|(x,t)\|^{\ell-2}\\
  &\leq C r^{\ell-\ell_0} R^{\ell-2}\to 0\quad\text{in }Q_R,
\end{align*}
where $\ell_0\in (\kappa,\ell)$ and we have used the fact that
$H_u(r)\geq r^{2\ell_0}$ for $0<r<r_u$. Hence, any blowup $u_0$ is
caloric in $Q_R$ for any $R>0$, meaning that it is caloric in the
entire $S_\infty=\R^n\times(-\infty,0]$.  On the other hand, by
Proposition~\ref{thm:exist-homogen-blowups}(iv), the blowup $u_0$ is
homogeneous in $S_\infty$ and therefore has a polynomial growth at
infinity. Then by the Liouville theorem we can conclude that $u_0$
must be a homogeneous caloric polynomial $p_\kappa$ of a certain
integer degree $\kappa$. Note that $p_\kappa=u_0\not\equiv 0$ by
construction. The properties of $u$ also imply that that
$p_\kappa(x',0,t)\geq 0$ for all $(x',t)\in S_\infty'$ and and
$p_\kappa(x',-x_n,t)=p_\kappa(x',x_n,t)$ for all $(x',x_n,t)\in
S_\infty$.

\medskip\noindent (ii) $\Rightarrow$ (iii) Let $p_\kappa$ be a blowup
of $u$ at the origin. Since $p_\kappa$ is a polynomial, clearly
$\kappa\in\N$. If $\kappa$ is odd, the nonnegativity of $p_\kappa$ on
$\R^{n-1}\times\{0\}\times\{-1\}$ implies that $p_\kappa$ vanishes
there identically, implying that $p_\kappa\equiv 0$ on $S_\infty'$. On
the other hand, from the even symmetry in $x_n$ we also have that
$\partial_{x_n} p_\kappa\equiv 0$ on $S_\infty'$. Since $p_\kappa$ is
caloric in $\R^n$ and $S_\infty'$ is not characteristic for the heat
operator, Holmgren's uniqueness theorem implies that $p_\kappa\equiv
0$ in $\R^n$, contrary to the assumption. Thus, $\kappa\in\{2m\mid
m\in\N\}$.

\medskip\noindent (iii) $\Rightarrow$ (ii) The proof of this
implication is stated as a separate Liouville-type result in
Lemma~\ref{lem:2m-homogen} below.

\medskip\noindent (ii) $\Rightarrow$ (i) Suppose that $0$ is not a
singular point and that over some sequence $r=r_j\to 0+$ we have
$\mathcal{H}^{n}(\Lambda (u_r)\cap Q_1')\geq \delta>0$. By
Lemma~\ref{lem:blowups-H1alpha-conv}, taking a subsequence if
necessary, we may assume that $u_{r_j}$ converges locally uniformly to
a blowup $u_0$. We claim that
$$
\mathcal{H}^{n}(\Lambda (u_0)\cap Q_1')\geq \delta>0.
$$
Indeed, otherwise there exists an open set $U$ in $S_\infty'$ with
$\mathcal{H}^{n}(U)<\delta$ such that $\Lambda (u_0)\cap
\overline{Q_1'} \subset U$. Then for large $j$ we must have $\Lambda
(u_{r_j})\cap \overline{Q_1'} \subset U$, which is a contradiction,
since $\mathcal{H}^{n}(\Lambda (u_{r_j})\cap \overline{Q_1'})\geq
\delta > \mathcal{H}^{n}(U)$. Since $u_0=p_\kappa$ is a polynomial,
vanishing on a set of positive $\mathcal{H}^{n}$-measure on
$S_\infty'$, it follows that $u_0$ vanishes identically on
$S_\infty'$. But then, applying Holmgren's uniqueness theorem one more
time, we conclude that $u_0$ must vanish on $S_\infty$, which is a
contradiction. This completes the proof of the theorem.
\end{proof}

The implication (iii) $\Rightarrow$ (ii) in
Proposition~\ref{prop:char-sing-point} is equivalent to the following
Liouville-type result, which is the parabolic counterpart of
Lemma~1.3.3 in \cite{GP}.

\begin{lemma}\label{lem:2m-homogen} Let $v$ be a parabolically
  $\kappa$-homogeneous solution of the parabolic Signorini problem in
  $S_\infty$ with $\kappa=2m$ for $m\in\N$. Then $v$ is a caloric
  polynomial.
\end{lemma}
	
This, in turn, is a particular case of the following lemma, analogous
to Lemma~1.3.4 in \cite{GP} in the elliptic case, which stems from
Lemma~7.6 in \cite{Mon2}.

\begin{lemma}\label{lem:Monn-homogen-harm} Let $v\in
  W^{1,1}_{2,\loc}(S_\infty)$ be such that $\Delta v-\partial_t v\leq
  0$ in $S_\infty$ and $\Delta v-\partial_t v=0$ in $S_\infty\setminus
  S'_\infty$. If $v$ is parabolically $2m$-homogeneous, $m\in\N$, and
  has a polynomial growth at infinity, then $\Delta v-\partial_t v=0$
  in $S_\infty$.
\end{lemma}

\begin{proof} Consider $\mu:=\Delta v-\partial_tv$ in
  $\R^n\times(-\infty,0)$. By the assumptions, $\mu$ is a nonpositive
  measure, supported on $\{x_n=0\}\times(-\infty,0)$. We are going to
  show that in fact $\mu=0$. To this end, let $P(x,t)$ be a
  parabolically $2m$-homogeneous caloric polynomial, which is positive
  on $\{x_n=0\}\times(-\infty,0)$. For instance, one can take
$$
P(x,t)=\sum_{j=1}^{n-1}\Re(x_j+i
x_n)^{2m}+(-1)^m\sum_{k=0}^m\frac{m!}{(m-k)!(2k)!} x_n^{2k}t^{m-k}.
$$
It is straightforward to check that $P$ is caloric. Moreover on
$\{x_n=0\}\times(-\infty,0)$ we have
$$
P=\sum_{j=1}^{n-1} x_j^{2m}+(-t)^m,
$$
so it is positive on $\{x_n=0\}\times(-\infty,0)$.  Further, let
$\eta\in C^\infty_0((0,\infty))$, with $\eta\geq 0$, and define
$$
\Psi(x,t)=\eta(t) G(x,t)=\frac{\eta(t)}{(-4\pi t)^{n/2}}e^{|x|^2/4t}.
$$
Note that we have the following identity (similar to that of $
G(x,t)$)
$$
\nabla\Psi=\frac{x}{2t}\Psi.
$$
We have
\begin{align*}
  \langle\Delta v,\Psi P\rangle &=-\int_{-\infty}^0\int_{\R^n}\nabla  v\cdot\nabla(\Psi P)\,dx\,dt\\
  &=-\int_{-\infty}^0\int_{\R^n} [\Psi\nabla v\cdot\nabla P+P\nabla v\cdot\nabla\Psi]\,dx\,dt\\
  &=\int_{-\infty}^0\int_{\R^n}[\Psi v\Delta P+v\nabla\Psi\cdot\nabla P-P\nabla v\cdot\nabla\Psi]\,dx\,dt\\
  &=\int_{-\infty}^0\int_{\R^n}\left[v\Delta P+\frac{1}{2t}
    v(x\cdot\nabla P)-\frac{1}{2t}P(x\cdot\nabla
    v)\right]\Psi\,dx\,dt.
\end{align*}
We now use the identities $\Delta P-\partial_t P=0$, $x\cdot\nabla
P+2t\partial_t P=2m P$, $x\cdot\nabla v+2t\partial_t v=2m v$ to arrive
at
\begin{align*}
  \langle\Delta v,\Psi P\rangle
  &=\int_{-\infty}^0\int_{\R^n}\left[2m P v-P(x\cdot\nabla v)\right]\frac{\Psi}{2t}\,dx\,dt\\
  &=\int_{-\infty}^0\int_{\R^n}\partial_t v\Psi P\,dx\,dt\\
  &=\langle\partial_t v,\Psi P\rangle.
\end{align*}
Therefore, $\langle\mu,\Psi P\rangle=\langle\Delta v-\partial_t v,\Psi
P\rangle=0$. Since $\mu$ is a nonpositive measure, this implies that
actually $\mu=0$ and the proof is complete.
\end{proof}

\begin{definition}\label{def:Pk}
  Throughout the rest of the paper we denote by $\P_\kappa$,
  $\kappa=2m$, $m\in \N$, the class of $\kappa$-homogeneous harmonic
  polynomials described in Proposition~\ref{prop:char-sing-point}(ii).
\end{definition}

In the rest of this section we state our main results concerning the
singular set: $\kappa$-differentiability at singular points
(Theorem~\ref{thm:k-diff-sing-p}) and a structural theorem on the
singular set (Theorem \ref{thm:sing-points-nonzero}). The proofs will
require additional technical tools (two families of monotonicity
formulas) that we develop in the next section. The proofs themselves
will be given in Section~\ref{sec:struct-sing-set}.

\begin{theorem}[$\kappa$-differentiability at singular points]\label{thm:k-diff-sing-p} Let
  $u\in\S^f(S_1^+)$ with $|f(x,t)|\leq M\|(x,t)\|^{\ell-2}$ in
  $S_1^+$, $|\nabla f(x,t)|\leq L\|(x,t)\|^{\ell-3}$ in $Q_{1/2}^+$,
  $\ell\geq 3$, and $0\in\Sigma_\kappa(u)$ for $\kappa=2m<\ell$,
  $m\in\N$.  Then, there exists a nonzero $p_\kappa\in\P_\kappa$ such
  that
  $$
  u(x,t)=p_\kappa(x,t)+o(\|(x,t)\|^{\kappa}),\quad t\leq 0.
  $$
  Moreover, if $v\in\S^\phi(Q_1^+)$ with $\phi\in
  H^{\ell,\ell/2}(Q_1')$, $(x_0,t_0)\in\Sigma_\kappa(v)$ and
  $u^{(x_0,t_0)}_k$ is obtained as in Proposition~\ref{prop:uk-def}
  for $v^{(x_0,t_0)}=v(x_0+\cdot, t_0+\cdot)$, then in the Taylor
  expansion
  $$
  u^{(x_0,t_0)}_k(x,t)=p_\kappa^{(x_0,t_0)}(x,t)+o(\|(x,t)\|^{\kappa}),\quad
  t\leq 0,
  $$
  the mapping $(x_0,t_0)\mapsto p^{(x_0,t_0)}_\kappa$ from
  $\Sigma_\kappa(v)$ to $\P_\kappa$ is continuous.
\end{theorem}

\begin{remark}\label{rem:Pk-norm}Note that since $\P_\kappa$ is a convex subset of a finite-dimensional
  vector space, namely the space of all $\kappa$-homogeneous
  polynomials, its topology can be induced from any norm on that
  space.  For instance, the topology can be induced from the embedding
  of $\P_\kappa$ into $L_2(S_1^+, G)$, or even into $L_2(Q_1')$, since
  the elements of $\P_\kappa$ can be uniquely recovered from their
  restriction to the thin space.
\end{remark}
\begin{remark}\label{rem:taylor-cntrex} We want to emphasize here that
  the asymptotic development, as stated in
  Theorem~\ref{thm:k-diff-sing-p}, does not generally hold for
  $t>0$. Indeed, consider the following example. Let
  $u:\R^n\times\R\to \R$ be a continuous function such that
  \begin{itemize}
  \item $u(x,t)=-t-x_n^2/2$ for $x\in\R^n$ and $t\leq 0$.
  \item In $\{x_n\geq0,t\geq0\}$, $u$ solves the Dirichlet problem
    \begin{alignat*}{2}
      \Delta u-\partial_tu&=0,&\quad& x_n>0,t>0,\\
      u(x,0)&=-x_n^2, && x_n\geq 0,\\
      u(x',0,t)&=0 && t\geq 0.
    \end{alignat*}
  \item In $\{x_n\leq0,t\geq0\}$, we extend the function by even
    symmetry in $x_n$:
$$
u(x',x_n,t)=u(x',-x_n,t).
$$
\end{itemize}
It is easy to see that $u$ solves the parabolic Signorini problem with
zero obstacle and zero right-hand side in all of
$\R^n\times\R$. Moreover, $u$ is homogeneous of degree two and clearly
$0\in\Sigma_2(u)$. Now, if $p(x,t)=-t-x_n^2/2$, then $p\in\P_2$ and we
have the following equalities:
\begin{alignat*}{2}
  &u(x,t)=p(x,t),&\quad&\text{for }t\leq 0,\\
  &u(x',0,t)=0,\quad p(x',0,t)=-t &\quad&\text{for } t\geq 0.
\end{alignat*}
So for $t\geq 0$ the difference $u(x,t)-p(x,t)$ is not
$o(\|(x,t)\|^2)$, despite being zero for $t\leq 0$.

\end{remark}

We next want to state a structural theorem for the singular set,
similar to the ones in \cite{Ca} for the classical obstacle problem
and \cite{GP} for the thin obstacle problem. In order to do so, we
define the spatial dimension $d=d_\kappa^{(x_0,t_0)}$ of
$\Sigma_\kappa(v)$ at a given point $(x_0,t_0)$ based on the
polynomial $p_\kappa^{(x_0,t_0)}$.

\begin{definition}[Spatial dimension of the singular
  set]\label{def:dim-sing-point}  For a singular point
  $(x_0,t_0)\in\Sigma_\kappa(v)$ we define
  \begin{multline*}
    d_\kappa^{(x_0,t_0)}=\dim\{\xi\in\R^{n-1}\mid \xi\cdot
    \nabla_{x'}\partial_{x'}^{\alpha'}\partial_t^jp_\kappa^{(x_0,t_0)}=0\\
    \text{for any $\alpha'=(\alpha_1,\ldots,\alpha_{n-1})$ and $j\geq
      0$ such that } |\alpha'|+2j=\kappa-1\},
  \end{multline*}
  which we call the \emph{spatial dimension} of $\Sigma_\kappa(v)$ at
  $(x_0,t_0)$. Clearly, $d_{\kappa}^{(x_0,t_0)}$ is an integer between
  $0$ and $n-1$. Then, for any $d=0,1,\ldots,n-1$ define
$$
\Sigma_\kappa^d(v):=\{(x_0,t_0)\in\Sigma_\kappa(v) \mid
d^{(x_0,t_0)}_\kappa=d\}.
$$
\end{definition}

The case $d=n-1$ deserves a special attention.

\begin{lemma}[Time-like singular points]\label{lem:timelike-blowup} Let $(x_0,t_0)\in
  \Sigma^{n-1}_\kappa(v)$, $\kappa=2m<\ell$. Then
$$
p_{\kappa}^{(x_0,t_0)}(x,t)=C(-1)^m\sum_{k=0}^m
\frac{t^{m-k}}{(m-k)!}\frac{x_n^{2k}}{2k!},
$$
for some positive constant $C$. In other words,
$p_{\kappa}^{(x_0,t_0)}$ depends only on $x_n$ and $t$ and is unique
up to a multiplicative factor. We call such singular points
\emph{time-like}.
\end{lemma}
\begin{proof} The condition $d_\kappa^{(x_0,t_0)}=n-1$ is equivalent
  to the following property of $p_\kappa=p_\kappa^{(x_0,t_0)}$:
$$
\nabla_{x'}\partial_{x'}^{\alpha'}\partial_t^j p_\kappa=0,
$$
for any multi-index $\alpha'=(\alpha_1,\ldots,\alpha_{n-1})$ and $j$
such that $|\alpha'|+2j=\kappa-1$. It is easy to see that this is
equivalent to vanishing of $\partial_{x_i}p_\kappa$ on $S_\infty'$ for
$i=1,\ldots,n-1$. On the other hand, $\partial_{x_i}p_\kappa$ is
caloric in $S_\infty$ and is also even symmetric in $x_n$ variable,
implying that $\partial_{x_n}\partial_{x_i}p_\kappa=0$ on
$S_\infty'$. Then, by Holmgren's uniqueness theorem
$\partial_{x_i}p_\kappa$ is identically $0$ in $S_\infty$, implying
that $p_\kappa(x,t)$ depends only on $x_n$ and $t$. The homogeneity of
$p_\kappa$ implies that we can write it in the form
$$
p_\kappa(x,t)=\sum_{k=0}^{m}
a_k\frac{t^{m-k}}{(m-k)!}\frac{x_n^{2k}}{(2k)!}.
$$
The rest of the proof is then elementary.
\end{proof}

\begin{definition}[Space-like and time-like manifolds] We say that a
  $(d+1)$-dimensional manifold $\mathcal{S}\subset\R^{n-1}\times\R$,
  $d=0,\ldots, n-2$, is \emph{space-like} of class $C^{1,0}$, if
  locally, after a rotation of coordinate axes in $\R^{n-1}$ one can
  represent it as a graph
$$
(x_{d+1},\ldots,x_{n-1})=g(x_1,\ldots,x_d,t),
$$
where $g$ is of class $C^{1,0}$, i.e., $g$ and $\partial_{x_i}g$,
$i=1,\ldots,d$ are continuous.

We say that $(n-1)$-dimensional manifold
$\mathcal{S}\subset\R^{n-1}\times\R$ is \emph{time-like} of class
$C^1$ if it can be represented locally as
$$
t=g(x_1,\ldots,x_{n-1}),
$$
where $g$ is of class $C^1$.

\end{definition}
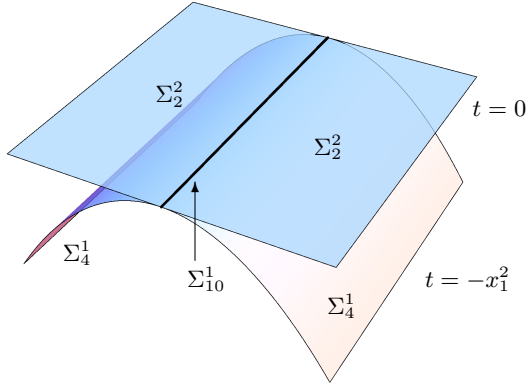
\begin{figure}[t]
  \begin{picture}(150,150)
    \put(0,0){\includegraphics[height=150pt]{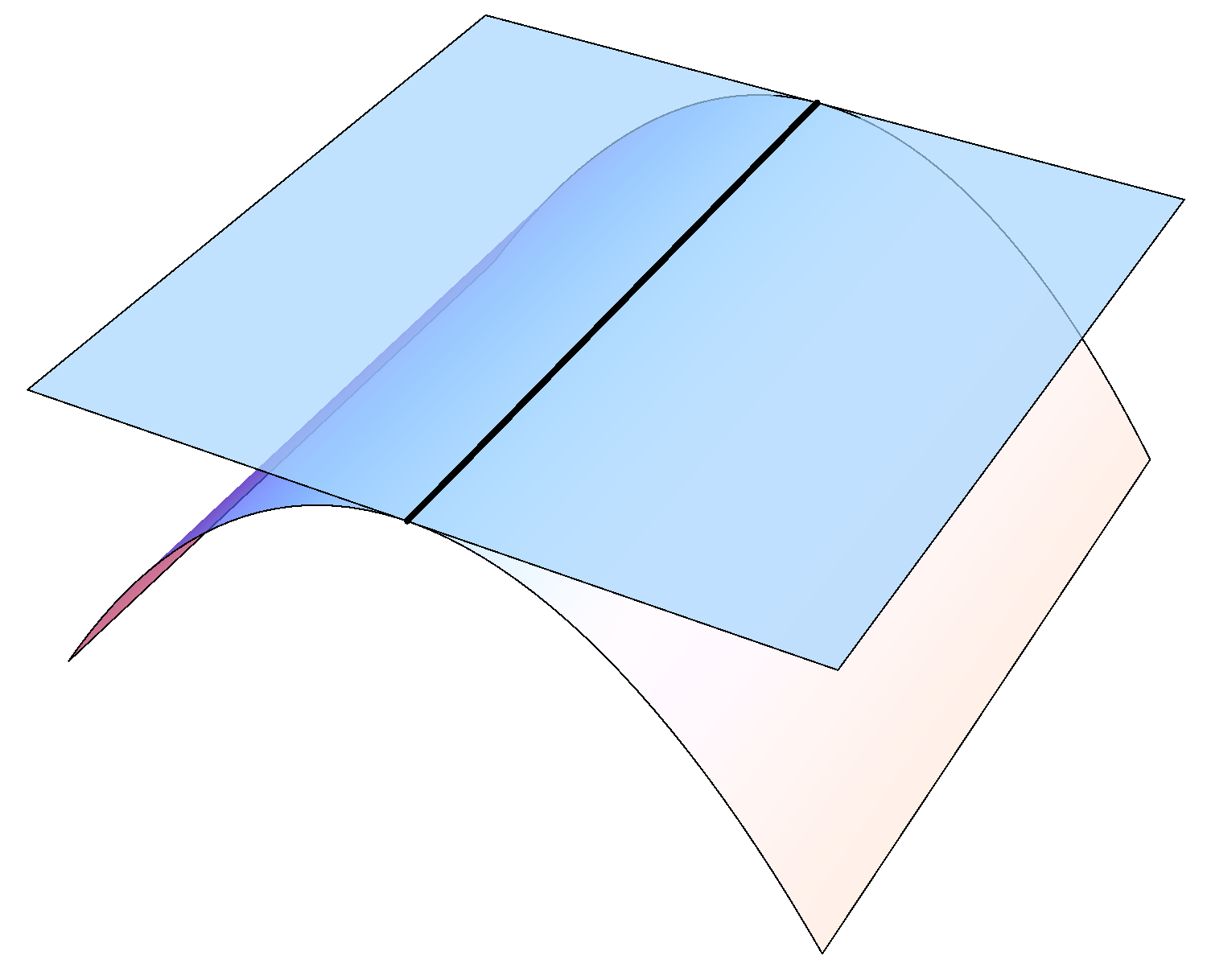}}
    \put(120,90){\small $\Sigma_2^2$} \put(60,110){\small
      $\Sigma_2^2$} \put(72,40){\small $\Sigma_{10}^1$}
    \put(75,50){\vector(0,2){30}} \put(125,30){\small $\Sigma_4^1$}
    \put(25,50){\small $\Sigma_4^1$} \put(162,40){\small $t=-x_1^2$}
    \put(180,105){\small $t=0$}
  \end{picture}
  \caption{Structure of the singular set $\Sigma(v)\subset \R^2\times
    (-\infty,0]$ for the solution $v$ with
    $v(x_1,x_2,0,t)=-t(t+x_1^2)^4$, $t\leq 0$ with zero thin
    obstacle. Note that the points on $\Sigma_4^1$ and $\Sigma_{10}^1$
    are space-like, and the points on $\Sigma_2^2$ are time-like.}
  \label{fig:sing-set-ex}
\end{figure}

\begin{theorem}[Structure of the singular set]\label{thm:sing-points-nonzero} Let
  $v\in\S_\phi(Q_1^+)$ with $\phi\in H^{\ell,\ell/2}(Q_1')$, $\ell\geq
  3$. Then, for any $\kappa=2m<\ell$, $m\in\N$, we have
  $\Gamma_\kappa(v)=\Sigma_\kappa(v)$. Moreover, for every
  $d=0,1,\ldots, n-2$, the set $\Sigma_\kappa^d(v)$ is contained in a
  countable union of $(d+1)$-dimensional space-like $C^{1,0}$
  manifolds and $\Sigma_{\kappa}^{n-1}(v)$ is contained in a countable
  union of $(n-1)$-dimensional time-like $C^1$ manifolds.
\end{theorem}

For a small illustration, see Fig.~\ref{fig:sing-set-ex}.

\section{Weiss and Monneau type monotonicity formulas}
\label{sec:weiss-monneau-type}

In this section we construct two families of monotonicity formulas
that will play a crucial role in the study of the singular set. They
generalize the corresponding formulas in \cite{GP} in the study of the
elliptic thin obstacle problem.

The first family of monotonicity formulas goes back to the work of
Weiss \cite{Wei1} in the elliptic case and \cite{Wei2} in the
parabolic case; see also \cite{CPS}.

\begin{theorem}[Weiss-type monotonicity formula]\label{thm:Weiss}
  Let $u\in\S^f(S_1^+)$ with $|f(x,t)|\leq M\|(x,t)\|^{\ell-2}$ in
  $S_1^+$, $\ell\geq 2$.  For any $\kappa\in(0,\ell)$, define the
  Weiss energy functional
  \begin{align*}
    W_u^\kappa(r) :=&
    \frac{1}{r^{2\kappa+2}}\int_{S_r^+}\Big(|t||\nabla
    u|^2-\frac\kappa2 u^2\Big) G\\
    =&\frac{1}{r^{2\kappa}}\Big(I_u(r)-\frac{\kappa}{2}H_u(r)\Big),\quad0<r<1.
  \end{align*}
  Then, for any $\sigma<\ell-\kappa$ there exists $C>0$ depending only
  on $\sigma$, $\ell$, $M$, and $n$, such that
  \[
  \frac{d}{dr}W_u^\kappa\geq-C r^{2\sigma-1},\quad\text{for a.e.\
  }r\in (0,1).
  \]
  In particular, the function
$$
r\mapsto W_u^\kappa(r)+C{r^{2\sigma}}
$$
is monotonically nondecreasing for $r\in(0,1)$.
\end{theorem}

\begin{proof}
  The proof is by direct computation using
  Proposition~\ref{prop:diff-form-signor}. We have
  \begin{align*}
    r^{2\kappa+1}\frac{d}{dr}W_u^\kappa(r) &= r I'_u(r)-2\kappa I_u(r)-\frac{\kappa}{2}r H'_u(r)+\kappa^2 H_u(r)\\
    &\geq r\Big(\frac{1}{r^3}\int_{S_r^+}(Zu)^2 G\,dx\,dt+\frac{2}{r^3}\int_{S_r^+} t (Zu) f G\,dx\,dt\Big)-2\kappa I_u(r)\\
    &\qquad-\frac{\kappa}{2} r\Big(\frac{4}{r} I_u(r)-\frac{4}{r^3}\int_{S_r^+} t u f G\,dx\,dt\Big)+\kappa^2 H_u(r)\\
    &=\frac{1}{r^2}\int_{S_r^+}(Zu)^2
    G\,dx\,dt+\frac{2}{r^2}\int_{S_r^+} t
    (Zu) f G\,dx\,dt\\
    &\qquad+\frac{2\kappa}{r^2}\int_{S_r^+} t u f G\,dx\,dt-4\kappa I_u(r)+\kappa^2 H_u(r)\\
    &=\frac{1}{r^2}\int_{S_r^+}(Zu)^2 G\,dx\,dt+\frac{2}{r^2}\int_{S_r^+} t (Zu) f G\,dx\,dt+\frac{2\kappa}{r^2}\int_{S_r^+} t u f G\,dx\,dt\\
    &\qquad-4\kappa\Big(\frac{1}{2r^2}\int_{S_r^+}u(Zu) G\,dx\,dt+\frac{1}{r^2}\int_{S_r^+} tuf G\,dx\,dt\Big)+\kappa^2\frac{1}{r^2}\int_{S_r^+}u^2\,dx\,dt\\
    &=\frac{1}{r^2}\int_{S_r^+}(Zu+tf-\kappa u)^2
    G\,dx\,dt-\frac{1}{r^2}\int_{S_r^+} t^2 f^2 G\,dx\,dt.
  \end{align*}
  Hence, using the integral estimate on $f$ as at the beginning of
  Section~\ref{sec:class-free-bound}, we obtain
  \begin{align*}
    \frac{d}{dr}W_u^\kappa(r) &\geq-\frac1{r^{2\kappa+3}}\int_{S_r^+} t^2 f^2 G\,dx\,dt\\
    &\geq-C\frac{r^{2\ell+2}}{r^{2\kappa+3}}\\
    &\geq -Cr^{2\sigma-1},
  \end{align*}
  which yields the desired conclusion.
\end{proof}

Note that in Theorem~\ref{thm:Weiss}, we do not require $0\in
\Gamma_\kappa^{(\ell)}(u)$. However, if we do so, then we will have
the following fact.
\begin{lemma}\label{lem:W0+} Let $u$ be as in Theorem~\ref{thm:Weiss},
  and assume additionally that $0\in\Gamma^{(\ell)}_\kappa(u)$,
  $\kappa<\ell$. Then,
  \[
  W^\kappa_u(0+)=0.
  \]
\end{lemma}
The proof will require the following growth estimate, which we will
use a few more times in the remaining part of the paper.

\begin{lemma}[Growth estimate]\label{lem:HuMr2kappa} Let $u\in \S^f(S_1^+)$ with
  $|f(x,t)|\leq M\|(x,t)\|^{\ell-2}$, $\ell\geq2$, and
  $0\in\Gamma_\kappa^{(\ell)}(u)$ with $\kappa<\ell$ and let
  $\sigma<\ell-\kappa$. Then
$$
H_r(u)\leq C\big( \|u\|_{L_2(S_1^+, G)}^2+M^2\big) r^{2\kappa},\quad
0<r<1,
$$
with $C$ depending only on $\sigma$, $\ell$, $n$.
\end{lemma}
\begin{proof}
  Take $\mu(r)=\big(\|u\|_{L_2(S_1^+, G)}^2+M^2\big)r^{2\ell-2\sigma}$
  and proceed as in the proof of Lemma~\ref{lem:HuMr3}. We omit the
  details.
\end{proof}

We can now prove Lemma~\ref{lem:W0+}.

\begin{proof}[Proof of Lemma~\ref{lem:W0+}]Since $\kappa < \ell$, by
  Lemma~\ref{lem:Hu-mu} we have
  \[
  \lim_{r\to 0+}\frac{I_u(r)}{H_u(r)} =\frac{\kappa}{2}.
  \]
  Further, by Lemma~\ref{lem:HuMr2kappa} above, we will have
  \[
  H_u(r)\leq Cr^{2\kappa}.
  \]
  Hence, we obtain
  \[
  \lim_{r\to 0+} W_u^\kappa(r) =\lim_{r\to
    0}\frac{H_u(r)}{r^{2\kappa}}\Big(\frac{I_u(r)}{H_u(r)}-\frac{\kappa}{2}\Big)
  = 0.\qedhere
  \]
\end{proof}

The next monotonicity formula is specifically tailored for singular
points. It goes back to the paper of Monneau \cite{Mon} for the
classical obstacle problem. The theorem below is the parabolic
counterpart of the Monneau-type monotonicity formula in \cite{GP}.

\begin{theorem}[Monneau-type monotonicity formula]\label{thm:monneau}
  Let $u\in\S^f(S_1^+)$ with $|f(x,t)|\leq M\|(x,t)\|^{\ell-2}$ in
  $S_1^+$, $|\nabla f(x,t)|\leq L\|(x,t)\|^{\ell-3}$ in $Q_{1/2}^+$,
  $\ell\geq 3$. Suppose that $0\in \Sigma_{\kappa}(u)$ with
  $\kappa=2m<\ell$, $m\in\N$. Further, let $p_\kappa$ be any
  parabolically $\kappa$-homogeneous caloric polynomial from class
  $\P_\kappa$ as in Definition~\ref{def:Pk}. For any such $p_\kappa$,
  define Monneau's functional as
  \begin{align*}
    M^\kappa_{u,p_\kappa}(r) :&=
    \frac{1}{r^{2\kappa+2}}\int_{S_r^+}(u-p_\kappa)^2 G,\quad 0<r<1,\\
    &=\frac{H_w(r)}{r^{2\kappa}},\quad\text{where }w=u-p_\kappa.
  \end{align*}
  Then, for any $\sigma<\ell-\kappa$ there exists a constant $C$,
  depending only on $\sigma$, $\ell$, $M$, and $n$, such that
  \[
  \frac{d}{dr}M^\kappa_{u,p_\kappa}(r)\geq-C\Big(1+\|u\|_{L_2(S_1^+,
    G)}+\|p_\kappa\|_{L_2(S_1^+, G)} \Big) r^{\sigma-1}.
  \]
  In particular, the function
  \[
  r\mapsto M_{u,p_\kappa}^\kappa(r)+C r^\sigma
  \]
  is monotonically nondecreasing for $r\in (0,1)$ for a constant $C$
  depending $\sigma$, $\ell$, $M$, $n$, $\|u\|_{L_2(S_1^+, G)}$, and
  $\|p_\kappa\|_{L_2(S_1^+, G)}$.
\end{theorem}

\begin{proof} First note that $W_{p_\kappa}^\kappa(r)$ is constant in
  $r$, which follows easily from the homogeneity of $p_\kappa$. Then,
  since also $0\in\Gamma_\kappa^{(\ell)}(p_\kappa)$, by
  Lemma~\ref{lem:W0+} we have $W^\kappa_{p_\kappa}(0+)=0$, implying
  that
  \[
  W_{p_\kappa}^\kappa(r) \equiv 0.
  \]
  Therefore, integrating by parts, we have
  \begin{align*}
    W_u^\kappa(r) &= W_u^\kappa(r)-W_{p_\kappa}^\kappa(r)\\
    &= W_w^\kappa(r)-\frac{2}{r^{2\kappa+2}}\int_{S_r^+} t \nabla w \nabla p_\kappa G-\frac{\kappa}{r^{2\kappa+2}}\int_{S_r^+} wp_\kappa G\\
    &= W_w^\kappa(r)+\frac{2}{r^{2\kappa+2}}\int_{S_r^+} tw(\Delta p_\kappa G+\nabla p_\kappa\nabla G)-\frac{\kappa}{r^{2\kappa+2}}\int_{S_r^+}wp_\kappa G\\
    &=W_w^\kappa(r)+\frac1{r^{2\kappa+2}}\int_{S_r^+}2tw\Big(\partial_t
    p_\kappa+\nabla
    p_\kappa \frac{x}{2t}\Big) G-\frac{\kappa}{r^{2\kappa+2}}\int_{S_r^+}wp_\kappa G\\
    &= W_w^\kappa(r)+\frac{1}{r^{2\kappa+2}}\int_{S_r^+}w(Zp_\kappa-\kappa p_\kappa) G\\
    &= W_w^\kappa(r).
  \end{align*}
  Next, we want to compute the derivative of
  $M_{u,p_\kappa}^\kappa$. With this objective in mind, we remark that
  we can compute the derivative of $H_w$ by a formula similar to that
  in Lemma~\ref{lem:differentiation-formulae-smooth}:
$$
H'_w(r)=\frac{4}{r} I_w(r)-\frac{4}{r^3}\int_{S_r^+} twf
G-\frac{4}{r^3}\int_{S_r'} tw_{x_n}w G,
$$
for a.e.\ $r\in(0,1)$. Indeed, we first note that
Lemma~\ref{lem:differentiation-formulae-smooth} holds for smooth
functions with a polynomial growth at infinity, since the same spatial
integration by parts used to derive
Lemma~\ref{lem:differentiation-formulae-smooth} is still valid. Then,
as in Proposition~\ref{prop:diff-form-signor} approximate $u$ by the
solutions $u^\epsilon$ of the penalized problem, apply
Lemma~\ref{lem:differentiation-formulae-smooth} for
$(H^\delta_{w^\epsilon})'(r)$, where
$w^\epsilon=u^\epsilon\zeta-p_\kappa$, and then pass to the limit as
$\epsilon\to 0$ and $\delta\to 0$ as in
Proposition~\ref{prop:diff-form-signor}. We thus arrive at the formula
for $H'_w(r)$ given above.

We therefore can write
\begin{align*}
  \frac{d}{dr}M_{u,p_\kappa}^\kappa(r) &= \frac{H'_w(r)}{r^{2\kappa}}-\frac{2\kappa}{r^{2\kappa+1}} H_w(r)\\
  &= \frac{4}{r^{2\kappa+1}} I_w(r)-\frac{4}{r^{2\kappa+3}}\int_{S_r^+}twf G-\frac{4}{r^{2\kappa+3}}\int_{S'_r} tw_{x_n}w G-\frac{2\kappa}{r^{2\kappa+1}} H_w(r)\\
  &= \frac{4}{r} W_w^\kappa(r)-\frac{4}{r^{2\kappa+3}}\int_{S_r^+}twf
  G+\frac{4}{r^{2\kappa+3}}\int_{S'_r} tu_{x_n}p_\kappa G.
\end{align*}
We now estimate each term on the right hand side.

To estimate the first term we use that, in view of
Theorem~\ref{thm:Weiss}, for an appropriately chosen $C$, the function
$W_u^\kappa(r)+Cr^{2\sigma}$ is nondecreasing, and therefore
$$
W_w^\kappa(r)=W_u^\kappa(r)\geq
W_u^\kappa(0+)-Cr^{2\sigma}=-Cr^{2\sigma},
$$
where in the last equality we have used Lemma~\ref{lem:W0+}.

The second term can be estimated by the Cauchy-Schwarz inequality, the
integral estimate on $f$ at the beginning of
Section~\ref{sec:class-free-bound}, and Lemma~\ref{lem:HuMr2kappa}:
\begin{align*}
  \frac{1}{r^{2\kappa+3}}\int_{S_r^+}twf G &\leq \frac{1}{r^{2\kappa+3}}\Big(\int_{S_r^+}w^2 G\Big)^{1/2}\Big(\int_{S_r^+} t^2 f^2 G\Big)^{1/2}\\
  &\leq \frac{C r}{r^{2\kappa+3}}
  \big(H_u(r)^{1/2}+H_{p_\kappa}(r)^{1/2}\big) C r^{\ell+1}\\
  &\leq C\Big(1+\|u\|_{L_2(S_1^+, G)}+\|p_\kappa\|_{L_2(S_1^+, G)}\Big) r^{\ell-\kappa-1}\\
  &\leq C\Big(1+\|u\|_{L_2(S_1^+, G)}+\|p_\kappa\|_{L_2(S_1^+,
    G)}\Big) r^{\sigma-1}.
\end{align*}
For the last term, just notice that on $S'_1$,
\[
t\leq 0,\quad u_{x_n}\leq0,\quad p_\kappa\geq 0
\]
and thus,
\[
\int_{S'_r} t u_{x_n}p_\kappa G\geq 0.
\]
Combining all these estimates, we obtain,
\begin{align*}
  \frac{d}{dr}M_{u,p_\kappa}^\kappa(r) &\geq-Cr^{2\sigma-1}-C\Big(1+\|u\|_{L_2(S_1, G)}+\|p_\kappa\|_{L_2(S_1^+)} \Big) r^{\sigma-1}\\
  &\geq -C\Big(1+\|u\|_{L_2(S_1^+, G)}+\|p_\kappa\|_{L_2(S_1^+, G)}
  \Big) r^{\sigma-1},
\end{align*}
which is the desired conclusion.
\end{proof}

\section{Structure of the singular set}
\label{sec:struct-sing-set}

In this section, we prove our main results on the singular set, stated
at the end of Section~\ref{sec:struct-sing-set} as
Theorems~\ref{thm:k-diff-sing-p} and \ref{thm:sing-points-nonzero}.

We start by remarking that in the following proofs it will be more
convenient to work with a slightly different type of rescalings and
blowups, than the ones used up to now. Namely, we will work with the
following $\kappa$-\emph{homogeneous rescalings}
$$
u^{(\kappa)}_r(x,t):=\frac{u(rx,r^2t)}{r^\kappa}
$$
and their limits as $r\to 0+$. The next lemma shows the viability of
this approach.

\begin{lemma}[Nondegeneracy at singular
  points]\label{lem:nondeg-sing-nonzero} Let $u\in\S^f(S_1^+)$ with
  $|f(x,t)|\leq M \|(x,t)\|^{\ell-2}$ in $S_1^+$, $|\nabla f(x,t)|\leq
  L\|(x,t)\|^{\ell-3}$ in $Q_{1/2}^+$, $\ell\geq 3$, and
  $0\in\Sigma_\kappa (u)$ for $\kappa<\ell$. Then, there exists
  $c=c_u>0$ such that
  $$
  H_u(r)\geq c\,r^{2\kappa},\quad\text{for any } 0<r<1.
  $$
\end{lemma}
\begin{proof} Assume the contrary. Then for a sequence $r=r_j\to 0$
  one has
  $$
  H_u(r)=\frac1{r^2}\int_{S_1^+}u^2 G=o(r^{2\kappa}).
  $$
  Since $0$ is a singular point, by
  Proposition~\ref{prop:char-sing-point}, we have that, over a
  subsequence,
  $$
  u_r(x,t)=\frac{u(rx,r^2t)}{H_u(r)^{1/2}}\to q_\kappa(x,t),
  $$
  as described in Theorem~\ref{thm:exist-homogen-blowups}, for some
  nonzero $q_\kappa\in\P_\kappa$.  Now, for such $q_\kappa$ we apply
  Theorem~\ref{thm:monneau} to $M^\kappa_{u,q_\kappa}(r)$. From the
  assumption on the growth of $u$ is is easy to recognize that
  $$
  M^\kappa_{u,q_\kappa}(0+)=\int_{S_1^+}q_\kappa^2
  G=\frac{1}{r^{2\kappa+2}} \int_{S_r^+} q_\kappa^2 G.
  $$
  Therefore, using the monotonicity of
  $M^\kappa_{u,q_\kappa}(r)+C\,r^\sigma$ (see
  Theorem~\ref{thm:monneau}) for an appropriately chosen $C>0$, we
  will have that
  $$
  C\,r^\sigma+\frac{1}{r^{2\kappa+2}} \int_{S_r^+} (u-q_\kappa)^2
  G\geq \frac{1}{r^{2\kappa+2}} \int_{S_r^+} q_\kappa^2 G,
  $$
  or equivalently
  $$
  \frac{1}{r^{2\kappa+2}}\int_{S_r^+} (u^2-2 u q_\kappa) G\geq
  -C\,r^\sigma.
  $$
  After rescaling, we obtain
  $$
  \frac{1}{r^{2\kappa}}\int_{S_1^+} (H_u(r)u_r^2-2H_u(r)^{1/2}
  r^{\kappa} u_r q_\kappa) G\geq - C\,r^\sigma,
  $$
  which can be rewritten as
  $$
  \int_{S_1^+} \Big(\frac{H_u(r)^{1/2}}{r^\kappa} u_r^2-2u_r
  q_\kappa\Big) G\geq -C\,\frac{r^{\kappa+\sigma}}{H_u(r)^{1/2}}.
  $$
  Now from the arguments in the proof of Lemma~\ref{lem:Hu-est}, we
  have $H_u(r)> c\,r^{2\kappa'}$, for any $\kappa'>\kappa$, for
  sufficiently small $r$. Hence, choosing $\kappa'<\kappa+\sigma$, we
  will have that $r^{\kappa+\sigma}/H_u(r)^{1/2}\to 0$. Thus, passing
  to the limit over $r=r_j\to 0$, we will arrive at
  $$
  -\int_{S_1^+} q_\kappa^2\geq 0,
  $$
  which is a contradiction, since $q_\kappa\not\equiv 0$.
\end{proof}

One consequence of the nondegeneracy at singular points is that the
singular set has a topological type $F_\sigma$; this will be important
in our application of Whitney's extension theorem in the proof of
Theorem~\ref{thm:sing-points-nonzero}.

\begin{lemma}[$\Sigma_\kappa(v)$ is
  $F_\sigma$]\label{lem:sing-set-F-sigma-nonzero} For any $v\in
  \S_\phi(Q_1^+)$ with $\phi\in H^{\ell,\ell/2}(Q_1')$, the set
  $\Sigma_\kappa(v)$ with $\kappa=2m<\ell$, $m\in \N$, is of
  topological type $F_\sigma$, i.e., it is a union of countably many
  closed sets.
\end{lemma}
\begin{proof} For $j\in\N$, $j\geq 2$, let $E_j$ be the set of points
  $(x_0,t_0)\in \Sigma_\kappa(v)\cap \overline{Q_{1-1/j}}$ satisfying
  \begin{equation}\label{eq:Ej}
    \begin{aligned}&\tfrac1j\, r^{2\kappa}\leq H_{u^{(x_0,t_0)}_k}(r) \leq
      j\,r^{2\kappa}\\&\qquad\text{for every $0<r<\min\{1-|x_0|,\sqrt{1+t_0}\}$}.
    \end{aligned}
  \end{equation}
  By Lemmas~\ref{lem:HuMr2kappa} and \ref{lem:nondeg-sing-nonzero}, we
  have that
$$
\Sigma_\kappa(v)\subset \bigcup_{j=2}^\infty E_j.
$$
So to complete the proof, it is enough to show that $E_j$ are
closed. Indeed, if $(x_0,t_0)\in \overline E_j$, then it readily
satisfies \eqref{eq:Ej}. That, in turn, implies that $(x_0,t_0)\in
\Gamma_\kappa^{(\ell)}(v)$ and by
Proposition~\ref{prop:char-sing-point} that $(x_0,t_0)\in
\Sigma_\kappa(v)$. Hence $(x_0,t_0)\in E_j$. This completes the proof.
\end{proof}

We next show the existence and uniqueness of blowups with respect to
$\kappa$-homogeneous rescalings.

\begin{theorem}[Uniqueness of $\kappa$-homogeneous blowups at singular
  points]\label{thm:uniq-blowup-sing-nonzero} Let $u\in\S^f(S_1^+)$
  with $|f(x,t)|\leq M\|(x,t)\|^{\ell-2}$ in $S_1^+$, $|\nabla
  f(x,t)|\leq L\|(x,t)\|^{\ell-3}$ in $Q_{1/2}^+$, $\ell\geq 3$, and
  $0\in\Sigma_\kappa(u)$ with $\kappa<\ell$. Then there exists a
  unique nonzero $p_\kappa\in\P_\kappa$ such that
  $$
  u_r^{(\kappa)}(x,t):=\frac{u(rx,r^2t)}{r^\kappa}\to p_\kappa(x,t).
  $$
\end{theorem}

\begin{proof} By Lemmas~\ref{lem:HuMr2kappa} and
  \ref{lem:nondeg-sing-nonzero}, there are positive constants $c$ and
  $C$ such that
$$
c\,r^{\kappa}\leq H_u(r)^{1/2}\leq C\, r^\kappa,\quad 0<r<1.
$$
This implies that any limit $u_0$ of the $\kappa$-homogeneous
rescalings $u_r^{(\kappa)}$ over any sequence $r=r_j\to 0+$ is just a
positive multiple of a limit of regular rescalings
$u_r(x,t)=u(rx,r^2t)/H_u(r)^{1/2}$, since over a subsequence
$H_u(r)^{1/2}/r^\kappa$ converges to a positive number.  Since the
limits of rescalings $u_r$ are polynomials of class $\P_\kappa$, we
obtain also that $u_0\in\P_\kappa$.

To show the uniqueness of $u_0$, we apply Theorem~\ref{thm:monneau}
with $p_\kappa=u_0$. This result implies that the limit
$M^\kappa_{u,u_0}(0+)$ exists and can be computed by
  $$
  M^\kappa_{u,u_0}(0+)=\lim_{r_j\to 0+}
  M^\kappa_{u,u_0}(r_j)=\lim_{j\to\infty} \int_{S_1^+}
  (u^{(\kappa)}_{r_j}-u_0)^2 G=0.
  $$
  The latter equality is a consequence of
  Theorem~\ref{thm:exist-homogen-blowups}, and of the argument at the
  beginning of this proof.  In particular, we obtain that
  $$
  \int_{S_1^+} (u^{(\kappa)}_r-u_0)^2 G= M^\kappa_{u,u_0}(r)\to 0
  $$
  as $r\to 0+$ (not just over $r=r_j\to 0+$!). Thus, if $u_0'$ is a
  limit of $u_r^{(\kappa)}$ over another sequence $r=r_j'\to 0$, we
  conclude that
  $$
  \int_{S_1^+} (u_0'-u_0)^2 G=0.
  $$
  This implies that $u_0'=u_0$ and completes the proof of the theorem.
\end{proof}

\begin{theorem}[Continuous dependence of blowup]
  \label{thm:cont-dep-blowup-nonzero} Let $v\in\S_\phi(Q_1^+)$ with
  $\phi\in H^{\ell,\ell/2}(Q_1')$, $\ell\geq 3$, and $\kappa=2m<\ell$,
  $m\in \N$. For $(x_0,t_0)\in\Sigma_\kappa(v)$, let $u_k^{(x_0,t_0)}$
  be as constructed in Proposition~\ref{prop:uk-def} for the function
  $v^{(x_0,t_0)}=v(x_0+\cdot,t_0+.)$, and denote by
  $p^{(x_0,t_0)}_\kappa$ the $\kappa$-homogeneous blowup of
  $u^{(x_0,t_0)}_k$ at $(x_0,t_0)$ as in
  Theorem~\ref{thm:uniq-blowup-sing-nonzero}, so that
  $$
  u^{(x_0,t_0)}_k(x,t)=p^{(x_0,t_0)}_\kappa(x,t)+o(\|(x,t)\|^{\kappa}).
  $$
  Then, the mapping $(x_0,t_0)\mapsto p_\kappa^{(x_0,t_0)}$ from
  $\Sigma_\kappa(v)$ to $\P_\kappa$ is continuous. Moreover, for any
  compact subset $K$ of $\Sigma_\kappa(v)$ there exists a modulus of
  continuity $\sigma=\sigma^K$, $\sigma(0+)=0$ such that
  $$
  |u^{x_0,t_0}_k(x,t)-p^{(x_0,t_0)}_\kappa(x,t)|\leq \sigma
  (\|(x,t)\|)\|(x,t)\|^{\kappa},\quad t\leq 0.
  $$
  for any $(x_0,t_0)\in K$.

\end{theorem}

\begin{proof} Given $(x_0,t_0)\in \Sigma_\kappa(v)$ and $\epsilon>0$
  fix $r_\epsilon=r_\epsilon(x_0,t_0)>0$ such that
  $$
  M^\kappa_{u^{(x_0,t_0)}_k,
    p_\kappa^{(x_0,t_0)}}(r_\epsilon)=\frac{1}{r_\epsilon^{2\kappa+2}}\int_{S_{r_\epsilon}^+}
  \big(u^{(x_0,t_0)}_k- p_\kappa^{(x_0,t_0)}\big)^2 G<\epsilon.
  $$
  Then, there exists $\delta_\epsilon=\delta_\epsilon(x_0,t_0)$ such
  that if $(x_0',t_0')\in\Sigma_\kappa(u)$ and
  $\|(x_0'-x_0,t_0'-t_0)\|<\delta_\epsilon$ one has
  $$
  M^\kappa_{u^{(x_0',t_0')}_k,
    p_\kappa^{(x_0,t_0)}}(r_\epsilon)=\frac{1}{r_\epsilon^{2\kappa+2}}\int_{S_{r_\epsilon}^+}
  \big(u^{(x_0',t_0')}_k- p_\kappa^{(x_0,t_0)}\big)^2 G<2\epsilon.
  $$
  This follows from the continuous dependence of $u^{(x_0,t_0)}_k$ on
  $(x_0,t_0)\in\Gamma(v)$ in $L_2(S_{r_\epsilon}^+, G)$ norm, which is
  a consequence of $H^{\ell,\ell/2}$ regularity of the thin obstacle
  $\phi$.  Then, from Theorem~\ref{thm:monneau}, we will have that
  $$
  M^\kappa_{u^{(x_0',t_0')}_k,
    p_\kappa^{(x_0,t_0)}}(r_\epsilon)=\frac{1}{r_\epsilon^{2\kappa+2}}\int_{S_{r_\epsilon}^+}
  \big(u^{(x_0',t_0')}_k- p_\kappa^{(x_0,t_0)}\big)^2 G < 2\epsilon+
  C\,r_\epsilon^\sigma,\quad 0<r<r_\epsilon,
  $$
  for a constant $C=C(x_0,t_0)$ depending on $L_2$ norms of
  $u^{(x_0',t_0')}_k$ and $p_\kappa^{(x_0,t_0)}$, which can be made
  uniform for $(x_0',t_0')$ in a small neighborhood of $(x_0,t_0)$.
  Letting $r\to 0$ we will therefore obtain
  $$
  M^\kappa_{u^{(x_0',t_0')}_k,
    p_\kappa^{(x_0,t_0)}}(0+)=\int_{S_{1}^+}
  \big(p^{(x_0',t_0')}_\kappa- p_\kappa^{(x_0,t_0)}\big)^2 G\leq
  2\epsilon+C\,r_\epsilon^\sigma.
  $$
  This shows the first part of the theorem (see
  Remark~\ref{rem:Pk-norm}).

  To show the second part, we notice that we have
  \begin{align*}
    \|u^{(x_0',t_0')}_k-p^{(x_0',t_0')}_\kappa\|_{L_2(S_r^+, G)}&\leq
    \|u^{(x_0',t_0')}_k-p^{(x_0,t_0)}_\kappa\|_{L_2(S_r^+, G)}\\
    &\qquad +
    \|p^{(x_0',t_0')}_\kappa-p^{(x_0,t_0)}_\kappa\|_{L_2(S_r^+,
      G)}\\&\leq 2(2\epsilon +C\,r_\epsilon^\sigma)^{1/2}r^{\kappa+1},
  \end{align*}
  for $\|(x_0'-x_0,t_0'-t_0)\|<\delta_\epsilon$, $0<r<r_\epsilon$, or
  equivalently
  \begin{equation}\label{eq:w-p-L2-nonzero}
    \|w^{x_0',t_0'}_r-p^{(x_0',t_0')}_\kappa\|_{L_2(S_1^+, G)}\leq 2(2\epsilon+C r_\epsilon^\sigma)^{1/2},
  \end{equation}
  where
$$
w^{(x_0',t_0')}_r(x,t):=\frac{u^{(x_0',t_0')}_k(rx,r^2t)}{r^\kappa}
$$
is the homogeneous rescaling of $u^{(x_0',t_0')}_k$.  Making a finite
cover of the compact $K$ with full parabolic cylinders $\tilde
Q_{\delta_\epsilon(x_0^i,t_0^i)}(x_0^i,t_0^i)$ for some
$(x_0^i,t_0^i)\in K$, $i=1,\ldots,N$, we see that
\eqref{eq:w-p-L2-nonzero} is satisfied for all $(x_0',t_0')\in K$,
$r<r_\epsilon^K:=\min\{r_\epsilon(x_0^i,t_0^i)\mid i=1,\ldots,N\}$ and
$C=C^K:=\max\{C(x_0^i,t_0^i)\mid i=1,\ldots, N\}$.

Now note that $w^{(x_0',t_0')}_r$ solves the parabolic Signorini
problem in $Q_1^+$ with zero thin obstacle and the right-hand side
$$
\big|g^{(x_0',t_0')}_r(x,t)\big|=\Big|\frac{f^{(x_0',t_0')}(rx,r^2t)}{r^{\kappa-2}}\Big|\leq
M r^{\ell-\kappa}\quad\text{in }Q_1^+.
$$
Besides, $p^{(x_0',t_0')}_\kappa$ also solves the parabolic Signorini
problem with zero thin obstacle, and zero right-hand side. This
implies
$$
(\Delta-\partial_t)\Big(w^{(x_0',t_0')}_r(x,t)-p^{(x_0',t_0')}_\kappa(x,t)\Big)_\pm\geq
-M r^{\ell-\kappa}\to 0\quad\text{in }Q_1,
$$
and therefore using the $L_\infty-L_2$ bounds as in the proof of
Theorem~\ref{thm:exist-homogen-blowups}(iii), we obtain that
$$
\|w^{(x_0',t_0')}_r-p^{(x_0',t_0')}_\kappa\|_{L_\infty(Q_{1/2})}\leq
C_\epsilon,
$$
for all $(x_0',t_0')\in K$, $r<r_\epsilon^K$ and $C_\epsilon\to 0$ as
$\epsilon\to 0$. It is now easy to see that this implies the second
part of the theorem.
\end{proof}

We are now ready to prove Theorems~\ref{thm:k-diff-sing-p} and
\ref{thm:sing-points-nonzero}.

\begin{proof}[Proof of Theorem~\ref{thm:k-diff-sing-p}] Simply combine
  Theorems~\ref{thm:uniq-blowup-sing-nonzero} and
  \ref{thm:cont-dep-blowup-nonzero}.
\end{proof}

\begin{proof}[Proof of Theorem~\ref{thm:sing-points-nonzero}]
  \mbox{}

  \smallskip
  \noindent\emph{Step 1: Parabolic Whitney's extension.}
  For any $(x_0,t_0)\in\Sigma_\kappa(v)$ let the polynomial
  $p^{(x_0,t_0)}_\kappa\in\P_\kappa$ be as in
  Theorem~\ref{thm:cont-dep-blowup-nonzero}. Write it in the expanded
  form
  $$
  p^{(x_0,t_0)}_\kappa(x)=\sum_{|\alpha|+2j=\kappa}
  \frac{a_{\alpha,j}(x_0,t_0)}{\alpha!j!}x^\alpha t^j.
  $$
  Then, the coefficients $a_{\alpha,j}(x,t)$ are continuous on
  $\Sigma_{\kappa}(v)$.  Next, for any multi-index
  $\alpha=(\alpha_1,\ldots,\alpha_n)$ and integer $j=0,\ldots,m$, let
  $$
  f_{\alpha,j}(x,t)=\begin{cases} 0, & |\alpha|+2j<\kappa\\
    a_{\alpha,j}(x,t), & |\alpha|+2j=\kappa,
  \end{cases}\qquad (x,t)\in\Sigma_\kappa(v).
$$
Then, we have the following compatibility lemma.

\begin{lemma}\label{lem:Whitney-compat-nonzero}
  Let $K=E_j$ for some $j\in\N$, as in
  Lemma~\ref{lem:sing-set-F-sigma-nonzero}. Then for any $(x_0,t_0),
  (x,t)\in K$
  \begin{multline}\label{eq:compat-1-nonzero}
    f_{\alpha,j}(x,t)=\sum_{|\beta|+2k\leq
      \kappa-|\alpha|-2j}\frac{f_{\alpha+\beta,j+k}(x_0,t_0)}{\beta!k!}(x-x_0)^\beta
    (t-t_0)^j\\ +R_{\alpha,j}(x,t;x_0,t_0)
  \end{multline}
  with
  \begin{equation}\label{eq:compat-2-nonzero}
    |R_{\alpha,j}(x,t;x_0,t_0)|\leq \sigma_{\alpha,j}(\|(x-x_0,t-t_0\|)\|(x-x_0,t-t_0)\|^{\kappa-|\alpha|-2j},
  \end{equation}
  where $\sigma_{\alpha,j}=\sigma_{\alpha,j}^K$ is a certain modulus
  of continuity.
\end{lemma}
\begin{proof}
  1) In the case $|\alpha|+2j=\kappa$ we have
$$
R_{\alpha,j}(x,t;x_0,t_0)=a_{\alpha,j}(x,t)-a_{\alpha,j}(x_0,t_0)
$$
and the statement follows from the continuity of $a_{\alpha,j}(x,t)$
on $\Sigma_\kappa(v)$.

\smallskip 2) For $0\leq |\alpha|+2j<\kappa$ we have

\begin{align*}
  R_\alpha(x,t,x_0,t_0)&=-\sum_{\substack{(\gamma,k)\geq
      (\alpha,j)\\|\gamma|+2k=\kappa}}
  \frac{a_{\gamma,k}(x_0,t_0)}{(\gamma-\alpha)!(k-j)!}(x-x_0)^{\gamma-\alpha}(t-t_0)^{k-j}\\
  &= - \partial^\alpha_x\partial_t^j p^{(x_0,t_0)}_\kappa
  (x-x_0,t-t_0).
\end{align*}
Suppose now that there exists no modulus of continuity
$\sigma_{\alpha,j}$ such that \eqref{eq:compat-2-nonzero} is satisfied
for all $(x_0,t_0),(x,t)\in K$. Then, there exists $\eta>0$ and a
sequence $(x_0^i,t_0^i), (x^i,t^i)\in K$, with
$$
\max\{|x^i-x_0^i|,|t^i-t_0^i|^{1/2}\}=:\delta_i\to 0,
$$
such that
\begin{multline}\label{eq:compat-contr-nonzero}
  \Big|\sum_{\substack{(\gamma,k)\geq(\alpha,j)\\|\gamma|+2k=\kappa}}
  \frac{a_{\gamma,k}(x_0^i)}{(\gamma-\alpha)!(k-j)!}(x^i-x_0^i)^{\gamma-\alpha}(t^i-t_0^i)^{k-j}\Big|\\\geq
  \eta \|(x^i-x_0^i,t^i-t_0^i)\|^{\kappa-|\alpha|-2j}.
\end{multline}
Consider the rescalings
$$
w^i(x,t)=\frac{u^{(x_0^i,t_0^i)}_k(\delta_i x,\delta_i
  t^2)}{\delta_i^\kappa},\quad
(\xi^i,\tau^i)=\Big(\frac{x^i-x_0^i}{\delta_i},\frac{t^i-t_0^i}{\delta_i^2}\Big).
$$
Without loss of generality we may assume that $(x_0^i,t_0^i)\to
(x_0,t_0)\in K$ and $(\xi^i,\tau^i)\to (\xi_0,\tau_0)\in \partial
\tilde Q_1$.  Then, by Theorem~\ref{thm:cont-dep-blowup-nonzero}, for
any $R>0$ and large $i$ we have for a modulus of continuity
$\sigma=\sigma^K$
$$
|w^i(x,t)-p^{(x_0^i,t_0^i)}_\kappa(x,t)|\leq
\sigma(\delta_i\|(x,t)\|)\|(x,t)\|^\kappa,\quad (x,t)\in S_R,
$$
and therefore
\begin{equation}\label{eq:wip}
  w^i(x,t)\to p^{(x_0,t_0)}_\kappa(x,t)\quad\text{in } L_\infty(Q_R).
\end{equation}
Note that we do not necessarily have the above convergence in the full
parabolic cylinder $\tilde Q_R$. Next, consider the rescalings at
$(x^i,t^i)$ instead of $(x_0^i,t_0^i)$
$$
\tilde w^i(x,t)=\frac{u^{(x^i,t^i)}_k(\delta_i x,\delta_i^2
  t)}{\delta_i^\kappa}.
$$
Then, by the same argument as above
\begin{equation}\label{eq:wip'}
  \tilde w^i(x,t)\to p^{(x_0,t_0)}_\kappa(x,t)\quad\text{in } L_\infty(Q_R).
\end{equation}
We then claim that the $H^{\ell,\ell/2}$ regularity of the thin
obstacle $\phi$ implies that
\begin{equation}\label{eq:wixi}
  w^i(x+\xi^i,t+\tau^i)-\tilde w^i(x,t)\to 0\quad\text{in } L_\infty(\tilde
  Q_R)
\end{equation}
for any $R>0$, or equivalently
\begin{equation}\label{eq:wixi'}
  w^i(x,t)-\tilde w^i(x-\xi^i,t-\tau^i)\to 0\quad\text{in } L_\infty(\tilde
  Q_R).
\end{equation}
Indeed, if $q^{(x_0,t_0)}_k(x',t)$ (as usual) denotes the $k$-th
parabolic Taylor polynomial of the thin obstacle $\phi(x')$ at
$(x_0,t_0)$, then\footnote{just note that the arguments inside $\phi$
  are the same}
\begin{align*}
  &\frac{q^{(x_0^i,t_0^i)}_k(\delta_i(x'+\xi^i),\delta_i^2(t+\tau^i))-q^{(x^i,t^i)}_k(\delta_i
    x',\delta_i^2t)}{\delta_i^\kappa}\\&=\frac{\phi(x_0^i+\delta_i(x'+\xi^i),t_0^i+\delta_i^2(t+\tau^i))+O(\delta_i^\ell\|(x'+\xi^i,t+\tau^i)\|^\ell)}{\delta_i^\kappa}\\
  &\qquad -\frac{\phi(x^i+\delta_ix',t^i+\delta_i^2t)+O(\delta_i^\ell\|(x',t)\|^\ell)}{\delta_i^\kappa}\\
  &=O(\delta_i^{\ell-\kappa})\to 0
\end{align*}
and this implies the convergence \eqref{eq:wixi}, if we write the
explicit definition of $w^i$ using the construction in
Proposition~\ref{prop:uk-def}.

To proceed further, we consider two different cases:

1) There are infinitely many indexes $i$ such that $\tau^i\geq 0$.

2) There are infinitely many indexes $i$ such that $\tau^i\leq 0$.

In both cases, passing to subsequences we may assume that $\tau^i\geq
0$ $(\leq 0)$ for all indexes $i$.

In case 1) we proceed as follows. If we take any $(x,t)\in Q_1$,
because of the nonpositivity of $\tau_i$ we have
$(x-\xi^i,t-\tau^i)\in Q_2$. Passing to the limit in \eqref{eq:wip},
\eqref{eq:wip'}, and \eqref{eq:wixi'}, we thus obtain
\begin{equation}\label{eq:p-translate}
  p^{(x_0,t_0)}_\kappa(x,t)=p^{(x_0,t_0)}_\kappa(x-\xi_0,t-\tau_0),\quad
  \text{for any }(x,t)\in Q_1.
\end{equation}
Because of the real analyticity of polynomials, it follows that
\eqref{eq:p-translate} holds in fact for all
$(x,t)\in\R^n\times\R$. But then, we also obtain
\begin{equation}\label{eq:p+translate}
  p^{(x_0,t_0)}_\kappa(x+\xi_0,t+\tau_0)=p^{(x_0,t_0)}_\kappa(x,t),\quad
  \text{for any }(x,t)\in\R^n\times\R.
\end{equation}
In particular, this implies that
$$
\partial_x^\alpha\partial_t^j
p^{(x_0,t_0)}_\kappa(\xi_0,\tau_0)=\partial_x^\alpha\partial_t^j
p^{(x_0,t_0)}_\kappa(0,0)=0,\quad |\alpha|+2j<\kappa.
$$
On the other hand, dividing both sides of
\eqref{eq:compat-contr-nonzero} by $\delta_i^{\kappa-|\alpha|-2j}$ and
passing to the limit, we obtain
$$
|\partial^\alpha_x\partial_t^j
p^{(x_0,t_0)}_\kappa(\xi_0,\tau_0)|=\Big|\sum_{\substack{(\gamma,k)\geq(\alpha,j)\\|\gamma|+2k=\kappa}}
\frac{a_{\gamma,k}(x_0)}{(\gamma-\alpha)!}\xi_0^{\gamma-\alpha}\tau_0^{k-j}\Big|\geq
\eta>0,
$$
a contradiction.

In case 2), when there are infinitely many indexes $i$ so that
$\tau^i\leq 0$, passing to the limit in \eqref{eq:wip},
\eqref{eq:wip'}, and \eqref{eq:wixi}, we obtain \eqref{eq:p+translate}
for $(x,t)\in Q_1$. Again by real analyticity, we have the same
conclusion for all $(x,t)\in\R^n\times\R$. Then, we complete the proof
arguing as in case 1).
\end{proof}

So in all cases, the compatibility conditions
\eqref{eq:compat-1-nonzero}--\eqref{eq:compat-2-nonzero} are satisfied
and we can apply the parabolic Whitney's extension theorem that we
have proved in Appendix~\ref{sec:parab-whitn-extens}, see
Theorem~\ref{thm:parab-Whitney}. Thus, there exists a function $F\in
C^{2m,m}(\R^n\times\R)$ such that
$$
\partial^\alpha_x\partial_t^j F=f_{\alpha,j}\quad\text{on } K,
$$
for any $|\alpha|+2j\leq \kappa$.

\medskip \emph{Step 2: Implicit function theorem.} Suppose now
$(x_0,t_0)\in \Sigma_\kappa^d(v)\cap K$.

We will consider two subcases: $d\leq n-2$ and $d=n-1$.

\smallskip 1) $d\in \{0,1,\ldots,n-2\}$. Then, there are multi-indexes
$(\beta'_i,k_i)$ with $|\beta_i'|+2k_i=\kappa-1$, for
$i=1,\ldots,n-1-d$ such that
$$
v_i=\nabla \partial_{x'}^{\beta'_i}\partial_t^{k_i}F(x_0,t_0)=\nabla \partial_{x'}^{\beta'_i}\partial_t^{k_i}p_\kappa^{(x_0,t_0)}
$$
are linearly independent.

On the other hand,
$$
\Sigma_\kappa^d(v)\cap K \subset \bigcap_{i=1}^{n-1-d}
\{\partial^{\beta_i'}_{x'}\partial_t^{k_i} F=0\}.
$$
Therefore, in view of the implicit function theorem, the linear
independence of $v_1,\ldots,v_{n-1-d}$ implies that
$\Sigma_\kappa^d(v)\cap K$ is contained in a $d$-dimensional
space-like $C^{1,0}$ manifold in a neighborhood of
$(x_0,t_0)$. Finally, since $\Sigma_k(u)=\bigcup_{j=1}^{\infty} E_j$
this implies the statement of the theorem in the case $d\in
\{0,1,\ldots, n-2\}$.

\smallskip 2) Suppose now $d=n-1$. In this case, by
Lemma~\ref{lem:timelike-blowup}, we have
$$
\partial_t^{m}F(x_0,t_0)=\partial_t^{m}p_\kappa^{(x_0,t_0)}\not=0.
$$
On the other hand,
$$
\Sigma_\kappa^d(v)\cap K \subset\{\partial_t^{m-1} F=0\}.
$$
Thus, by the implicit function theorem we obtain that
$\Sigma_\kappa^d(v)\cap K$ in a neighborhood of $(x_0,t_0)$ is
contained in a time-like $(n-1)$-dimensional $C^1$ manifold, as
required. This completes the proof of the theorem.
\end{proof}

\appendix
\section{Estimates in Gaussian spaces: Proofs}
\label{sec:est-gauss-proofs}

In this section we give the proofs of the estimates stated in
Section~\ref{sec:estimates-w2-1_2s_1+}.

\begin{proof}[Proof of Lemma~\ref{lem:w212}]
  \setcounter{step}{0} \step{1-W21-gauss} For the given $\rho>0$,
  choose $\rho<\tilde\rho<1$. Then, note that without loss of
  generality we may assume that $u(\cdot,-1)=0$, by multiplying $u$
  with a smooth cutoff function $\eta(t)$ such that $\eta=1$ on
  $[-\tilde\rho^2,0]$ and $\eta=0$ near $t=-1$.

  Next, fix a cutoff function $\hat \zeta_0\in C^\infty_0(\R^n)$ and
  let $R$ be so large that $\hat\zeta_0$ vanishes outside $B_{R-1}$
  and $u$ vanishes outside $B_{R-1}\times(-1,0]$. Then, approximate
  $u$ with the solutions of the penalized problem
  \begin{align*}
    \Delta u^\epsilon-\partial_t u^\epsilon=f^\epsilon&\quad\text{in }
    B_R^+\times(-1,0],\\
    \partial_{x_n}u^\epsilon=\beta_\epsilon(u^\epsilon)&\quad\text{on } B_R'\times(-1,0],\\
    u^\epsilon=0&\quad\text{on }(\partial B_R)^+\times(-1,0],\\
    u^\epsilon(\cdot, -1)=0&\quad\text{on }B_R^+,
  \end{align*}
  where $f^\epsilon$ is a mollification of $f$.

  Let now $r\in [\rho,\tilde\rho]$ be arbitrary. Then, for any small
  $\delta>0$ and $\eta\in W^{1,0}_2(B_R^+\times(-r^2,-\delta^2])$,
  vanishing on $(B_R^+\setminus B_{R-1}^+)\times(-r^2,-\delta^2]$, we
  will have
$$
\int_{S_r^+\setminus S_\delta^+} [\nabla u^\epsilon\nabla
\eta+u_t^\epsilon\eta+ f^\epsilon\eta] dxdt=-\int_{S_r'\setminus
  S_\delta'}\beta_\epsilon(u^\epsilon)\eta dx'dt.
$$
Now, for the cutoff function $\hat\zeta_0$ as above, define the family
of homogeneous functions in $S_1$ by letting
$$
\zeta_k(x,t)=|t|^{k/2}\hat\zeta_0(x/\sqrt{|t|}).
$$
Then, choosing $\eta=u^\epsilon\zeta_1^2 G$, we will have
\begin{align*}
  &\int_{S_r^+\setminus S_\delta^+}|\nabla u^\epsilon|^2\zeta_1^2
  G+(u^\epsilon\nabla u^\epsilon\nabla G\zeta_1^2+ u^\epsilon
  u^\epsilon_t\zeta_1^2 G)+2u^\epsilon\zeta_1\nabla u^\epsilon\nabla\zeta_1  G+f^\epsilon u^\epsilon\zeta_1^2 G\\
  &\qquad=-\int_{S_r'\setminus
    S_\delta'}\beta_\epsilon(u^\epsilon)u^\epsilon\zeta_1^2 G
  dx'dt\leq 0,
\end{align*}
where we have used that $s\beta_\epsilon(s)\geq 0$.  Next, recalling
that $\nabla G=\frac{x}{2t} G$ and that $Z((u^\epsilon)^2)=2
u^\epsilon(Zu^\epsilon)=2u^\epsilon(x\nabla u^\epsilon+2t\partial_t
u^\epsilon)$, we can rewrite the above inequality as
$$
\int_{S_r^+\setminus S_\delta^+}|\nabla u^\epsilon|^2\zeta_1^2
G+\frac{1}{4t}Z((u^\epsilon)^2)\zeta_1^2 G+2u^\epsilon\zeta_1 \nabla
u^\epsilon\nabla \zeta_1 G+f^\epsilon u^\epsilon\zeta_1^2 G\leq 0.
$$
We then can use the standard arguments in the proof of energy
inequalities, except that we need to handle the term involving
$Z((u^\epsilon)^2)$.  Making the change of variables $t=-\lambda^2$,
$x=\lambda y$ and using the identities
$$
G(\lambda x,-\lambda^2)=\lambda^n G(x,-1),\quad \zeta_1(\lambda
y,-\lambda^2)=\lambda \hat\zeta_0(y),$$ we obtain
\begin{align*}
  &\int_{S_r^+\setminus S_\delta^+}\frac{1}{2t}Z((u^\epsilon)^2)\zeta_1^2 G dxdt\\
  &\qquad =-\int_{\delta}^r\int_{\R^n} \lambda
  Z((u^\epsilon)^2)(\lambda y,-\lambda^2)\hat\zeta_0^2(y) G(y,-1)
  dyd\lambda\\
  &\qquad =
  -\int_\delta^r\int_{\R^n}\lambda^2\left[\frac{d}{d\lambda}(u^\epsilon(\lambda
    y,-\lambda^2)^2)\right]\hat\zeta_0^2(y) G(y,-1)dyd\lambda\\
  &\qquad= -\int_{\R^n_+} [r^2u^\epsilon(ry,-r^2)^2-\delta^2
  u^\epsilon(\delta
  y,-\delta^2)^2]\hat\zeta_0^2(y) G(y,-1)dy\\
  &\qquad\qquad+\int_\delta^r\int_{\R^n_+}
  2\lambda u^\epsilon(\lambda y,-\lambda^2)^2\hat\zeta_0^2 G(y,-1)dyd\lambda\\
  &\qquad \geq -r^2\int_{\R^n_+}
  u^\epsilon(ry,-r^2)^2\hat\zeta_0^2(y) G(y,-1)dy\\
  &\qquad=-r^2\int_{\R^n_+}
  u^\epsilon(\cdot,-r^2)\zeta_0^2G(\cdot,-r^2)dy,
\end{align*}
where we have used integration by parts is $\lambda$ variable.  Thus,
using Young's inequality, we conclude that
\begin{multline}
  \int_{S_r^+\setminus S_\delta^+}|\nabla u^\epsilon|^2\zeta_1^2 G\leq
  C_{n,\rho}\Big(\int_{\R^n_+} u^\epsilon(\cdot,-r^2)^2\zeta_0^2
  G(\cdot,-r^2)\\+\int_{S_r^+\setminus S_\delta^+}
  [(u^\epsilon)^2(|\nabla \zeta_1|^2+\zeta_0^2) +
  (f^\epsilon)^2\zeta_2^2] G\Big).
\end{multline}
Now, integrating over $r\in[\rho,\tilde \rho]$, we obtain
$$
\int_{S_\rho^+\setminus S_\delta^+}|\nabla u^\epsilon|^2\zeta_1^2
G\leq C_{n,\rho}\int_{S_{\tilde \rho}^+\setminus S_\delta^+}
[(u^\epsilon)^2(|\nabla \zeta_1|^2+\zeta_0^2) +
(f^\epsilon)^2\zeta_2^2] G.
$$
Letting first $\epsilon\to 0+$, then letting the support of the cutoff
function $\hat\zeta_0$ sweep $\R^n$, $\delta\to 0$, and replacing the
integral over $S_{\tilde\rho}^+$ by $S_1^+$, we conclude that
$$
\int_{S_{\rho}^+}|t||\nabla u|^2 G\leq C_{n,\rho} \int_{S_1^+}(u^2
+|t|^2f^2)G.
$$

\step{2-W21-gauss} Let $\hat\zeta_0$, $R$, homogeneous functions
$\zeta_k$, the approximations $u^\epsilon$, $r\in[\rho,\tilde \rho]$,
and $\delta$ be as above. Then, choosing as a test function
$\eta_{x_i}$, $i=1,2,\ldots, n-1$, with $\eta\in
W^{2,1}_2(S_r^+\setminus S_\delta^+)$ vanishing in $(B_R^+\setminus
B_{R-1}^+)\times(-r^2,-\delta^2]$ and integrating by parts, we obtain
$$
\int_{S_r^+\setminus S_\delta^+}[\nabla u_{x_i}^\epsilon\nabla
\eta+u^\epsilon_{x_it}
\eta+f^\epsilon\eta_{x_i}]dxdt=-\int_{S_r'\setminus
  S_\delta'}\beta_\epsilon'(u^\epsilon)u^\epsilon_{x_i}\eta dx'dt.
$$
Then, plugging $\eta=u_{x_i}^\epsilon\zeta^2_2 G$ in the above
identity, we will have
\begin{align*}
  &\int_{S_r^+\setminus S_\delta^+} |\nabla
  u_{x_i}^\epsilon|^2\zeta_2^2 G+(u_{x_i}^\epsilon\nabla
  u_{x_i}^\epsilon\nabla G\zeta_2^2+u_{x_it}^\epsilon u_{x_i}^\epsilon
  \zeta_2^2 G)+2u_{x_i}^\epsilon\zeta_2\nabla
  u_{x_i}^\epsilon\nabla \zeta_2 G\\
  &\qquad+\int_{S_r^+\setminus S_\delta^+} f^\epsilon
  u_{x_ix_i}\zeta_2^2 G+ 2f^\epsilon
  u_{x_i}^\epsilon\zeta_2(\zeta_2)_{x_i} G+f^\epsilon
  u_{x_i}^\epsilon\zeta_2^2 G_{x_i}\\
  &\qquad=-\int_{S_r'\setminus S_\delta'}
  \beta'_\epsilon(u^\epsilon)(u_{x_i}^\epsilon)^2\zeta_2^2 G \leq 0.
\end{align*}
Then, again using the identity $\nabla G=\frac{x}{2t} G$, we may
rewrite the above inequality as
\begin{align*}
  &\int_{S_r^+\setminus S_\delta^+} |\nabla
  u_{x_i}^\epsilon|^2\zeta_2^2 G+\frac{1}{4t}
  Z((u_{x_i}^\epsilon)^2)\zeta_2^2 G +2u_{x_i}^\epsilon\zeta_2\nabla
  u_{x_i}^\epsilon\nabla \zeta_2 G\\
  &\qquad\leq \int_{S_r^+\setminus S_\delta^+} |f^\epsilon|
  |u_{x_ix_i}^\epsilon|\zeta_2^2 G+ 2|f^\epsilon|
  |u_{x_i}^\epsilon||\zeta_2||\nabla \zeta_2| G+|f^\epsilon|
  |u_{x_i}^\epsilon|\zeta_2^2\frac{|x|}{2|t|} G.
\end{align*}
Arguing as in step \stepref{1-W21-gauss} above, we have
\begin{align*}
  \int_{S_r^+\setminus
    S_\delta^+}\frac{1}{2t}Z((u_{x_i}^\epsilon)^2)\zeta_2^2
  G&=-\int_{\R^n_+}
  [r^4u_{x_i}^\epsilon(ry,-r^2)^2-\delta^4u_{x_i}^\epsilon(\delta
  y,-\delta^2)^2]\hat\zeta_0(y)^2 G(y,-1)
  \\
  &\qquad+4\int_\delta^r\int_{\R^n_+}\lambda^3(u_{x_i}^\epsilon)^2\hat\zeta_0^2
  G\\&\geq -r^4\int_{\R^n_+}
  u_{x_i}^\epsilon(ry,-r^2)^2\hat\zeta_0(y)^2 G(y,-1)dy \\&=
  -r^2\int_{\R^n_+} u_{x_i}^\epsilon(\cdot,-r^2)^2\zeta_1^2
  G(\cdot,-r^2).
\end{align*}
We further estimate, by the appropriate Young inequalities,
\begin{align*}
  \int_{S_r^+\setminus
    S_\delta^+}|f^\epsilon||u_{x_i}^\epsilon|\zeta_2^2\frac{|x|}{2|t|}
  G&\leq C_{n,\rho}\int_{S_r^+\setminus S_\delta^+}
  (f^\epsilon)^2\zeta_2^2 G+c_{n,\rho}\int_{S_r^+\setminus
    S_\delta^+}|\nabla
  u^\epsilon|^2\zeta_1^2\frac{|x|^2}{t} G\\
  &\leq C_{n,\rho}\int_{S_r^+\setminus S_\delta^+}
  (f^\epsilon)^2\zeta_2^2 G \\
  &\qquad+c_{n,\rho}\int_{S_r^+\setminus S_\delta^+}[ |\nabla
  u^\epsilon|^2(\zeta_1^2+|\nabla \zeta_2|^2)+|D^2 u|^2\zeta_2^2]G,
\end{align*}
with a small constant $c_{n,\rho}>0$, where in the last step we have
used that the following claim.

\begin{claim}\label{clm:v2x2t} For any $v\in W^1_2(\R^n, G)$ and $t<0$
  we have
$$
\int_{\R^n} v^2\frac{|x|^2}{|t|} G(x,t)\leq C_n\int_{\R^n}
(v^2+|t||\nabla v|^2) G.
$$
\end{claim}
\begin{proof} Using that $\nabla G=\frac{x}{2t} G$, and then
  integrating by parts, we obtain
  \begin{align*}
    \int_{\R^n} v^2\frac{|x|^2}{2|t|} G&=-\int_{\R^n}v^2x\cdot \nabla
    G=\int_{\R^n}\div(x v^2) G \\
    &=n\int_{\R^n} v^2 G+\int 2v (x\nabla v)  G\\
    &\leq n\int_{\R^n} v^2 G+\int_{\R^n}v^2\frac{|x|^2}{4|t|}
    G+\int_{\R^n}4|t||\nabla v|^2 G,
  \end{align*}
  which implies the desired estimate.
\end{proof}
Combining the estimates above, we obtain
\begin{align*}
  &\int_{S_r^+\setminus S_\delta^+} |\nabla u_{x_i}^\epsilon|^2\zeta_2^2 G\\
  &\qquad\leq C_{n,\rho}\Big(\int_{\R^n_+} u_{x_i}^\epsilon(\cdot,
  -r^2)^2\zeta_1^2 G(\cdot,-r^2)+\int_{S_r^+\setminus S_\delta^+}
  [|\nabla u^\epsilon|^2(\zeta_1^2+|\nabla \zeta_2|^2)+
  (f^\epsilon)^2\zeta_2^2] G\Big)\\
  &\qquad\qquad+c_{n,\rho}\int_{S_r^+\setminus S_\delta^+} |D^2
  u^\epsilon|^2\zeta_2^2 G.
\end{align*}

\step{3-W21-gauss} Using the notations of the previous step, taking a
test function $\eta_{x_n}$ and integrating by parts, we will obtain
$$
\int_{S_r^+\setminus S_\delta^+}[\nabla u_{x_n}^\epsilon\nabla
\eta+u^{\epsilon}_{x_nt}\eta+f^\epsilon \eta_{x_n}]
dxdt=-\int_{S_r'\setminus S_\delta'} [u_t^\epsilon
\eta+\nabla'u^\epsilon\nabla'\eta]dx'dt.
$$
Plugging $\eta=u_{x_n}^\epsilon\zeta_2^2 G$, we will have
\begin{align*}
  &\int_{S_r^+\setminus S_\delta^+}|\nabla
  u_{x_n}^\epsilon|^2\zeta_2^2 G+(u_{x_n}^\epsilon\nabla
  u_{x_n}^\epsilon\nabla G\zeta_2^2+u_{x_nt}^\epsilon u_{x_n}^\epsilon
  \zeta_2^2 G)+2u_{x_n}^\epsilon\zeta_2\nabla
  u_{x_n}^\epsilon\nabla \zeta_2 G\\
  &\qquad+\int_{S_r^+\setminus S_\delta^+} f^\epsilon
  u_{x_nx_n}\zeta_2^2 G+ 2f^\epsilon
  u_{x_n}^\epsilon\zeta_2(\zeta_2)_{x_n} G+f^\epsilon
  u_{x_n}^\epsilon\zeta_2^2 G_{x_n}\\
  &\qquad=-\int_{S_r'\setminus S_\delta'} [\nabla'u^\epsilon\nabla'
  G\beta_\epsilon(u^\epsilon)\zeta_2^2+u_t^\epsilon
  \beta_\epsilon(u^\epsilon)\zeta_2^2 G]+|\nabla'
  u^\epsilon_{x_i}|^2\beta'_\epsilon(u^\epsilon)\zeta_2^2 G\\
  &\qquad\qquad-\int_{S_r'\setminus
    S_\delta'}2u^\epsilon_{x_n}\nabla'u^\epsilon\nabla'\zeta_2\zeta_2
  G.
\end{align*}
We therefore have
\begin{align*}
  &\int_{S_r^+\setminus S_\delta^+}|\nabla
  u_{x_n}^\epsilon|^2\zeta_2^2 G+\frac{1}{4t} Z
  ((u_{x_n}^\epsilon)^2)\zeta_2^2 G +2u_{x_n}^\epsilon\zeta_2\nabla
  u_{x_n}^\epsilon\nabla \zeta_2 G\\
  &\qquad+\int_{S_r^+\setminus S_\delta^+} f^\epsilon
  u_{x_nx_n}\zeta_2^2 G+ 2f^\epsilon
  u_{x_n}^\epsilon\zeta_2(\zeta_2)_{x_n} G+f^\epsilon
  u_{x_n}^\epsilon\zeta_2^2 G_{x_n}\\
  &\qquad\leq-\int_{S_r'\setminus S_\delta'} \frac{1}{4t}
  Z(\mathcal{B}_\epsilon(u^\epsilon))\zeta_2^2
  G+2u^\epsilon_{x_n}\nabla'u^\epsilon\nabla'\zeta_2\zeta_2 G
  =J_1+J_2.
\end{align*}
To estimate $J_1$ we argue as before, however we now take into account
that the spatial dimension is less by one:
\begin{align*}
  J_1&=-\int_{S_r'\setminus S_\delta^+} \frac{1}{4t}
  Z(\mathcal{B}_\epsilon(u^\epsilon))\zeta_2^2 G\\
  &
  =\frac12\int_{\R^{n-1}}[r^3\mathcal{B}_\epsilon(u^\epsilon(ry',-r^2))-\delta^3\mathcal{B_\epsilon}(u^\epsilon(\delta
  y',-\delta^2))]\hat\zeta_0(y')^2 G(y',-1)dy'\\
  &\qquad -\frac32\int_\delta^r\int_{\R^{n-1}}\lambda^2\mathcal{B}_\epsilon(u^\epsilon(\lambda y',-\lambda^2))\hat\zeta_0(y')^2G(y'-1)dy'd\lambda\\
  &\leq C_{n,\rho}\epsilon,
\end{align*}
since $\mathcal{B}_\epsilon(s)\geq0$ for any $s\in\R$ and we have used
that $\mathcal{B}_\epsilon(u^\epsilon)\leq C_{n,\rho}\epsilon$ on
$B_{R-1}'\times[\rho,\tilde \rho]$. For more details see the proof of
Proposition~\ref{prop:diff-form-signor},
step~\stepref{3.3.2-diff-form}.

In order to estimate $J_2$ we write it as a solid integral
\begin{align*}
  J_2&=\int_{S_r^+\setminus S_\delta^+}\partial_{x_n}
  (u_{x_n}^\epsilon
  u_{x_i}(\zeta_2^2)_{x_i} G)\\
  &=\int_{S_r^+\setminus S_\delta^+}u^\epsilon_{x_nx_n}
  u_{x_i}^\epsilon(\zeta_2^2)_{x_i} G+ u_{x_n}^\epsilon
  u_{x_ix_n}^\epsilon (\zeta_2^2)_{x_i} G+u_{x_n}^\epsilon
  u_{x_i}^\epsilon(\zeta_2^2)_{x_ix_n} G\\
  &\qquad+\int_{S_r^+\setminus S_\delta^+}u_{x_n}^\epsilon
  u_{x_i}^\epsilon(\zeta_2^2)_{x_i} G_{x_n}=J_{21}+J_{22}.
\end{align*}
The terms in $J_{21}$ are estimated in a standard way by the
appropriate Young inequalities. The integral $J_{22}$ is estimated by
using Claim~\ref{clm:v2x2t}:
\begin{align*}
  |J_{22}|&\leq \int_{S_r^+\setminus S_\delta^+} |\nabla u^\epsilon|^2
  \zeta_2 |\nabla
  \zeta_2|\frac{x_n}{t} G dxdt\\
  &\leq C_{n,\rho}\int_{S_r^+\setminus S_\delta^+}|\nabla
  u^\epsilon|^2|\nabla \zeta_2|^2 G+c_{n,\rho}\int_{S_r^+\setminus
    S_\delta^+}|\nabla
  u^\epsilon|^2\zeta_1^2\frac{|x|^2}{|t|} G\\
  &\leq C_{n,\rho}\int_{S_r^+\setminus S_\delta^+}|\nabla
  u^\epsilon|^2|\nabla \zeta_2|^2 G+c_{n,\rho}\int_{S_r^+\setminus
    S_\delta^+}[ |\nabla u^\epsilon|^2(\zeta_1^2+|\nabla \zeta_2|^2)
  G+|D^2 u|^2\zeta_2^2 G],
\end{align*}
for a small constant $c_{n,\rho}>0$.

We further treat the term
$$
\int_{S_r^+\setminus S_\delta^+}\frac1{4t} Z
((u_{x_n}^\epsilon)^2)\zeta_2^2 G
$$
analogously to the similar term with $u_{x_i}^\epsilon$ in
\stepref{2-W21-gauss}. Collecting all estimates in this step, combined
with appropriate Young inequalities, we obtain
\begin{align*}
  \int_{S_r^+\setminus S_\delta^+} |\nabla
  u_{x_n}^\epsilon|^2\zeta_2^2 G&\leq C_{n,\rho}\int_{\R^n_+}
  u_{x_n}^\epsilon(\cdot,
  -r^2)^2\zeta_1^2 G(\cdot,-r^2)\\
  &\qquad+C_{n,\rho}\int_{S_r^+\setminus S_\delta^+} [|\nabla
  u^\epsilon|^2(\zeta_1^2+|\nabla \zeta_2|^2+|D^2\zeta_3|^2)+
  (f^\epsilon)^2\zeta_2^2] G\\
  &\qquad+c_{n,\rho}\int_{S_r^+\setminus S_\delta^+} |D^2
  u^\epsilon|^2\zeta_2^2 G + C_{n,\rho}\epsilon.
\end{align*}

\step{4-W21-gauss} Now combining the estimates in
\stepref{2-W21-gauss} and \stepref{3-W21-gauss} above and integrating
over $r\in[\rho,\tilde \rho]$ we obtain
\begin{align*}
  \int_{S_\rho^+\setminus S_\delta^+} |D^2 u^\epsilon|^2\zeta_2^2 G&
  \leq C_{n,\rho}\int_{S_{\tilde \rho}^+\setminus S_\delta^+}[|\nabla
  u^\epsilon|^2(\zeta_1^2+|\nabla
  \zeta_2|^2+|D^2\zeta_3|^2) +(f^\epsilon)^2\zeta_2^2] G\\
  &\qquad+C_{n,\rho}\epsilon.
\end{align*}
As before, passing to the limit as $\epsilon\to0$, increasing the
support of $\hat\zeta_0$, and then letting $\delta\to 0$, we conclude
that
$$
\int_{S_{\rho}^+} |t|^2|D^2u|^2 G\leq C_{n,\rho}\int_{S_1^+}
(|t||\nabla u|^2 +|t|^2 f^2) G.
$$
Finally noticing that $u_t=\Delta u-f$, we obtain the desired integral
estimate for $u_t$ as well. The proof is complete.
\end{proof}

\begin{proof}[Proof of Lemma~\ref{lem:w112-diff}] The proof is very
  similar to part \stepref{1-W21-gauss} of the proof of
  Lemma~\ref{lem:w212}. Indeed, for approximations $u_i^\epsilon$,
  $i=1,2$, we have the integral identities
$$
\int_{S_r^+\setminus S_\delta^+} [\nabla u_i^\epsilon\nabla
\eta+\partial_tu_i^\epsilon\eta+ f_i^\epsilon\eta]
dxdt=-\int_{S_r'\setminus S_\delta'}\beta_\epsilon(u_i^\epsilon)\eta
dx'dt.
$$
Taking the difference, choosing
$\eta=(u_1^\epsilon-u_2^\epsilon)\zeta_1^2 G$, and using the
inequality
$$
[\beta_\epsilon(u_1^\epsilon)-\beta_\epsilon(u_2^\epsilon)](u_1^\epsilon-u_2^\epsilon)\geq
0,
$$
we complete the proof as in step \stepref{1-W21-gauss} of the proof of
Lemma~\ref{lem:w212}.
\end{proof}

\section{Parabolic Whitney's extension theorem}
\label{sec:parab-whitn-extens}

Let $E$ be a compact subset of $\R^n\times\R$ and $f:E\to \R$ a
certain continuous function. Here we want to establish a theorem of
Whitney type (see \cite{Whi}) that will allow the extension of the
function $f$ to a function of class $C^{2m,m}(\R^n\times\R)$,
$m\in\N$. In fact, for that we need to have a family of functions
$\{f_{\alpha,j}\}_{|\alpha|+2j\leq m}$, where
$\alpha=(\alpha_1,\ldots,\alpha_n)$ and $j$ is a nonnegative integer.

\begin{theorem}[Parabolic Whitney's extension]\label{thm:parab-Whitney} Let $\{f_{\alpha,j}\}_{|\alpha|+2j\leq m}$ be a family
  of functions on $E$, with $f_{0,0}\equiv f$, satisfying the
  following compatibility conditions: there exists a family of moduli
  of continuity $\{\omega_{\alpha,j}\}_{|\alpha|+2j\leq 2m}$, such
  that
$$
f_{\alpha,j}(x,t)=\sum_{|\beta|+2k\leq 2m-|\alpha|-2j}
\frac{f_{\alpha+\beta,j+k}(x_0,t_0)}{\beta!k!}(x-x_0)^\beta(t-t_0)^k+R_{\alpha,j}(x,t;
x_0,t_0)
$$
and
$$
|R_{\alpha,j}(x,t; x_0,t_0)|\leq
\omega_{\alpha,j}(\|(x-x_0,t-t_0)\|)\|(x-x_0,t-t_0)\|^{2m-|\alpha|-2j}.
$$
Then, there exists a function $F\in C^{2m,m}(\R^n\times\R)$ such that
$F=f$ on $E$ and moreover $\partial_x^\alpha\partial_t^j
F=f_{\alpha,j}$ on $E$, for $|\alpha|+2j\leq 2m$.
\end{theorem}
The construction of the extension is done in the following four steps.

\medskip \emph{Step 1:} Parabolic Whitney cube decomposition of
$E^c=(\R^{n}\times\R)\setminus E$. We say that $Q$ is a parabolic
($k$-)dyadic cube if it has a form
$$Q=[a_1 2^{-k}, (a_1+1)2^{-k}]\times\cdots\times [a_n 2^{-k},
(a_n+1)2^{-k}]\times[b2^{-2k},(b+1)2^{-2k}],$$ where $a_i, b\in\Z$. We
will call $2^{-k}$ the size of $Q$ and denote it by $\ell(Q)$. We will
also call $((a_1+\frac12)2^{-k},\ldots, (a_n+\frac12)2^{-k},
(b+\frac12)2^{-2k}]$ the center of $Q$.

\smallskip The proof of the following lemma is very similar to its
Euclidean counterpart and is therefore omitted. (A slightly different
version of this lemma can be found in \cite{FS}*{Lemma~1.67}, for more
general homogeneous spaces.)

\begin{lemma}[Parabolic Whitney cube decomposition] For any closed set
  $E$ there exists a family $\mathcal{W}=\{Q_i\}$ of parabolic dyadic
  cubes with the following properties:
  \begin{enumerate}[label=\textup{(\roman*)}]
  \item $\bigcup_i Q_i=\Omega=E^c$,
  \item $Q_i^\circ\cap Q_j^\circ=\emptyset$ for $i\not=j$,
  \item $c_n\ell(Q_i)\leq \dist_p(Q_i,E)\leq C_n \ell(Q_i)$ for some
    positive constants $c_n, C_n$ depending only on the dimension
    $n$.\qed
  \end{enumerate}
\end{lemma}

For every $Q_i$ let $(x_i,t_i)$ be the center and $\ell_i$ the size of
the parabolic cube $Q_i$. Then let
$Q_i^*=\delta_{1+\epsilon}[Q_i-(x_i,t_i)]+(x_i,t_i)$, where
$\delta_\lambda: (x,t)\mapsto (\lambda x, \lambda^2t)$ is the
parabolic dilation. Clearly, the family of $\{Q_i^*\}$ is no longer
disjoint, however, we every point in $E^c$ has a small neighborhood
that intersects at most $N=N_n$ cubes $Q_i^*$, provided
$0<\epsilon<\epsilon_n$ is small. Then define
$$
\phi_i(x,t)=\phi\left(\frac{x-x_i}{\ell_i},\frac{t-t_i}{\ell_i^2}\right),
$$
where $\phi\in C^\infty(\R^n\times\R)$ such that
$$
\phi\geq 0,\quad \phi>0\text{ on } I_1,\quad \supp \phi\subset
I_{1+\epsilon},
$$
where
$$
I_\lambda=[-\lambda/2,\lambda/2]\times\cdots\times[-\lambda/2,/2]\times[-\lambda^2/2,\lambda^2/2].
$$
We also observe that
\begin{equation*}
  |\partial_x^\alpha\partial_t^j \phi_i(x,t)|\leq A_{\alpha,j} \ell_i^{-|\alpha|-2j},
\end{equation*}
for some constants $A_{\alpha,j}$.  Next, we define a partition of
unity $\{\phi_i^*\}$ subordinate to $\{Q_i^*\}$ as follows. Let
$$
\phi_i^*(x,t)=\frac{\phi_i(x,t)}{\Phi(x,t )},\quad
\Phi(x,t)=\sum_{k}\phi_k(x,t), \quad (x,t)\in E^c.
$$
Note that in the definition of $\Phi$, the sum is locally finite and
therefore $C^\infty$ in $E^c$, and also satisfies $1\leq \Phi\leq N$.
Then $\phi_i^*$ are also $C^\infty$ in $E^c$ and we have
$$
\sum_{i}\phi_i^*=1\quad\text{in }E^c,
$$
where again the sum is locally finite. Moreover, it is easy to see
that, similarly to $\phi_i$, we have the estimates
\begin{equation}\label{eq:WE-dphi*-est}
  |\partial_x^\alpha\partial_t^j \phi_i^*(x,t)|\leq A_{\alpha,j}^* \ell_i^{-|\alpha|-2j}.
\end{equation}
\medskip

\emph{Step 2:} For every $(x_0,t_0)\in E$ let
$$
P(x,t; x_0,t_0)=\sum_{|\alpha|+2j\leq 2m}
\frac{f_{\alpha,j}(x_0,t_0)}{\alpha!j!}(x-x_0)^\alpha (t-t_0)^j.
$$
In addition to $P$, it is convenient to introduce
$$
P_{\alpha,j}(x,t; x_0,t_0)=\sum_{|\beta|+2k\leq 2m-|\alpha|-2j}
\frac{f_{\alpha+\beta,j+k}(x_0,t_0)}{\beta!k!}(x-x_0)^\beta(t-t_0)^k
$$
for $|\alpha|+2j\leq 2m$. Note that in fact
$P_{\alpha,j}(x,t;x_0,t_0)=\partial_x^\alpha\partial_t^j
P(x,t;x_0,t_0)$.  Then by definition
$$
f_{\alpha,j}(x,t)=P_{\alpha,j}(x,t; x_0,t_0)+R_{\alpha,j}(x,t;
x_0,t_0),
$$
for any $(x,t), (x_0,t_0)\in E$.

\begin{lemma}\label{lem:WE-diff} For any $(x_0,t_0), (x_1,t_1)\in E$ and
  $(x,t)\in\R^n\times\R$, we have
$$
P(x,t;x_1,t_1)-P(x,t;x_0,t_0)=\sum_{|\beta|+k\leq 2m}
R_{\beta,k}(x_1,t_1;x_0,t_0)
\frac{(x-x_1)^\beta}{\beta!}\frac{(t-t_1)^k}{k!}
$$
and more generally
\begin{multline*}
  P_{\alpha,j}(x,t;x_1,t_1)-P_{\alpha,j}(x,t;x_0,t_0)=\\\sum_{|\beta|+k\leq
    2m-|\alpha|-2j} R_{\alpha+\beta,j+k}(x_1,t_1;x_0,t_0)
  \frac{(x-x_1)^\beta}{\beta!}\frac{(t-t_1)^k}{k!}.
\end{multline*}
\end{lemma}
\begin{proof} We will prove the latter formula. It is enough to check
  that the partial derivatives $\partial_x^\beta\partial_t^k$ of both
  sides equal to each other for $|\beta|+2k\leq 2m-|\alpha|-2j$, as
  both sides are polynomials of parabolic degree $ 2m-|\alpha|-2j$. We
  have
  \begin{align*}
    \partial_x^\beta\partial_t^k P_{\alpha,j}(x,t;
    x_1,t_1)\big|_{(x,t)=(x_1,t_1)}&=f_{\alpha+\beta,j+k}(x_1,t_1)\\
    \partial_x^\beta\partial_t^k P_{\alpha,j}(x,t;
    x_0,t_0)\big|_{(x,t)=(x_1,t_1)}&=P_{\alpha+\beta,j+k}(x_1,t_1;x_0,t_0),
  \end{align*}
  which implies the desired equality.
\end{proof}

\emph{Step 3:} We are now ready to define the extension function $F$.
For every $Q_i$ let $(y_i,s_i)\in E$ be such that
$\dist_p(Q_i,E)=\dist_p(Q_i,(y_i,s_i))$. Note that $(y_i,s_i)$ is not
necessarily unique. Then define
$$
F(x,t)=
\begin{cases}
  f(x,t)=f_{0,0}(x,t), & (x,t)\in E\\
  \displaystyle{\sum_{i} P(x,t;y_i,s_i)\phi^*_i(x,t)}, &(x,t)\in E^c.
\end{cases}
$$
From the local finiteness of the partition of unity, it is clear that
$F$ is $C^\infty$ in $E^c$. Then we can define
$$
F_{\alpha,j}(x,t)=
\begin{cases}
  f_{\alpha,j}(x,t), & (x,t)\in E\\
  \partial^\alpha_x\partial^j_t F(x,t), &(x,t)\in E^c,
\end{cases}
$$
for $|\alpha|+2j\leq 2m$.

\begin{lemma}\label{lem:WE-F-P} There exist moduli of continuity $\tilde\omega=\tilde\omega_{0,0}$ and
  $\tilde\omega_{\alpha,j}$, $|\alpha|+2j\leq 2m$, such that for
  $(x,t)\in\R^n\times\R$ and $(x_0,t_0)\in E$ we have
  \begin{align*}
    |F(x,t)-P(x,t;x_0,t_0)|&\leq \tilde\omega
    (\|(x-x_0,t-t_0)\|)\|(x-x_0,t-t_0)\|^{2m},\\
    \intertext{and more generally}
    |F_{\alpha,j}(x,t)-P_{\alpha,j}(x,t;x_0,t_0)|&\leq
    \tilde\omega_{\alpha,j}(\|(x-x_0,t-t_0)\|)\|(x-x_0,t-t_0)\|^{2m-|\alpha|-2j},
  \end{align*}
  for $|\alpha|+2j\leq 2m$.
\end{lemma}
\begin{proof} Note that for $(x,t)\in E$, the estimates follows from
  the compatibility assumptions. For $(x,t)\in E^c$, we have
  \begin{align*}
    &|F(x,t)-P(x,t;x_0,t_0)|=\big|\sum_i [P(x,t;y_i,s_i)-P(x,t;x_0,t_0)]\phi_i^*(x,t)\big|\\
    &\qquad\leq\sum_{i} \sum_{|\alpha|+2j\leq
      2m}|R_{\alpha,j}(y_i,s_i;x_0,t_0)
    |\frac{|x-y_i|^\alpha}{\alpha!}\frac{|t-s_i|^j}{j!}\phi_i^*(x,t)\\
    &\qquad\leq \sum_{i, |\alpha|+2j\leq
      2m}\omega_{\alpha,j}(C_n\|(x-y_i,t_0-s_i)\|)\|(x-y_i,t-s_i)\|^{2m-|\alpha|-2j+|\alpha|+2j}\phi_i^*(x,t)\\
    &\qquad\leq \tilde\omega(\|(x-x_0,t-t_0)\|)\|(x-x_0,t-t_0)\|^{2m},
  \end{align*}
  using that
$$
\|(x_0-y_i,t_0-s_i)\|\leq C_n\|(x-y_i,t-s_i)\|\leq
C_n^2\|(x-x_0,t-t_0)\|
$$
for $(x,t)\in Q_i^*$.

The second estimate in the lemma is obtained in a similar way.
Indeed, we can write
\begin{align*}
  &|F_{\alpha,j}(x,t)-P_{\alpha,j}(x,t)|=|\partial_x^\alpha\partial_t^j
  [F(x,t)-P(x,t;x_0,t_0]|\\
  &\qquad=|\partial_x^\alpha\partial_t^j\sum_i
  [P(x,t;y_i,s_i)-P(x,t;x_0,t_0)]\phi_i^*(x,t)|\\
  &\qquad=\sum_{i,\beta\leq \alpha, k\leq j}
  C^{\alpha,j}_{\beta,k}[P_{\beta,k}(x,t;y_i,s_i)-P_{\beta,k}(x,t;x_0,t_0)]\partial_x^{\alpha-\beta}\partial_t^{j-k}\phi_i^*(x,t)
\end{align*}
and then we argue as above by using Lemma~\ref{lem:WE-diff} and the
estimates \eqref{eq:WE-dphi*-est}.
\end{proof}

\begin{proof}[Proof of Theorem~\ref{thm:parab-Whitney}] Note that
  Lemma~\ref{lem:WE-F-P} implies that
  \begin{alignat*}{2}
    \partial_x^\beta
    F_{\alpha,j}(x,t)&=F_{\alpha+\beta,j}(x,t),&\quad&\text{for
      $|\beta|=1$,
      if $\alpha|+2j\leq 2m-1$}\\
    \partial_t F_{\alpha,j}(x,t)&=F_{\alpha,j+1}(x,t),&\quad&\text{if
      $\alpha|+2j\leq 2m-2$}
  \end{alignat*}
  at every $(x,t)\in E$. The same equalities hold also in $E^c$, by
  the definition of $F_{\alpha,j}$.  Thus, arguing by induction in the
  order of the derivative $|\alpha|+2j\leq 2m$ and using
  Lemma~\ref{lem:WE-F-P}, we prove that everywhere in $\R^n\times\R$
$$
\partial_x^\alpha\partial_t^j F=F_{\alpha,j},\quad |\alpha|+2j\leq 2m.
$$
We also note that $F_{\alpha,j}$ are continuous by
Lemma~\ref{lem:WE-F-P}.  The proof is complete.
\end{proof}


\end{document}